\documentclass{amsart}
\usepackage{graphicx}
\usepackage{amssymb,amscd,amsthm,amsxtra}
\usepackage{latexsym}
\usepackage{epsfig}
\usepackage{color}

\newtheorem{thm}{Theorem}[section]
\newtheorem{cor}[thm]{Corollary}
\newtheorem{lem}[thm]{Lemma}
\newtheorem{prop}[thm]{Proposition}
\theoremstyle{definition}
\newtheorem{defn}[thm]{Definition}
\theoremstyle{remark}
\newtheorem{rem}[thm]{\bf Remark}
\numberwithin{equation}{section}

\newcommand{\R}{\mathbb R}
\newcommand{\be}{\begin{equation}}
\newcommand{\ee}{\end{equation}}

\newcommand{\ep}{\eps}

\newcommand{\eps}{\varepsilon}

\newcommand{\comment}[1]{}

\begin{document}

\title[Regularity of Lipschitz free boundaries for the $p(x)$-Laplacian]{Regularity of Lipschitz free boundaries for a $p(x)$-Laplacian problem with right hand side}

\author[Fausto Ferrari]{Fausto Ferrari}
\address{Dipartimento di Matematica dell'Universit\`a di Bologna, Piazza di Porta S. Donato, 5, 40126 Bologna, Italy.}
\email{\tt fausto.ferrari@unibo.it}
\author[Claudia Lederman]{Claudia Lederman}
\address{IMAS - CONICET and Departamento  de
Ma\-te\-m\'a\-ti\-ca, Facultad de Ciencias Exactas y Naturales,
Universidad de Buenos Aires, (1428) Buenos Aires, Argentina.}
\email{\tt  clederma@dm.uba.ar}
\thanks{F. F. was partially supported by INDAM-GNAMPA 2019 project: {\it Proprietà di regolarit\`a delle soluzioni viscose con applicazioni a problemi di frontiera libera} and INDAM-GNAMPA 2022 project: {\it Regolarit\`a locale e globale per
problemi completamente
non lineari. } }
\thanks{C. L. was partially supported by the project GHAIA Horizon 2020 MCSA RISE  2017 programme grant 777822 and  by the grants CONICET PIP 11220150100032CO 2016-2019,  UBACYT 20020150100154BA and ANPCyT PICT 2019-00985. C. L. wishes to thank the Department of Mathematics of the University of Bologna, Italy, for the kind hospitality.}

\begin{abstract}
We continue our study in \cite{FL} on viscosity solutions to a one-phase free boundary problem for the $p(x)$-Laplacian with non-zero right hand side. We first prove that viscosity solutions are locally Lipschitz continuous, which is the optimal regularity for the problem. Then we prove that Lipschitz free boundaries of viscosity solutions are $C^{1,\alpha}$. We also present some applications of our results.

Moreover, we obtain new results for the operator under consideration that are of independent interest, such as a Harnack inequality.
\end{abstract}

\maketitle

\section{Introduction and main results}\label{section1}

In this paper we continue our  study in \cite{FL} on a one-phase  free boundary problem governed by the $p(x)$-Laplacian  with non-zero right hand side. More precisely, we denote  by
$$
\Delta_{p(x)}u:=\mbox{div} (|\nabla u|^{p(x)-2}\nabla u),
$$
where $p$ is a  function such that $1<p(x)<+\infty$. Then our problem is the following:
\begin{equation}  \label{fb}
\left\{
\begin{array}{ll}
\Delta_{p(x)} u = f, & \hbox{in $\Omega^+(u):= \{x \in \Omega : u(x)>0\}$}, \\
\  &  \\
|\nabla u|= g, & \hbox{on $F(u):= \partial \Omega^+(u) \cap
\Omega.$} 
\end{array}
\right.
\end{equation}
Here $\Omega  \subset \mathbb{R}^n$ is a bounded domain,  $p\in C^1(\Omega)$ is a Lipschitz continuous function,  
$f\in C(\Omega)\cap L^{\infty}(\Omega)$ and $g\in C^{0, \beta}(\Omega)\cap L^{\infty}(\Omega),$ $g> 0.$ 

This problem comes out  naturally from limits of a singular perturbation problem with
forcing term as in \cite{LW1}, where  solutions to  
\eqref{fb}, arising in the study of flame propagation with nonlocal and electromagnetic 
effects, are analyzed. On the other hand, \eqref{fb} appears by minimizing the following functional
\begin{equation}\label{AC-energy}
J(v)=\int_{\Omega}\left(\frac{|\nabla v|^{p(x)}}{p(x)}+\lambda(x)\chi_{\{v>0\}}+f(x)v\right)dx
\end{equation}
studied in \cite{LW3}, as well as in the seminal paper by Alt and Caffarelli \cite{AC} in the case $p(x)\equiv 2$ and $f\equiv 0.$ We refer also to \cite{LW4}, where 
\eqref{fb} appears in the study of an optimal design problem.

We are interested in the regularity of both the solutions and the free boundaries of viscosity solutions of (\ref{fb}). This problem has been already faced in \cite{LW2} for weak solutions of (\ref{fb}), with the aid of the techniques developed  in \cite{AC}. 

In the present work we are following the strategy introduced in the important paper by De Silva \cite{D}, that was inspired by \cite{S}, for one-phase problems and linear non-divergence operators.  \cite{D} was further extended to two-phase problems in different settings see 
\cite{DFS1, DFS2, DFS3}. The same technique was applied to the $p$-Laplace operator ($p(x)\equiv p$ in \eqref{fb}) for the one phase case, with $p\ge 2$, in \cite{LR}. 

In the linear homogeneous case, $f\equiv 0,$ \eqref{fb} was studied for viscosity solutions in the pioneer works by
Caffarelli \cite{C1,C2}. 
The results in \cite{C1,C2} have been widely generalized to different
classes of homogeneous elliptic problems. See for example \cite{CFS, FS1,
FS2} for linear operators, \cite{AF,F1, F2, Fe1,W1, W2} for fully nonlinear
operators and \cite{LN1, LN2} for the $p$-Laplacian.

	We recall that problem \eqref{fb} was originally studied in the linear homogeneous case in \cite{AC}, associated to \eqref{AC-energy}. These techniques were generalized to the linear case with $f\not\equiv 0$ in \cite{GS, Le}. In the homogeneous case,   to a quasilinear uniformly elliptic situation \cite{ACF}, to the $p$-Laplacian \cite{DP}, to an Orlicz setting \cite{MW} and to the $p(x)$-Laplacian with $p(x)\ge 2$ \cite{FMW}. Finally,  \eqref{fb} with $1<p(x)<\infty$ and $f\not\equiv 0$ was dealt with in \cite{LW2}.

\medskip

In \cite{FL} we proved that  flat free boundaries of viscosity solutions to
\eqref{fb} are $C^{1,\alpha}.$  Here we first  prove that viscosity solutions are locally Lipschitz continuous, which is the optimal regularity for the problem. Then we prove that  Lipschitz free boundaries of  viscosity solutions are $C^{1,\alpha}.$

We devote this sequel to the study of these issues, which brought challenging difficulties due to the nonlinear behavior of the $p(x)$-Laplacian and, as a consequence, we present several novelties  that are  described in detail below.

\smallskip
 
 Our main results are  the following (for notation and the precise definition of viscosity solution to \eqref{fb} we refer to Section \ref{section2})
\begin{thm}[Optimal regularity]
\label{Lip-contin} Let
$u$ be a viscosity solution to \eqref{fb}
in $B_1$. 
There exists a  constant $C>0$ such that
\begin{equation*}  
\|\nabla u\|_{L^{\infty}(B_{1/2})}\leq C.
\end{equation*}
\end{thm}

\begin{thm}[Lipschitz implies $C^{1,\alpha}$]
\label{Lipmain} Let $u$ be a viscosity solution to \eqref{fb} in $B_1$, with
$0\in F(u)$. 
 If $F(u)$ is a Lipschitz graph in a neighborhood of $0$, then 
$F(u)$ is $C^{1,\alpha}$ in a (smaller) neighborhood of $0$.
\end{thm}

\medskip

In addition to the assumptions already stated above, we suppose that
there exist positive numbers  $p_{\min},p_{\max}$ and $\gamma_0$ such that $1<p_{\min}\leq p(x)\leq p_{\max}<\infty$ and $g(x)\ge\gamma_0>0.$ 

In Theorem \ref{Lip-contin} the constant $C$ depends only on  $p_{\min}$, $p_{\max}$,  $\|\nabla p\|_{L^{\infty}(B_{3/4})}$, $ \|f\|_{L^{\infty}(B_{3/4})}$, $\|g\|_{C^{0,\beta}(\overline{B_{3/4}})}$, $\beta$, $ \|u\|_{L^{\infty}(B_{3/4})}$ and $n$ (the dimension of the space).

In Theorem \ref{Lipmain} the constant $\alpha$ depends only  on $p_{\min}$, $p_{\max}$,  $\beta$,  $\|g\|_{L^{\infty}({B_\rho})}$, $\gamma_0$ and $n$, where $\rho$ is the radius of the ball $B_\rho$ where $F(u)$ is Lipschitz. Moreover, the size of the neighborhood where $F(u)$ is $C^{1,\alpha}$ depends only on $\rho$,   
 $p_{\min}$, $p_{\max},$  $\|\nabla p\|_{L^{\infty}(B_\rho)}$, $ \|f\|_{L^{\infty}(B_\rho)}$, $\|g\|_{C^{0,\beta}(\overline{B_\rho})}$, $\gamma_0$, $\beta$, $ \|u\|_{L^{\infty}(B_{3\rho/4})}$, $n$ and the Lipschitz constant of $F(u)$.

After we develop the necessary tools, Theorem \ref{Lipmain} follows from Theorem 1.1  in \cite{FL}  ---where we proved that flat free boundaries are $C^{1,\alpha}$--- and from the main
result in  \cite{LN1}, via a blow-up argument.

As already mentioned, problem \eqref{fb} was faced in \cite{LW2} for weak solutions with the techniques developed in \cite{AC}. We want to emphasize at this point that the approach in \cite{AC} for weak solutions gives that {\it flat} free boundaries are $C^{1,\alpha}$. Alt - Caffarelli's approach does not include the result {\it Lipschitz} free boundaries are $C^{1,\alpha}$. One of the consequences of our Theorem \ref{Lipmain} is an analogous
result for weak solutions of \eqref{fb} (see Corollary \ref{Lip-lw2}).

Among the novelties that our work presents, we also refer to Section \ref{asymptoticI} where we prove some auxiliary results that are crucial in the proof of our main theorem. In that section we revisit some  lemmas  that are well known in the linear setting (see  \cite{CS} and the Appendix in \cite{C2}), for the case of $p_0$-harmonic functions  (i.e., $\Delta_{p_0} u=0$, $p_0\in (1,\infty)$). Our results concern the existence of  first order expansions at one side regular boundary points of positive Lipschitz $p_0$-harmonic functions, vanishing at the boundary of a domain. This part required great effort and passed through the equivalence of the notions of weak and viscosity solution in the case of the $p_0$-Laplace operator. Moreover, our proof can be applied to a general class of fully nonlinear degenerate elliptic operators (see Remark \ref{extens-asympt}). We strongly believe that these results are of independent interest.

We remark that, as was the case in \cite{FL}, carrying out, for the inhomogeneous $p(x)$-Laplace operator, the strategy devised in \cite{D}   required that we develop new tools. In fact, the $p(x)$-Laplacian  is a nonlinear operator that appears naturally in divergence form from minimization problems, i.e., in the form ${\rm div}A(x,\nabla u)=f(x)$, with
\begin{equation*}\label{Fan-1.6}
\lambda |\eta|^{p(x)-2}|\xi|^2\le\sum_{i,j=1}^n\frac{\partial A_i}{\partial \eta_j}(x,\eta)\xi_i\xi_j\le \Lambda |\eta|^{p(x)-2}|\xi|^2, \quad \xi\in\R^n,
\end{equation*}
where $0<\lambda\le\Lambda$. This operator is singular in the regions where $1<p(x)<2$  and  degenerate in the ones where $p(x)>2$.

Let us stress that  the main arguments in the  approach introduced in \cite{D} are based on  Harnack inequality. However, Harnack inequality for the $p(x)$-Laplacian has a different form from the standard one ---still valid for the $p_0$-Laplace operator--- even in the homogeneous case. Namely, Harnack inequality for the inhomogeneous equation $\Delta_{p(x)}u=f$, with $f$ bounded,  states that, for any nonnegative weak solution $u$ in $B_{4R},$ there exist  constants $C>0$ and $\mu\ge 0$ such that 
\begin{equation}\label{harnack-p(x)}
\sup_{B_{R}}u\leq C(\inf_{B_{R}}{u}+ R+\mu R),
\end{equation}
 where  $C$ depends on $u$, and $\mu$ depends on $||f||_{L^{\infty}(B_{4R})}$ (among other dependencies). We refer to \cite{Wo} for the the proof and further details on Harnack inequality for the inhomogeneous $p(x)$-Laplacian.

The presence of the extra term appearing in the right hand side of \eqref{harnack-p(x)} ---{\it even present  when $f\equiv 0$}, in which case $\mu=0$--- brought a major difficulty in the application of the strategy of \cite{D} for problem \eqref{fb}, under a small perturbation assumption. In order to successfully apply that strategy, we proved a new Harnack inequality for the inhomogeneous $p(x)$-Laplacian (Theorem \ref{harnack-small-perturb}) that is appropriate for small perturbation settings. Our result roughly says that if $||f||_{L^{\infty}}$ is small and $p(x)$ is close to a constant $p_0$, then the constant terms appearing in the right hand side of \eqref{harnack-p(x)} can be taken small.

This constitutes a key result in our proof of the nondegeneracy of viscosity solutions of \eqref{fb} with Lipschitz free boundaries, and it eventually leads to our main Theorem  \ref{Lipmain}. Let us emphasize that, in light of the discussion above on inequality \eqref{harnack-p(x)}, our Harnack inequality  for small perturbation settings is  indeed  of independent interest.

Another important matter,  not present in other free boundary problems treated with the present approach, is the a priori control on the dependence on $u$ in the constant $C$ appearing in Harnack inequality \eqref{harnack-p(x)}. This control is required in order to perform iteration and blow-up arguments. This same fact made the proof of Theorem \ref{Lip-contin} much more delicate.

As already mentioned, in Theorem \ref{Lipmain} we make use of the main result in \cite{LN1}. It is worth noticing that the application
of this result to our problem required nontrivial arguments  due to the different notion of solution employed in \cite{LN1} (see Secton 6, Theorem \ref{Lipmain} and Propositions \ref{casealphageq1} and 
\ref{alpha-leq-1}).

Let us remark that, as a by-product  of our theorems on the regularity of $F(u)$, we get in Corollary \ref{higher-reg-F(u)} further regularity results for $F(u)$,
under additional regularity assumptions on the data $p, f$ and $g$.

We also discuss some applications of Theorem 1.1 in \cite{FL} and  Theorem \ref{Lipmain} in the present paper (see Section
\ref{applic} and, in particular, Remark \ref{rem-concl-applic}).

Let us mention as well that our results in Sections \ref{conseq} and \ref{applic} are new even for $p(x)\equiv p_0$, with $p_0$ a constant.

We finally point out that the $p(x)$-Laplacian  is a particular case of operator with nonstandard growth. Partial differential equations with nonstandard growth have been receiving a lot of attention due to their wide range of applications. Among them we mention the modeling of non-Newtonian fluids,  for instance, electrorheological \cite{R} or thermorheological fluids \cite{AR}. Other applications  include non-linear elasticity 	\cite{Z1},  image reconstruction \cite{AMS,CLR} and the modeling of electric conductors \cite{Z2}, to cite a few.

Our work is organized as follows. In Section \ref{section2} we provide notation and basic definitions. We also recall the relationship between the different notions of solutions to $\Delta_{p(x)}u=f$ we are using.  In Section \ref{section4} we obtain a Harnack inequality for the inhomogeneous equation $\Delta_{p(x)}u=f$ (Theorem \ref{harnack-small-perturb}) that is appropriate for  small perturbation settings.
 Next, in Section \ref{section6} we prove the local Lipschitz continuity of viscosity solutions of \eqref{fb}, Theorem \ref{Lip-contin}. We then show the nondegeneracy of these solutions under the additional assumption that $F(u)$ is a Lipschitz graph. 
In Section \ref{asymptoticI} we obtain a result on asymptotic developments of positive Lipschitz $p_0$-harmonic functions at one side regular boundary points that we  use in Theorem \ref{Lipmain} and in Section \ref{conseq}. 
Then, in Section \ref{section7} we prove our main result, Theorem \ref{Lipmain}. In Section \ref{conseq} we discuss some consequences, and finally, in Section \ref{applic}, we present some applications of our results. For the sake of completeness, in Appendix \ref{appA1} we introduce the Sobolev spaces with variable exponent, which are the appropriate spaces to work with weak solutions of the $p(x)$-Laplacian. We conclude the paper with Appendix \ref{app-liouv}, where we include a Liouville type result that we use in our main theorem.

\begin{subsection}{Assumptions}\label{assump}

\smallskip

Throughout the paper we let $\Omega\subset\R^n$  be a bounded domain.

\bigskip

\noindent{\bf Assumptions on $p(x)$.} We
assume that the function $p(x)$ verifies
\begin{equation*}
p\in C^1(\Omega),\qquad 1<p_{\min}\le p(x)\le p_{\max}<\infty,\qquad  \nabla p \in L^{\infty}(\Omega),
\end{equation*}
for some positive constants $p_{\min}$ and $p_{\max}$.

\bigskip

\noindent{\bf Assumptions on $f$.} We  assume that function $f$ verifies
\begin{equation*}
f\in C(\Omega)\cap L^{\infty}(\Omega).
\end{equation*}
                                                                                               
\bigskip

\noindent{\bf Assumptions on $g$.} We assume that the
function $g$ verifies
\begin{equation*}
g\in C^{0, \beta}(\Omega)\cap L^{\infty}(\Omega),\qquad g(x)\ge\gamma_0>0,
\end{equation*}
for some positive constants $0<\beta<1$ and $\gamma_0$.

\end{subsection}

\medskip

\section{Basic definitions, notation and preliminaries}\label{section2}

In this section, we provide notation, basic definitions and some preliminaries that will be relevant for  our work.

\medskip

\noindent {\bf Notation.} For any continuous function $u:\Omega\subset \mathbb{R}^n\to \mathbb{R}$ we denote
\begin{equation}\label{def-fb}
\Omega^+(u):= \{x \in \Omega : u(x)>0\},\qquad F(u):= \partial \Omega^+(u) \cap \Omega. 
\end{equation} 
We refer to the set $F(u)$ as the {\it free boundary} of $u$, while $\Omega^+(u)$ is its {\it positive phase} (or {\it side}).

Throughout the paper, when we say that {\it $F(u)$ is Lipschitz} we are assuming that 
\begin{equation*}\label{def-fb-lip}
\Omega^+(u)= \{x=(x',x_n) \in \Omega: x_n>\psi(x')\},
\end{equation*} 
in an appropriate coordinate system, with $\psi$ Lipschitz on $\mathbb{R}^{n-1}$.

\bigskip

We begin with some remarks on the $p(x)$-Laplacian. In particular, we  recall the relationship between the different notions of solutions to $\Delta_{p(x)}u=f$ we are using, namely, weak and viscosity solutions. Then we give the definition of viscosity solution to problem (\ref{fb}) and we deduce some consequences. We here refer to the usual $C$-viscosity definition of sub/supersolution and solution of an elliptic PDE, see e.g., \cite{CIL}.

We start by observing that 
direct calculations show that, for $C^2$ functions $u$ such that $\nabla u(x)\not=0$,
\begin{equation*}
\begin{split}
&\Delta_{p(x)}u=\mbox{div} (|\nabla u|^{p(x)-2}\nabla u)\\
&=|\nabla u(x)|^{p(x)-2}\left(\Delta u+(p(x)-2)\Delta_\infty^Nu+\langle \nabla p(x),\nabla u(x)\rangle\log |\nabla u(x)| \right),
\end{split}
\end{equation*}
where
$$
\Delta_\infty^Nu:=\Big\langle D^2 u(x)\frac{\nabla u(x)}{|\nabla u(x)|}\,,\,\frac{\nabla u(x)}{|\nabla u(x)|}\Big\rangle
$$
denotes the normalized $\infty$-Laplace operator.

We also deduce that
\begin{equation}\label{p(x)-vs-pucci}
\begin{split}
&|\nabla u(x)|^{p(x)-2}\left(\mathcal{M}_{\lambda_0,\Lambda_0}^-(D^2u(x))+\langle \nabla p(x),\nabla u(x)\rangle\log |\nabla u(x)| \right)\\
&\leq \Delta_{p(x)}u\leq |\nabla u(x)|^{p(x)-2}\left(\mathcal{M}_{\lambda_0,\Lambda_0}^+(D^2u(x))+\langle \nabla p(x),\nabla u(x)\rangle\log |\nabla u(x)| \right),
\end{split}
\end{equation}
where, $\lambda_0:=\min\{1,p_{\min}-1\}$ and $\Lambda_0:=\max\{1,p_{\max}-1\}.$  
As usual, if $0<\lambda\leq \Lambda$ are numbers, and $e_i$ is the $i-$th  eigenvalue of the $n\times n$ symmetric matrix $M,$ then $\mathcal{M}_{\lambda,\Lambda}^+$ and $\mathcal{M}_{\lambda,\Lambda}^-$ denote the extremal Pucci operators and are defined  (see \cite{CC}) as
\begin{equation}\label{def-pucci}
\begin{split}
\mathcal{M}_{\lambda,\Lambda}^+(M)&=\lambda\sum_{e_i<0}e_i+\Lambda\sum_{e_i>0}e_i,\\
\mathcal{M}_{\lambda,\Lambda}^-(M)&=\Lambda\sum_{e_i<0}e_i+\lambda\sum_{e_i>0}e_i.
\end{split}
\end{equation}

First we need (see  Appendix \ref{appA1} for the definition of Sobolev spaces with variable exponent)
\begin{defn}\label{defnweak} 
Assume that $1<p_{\min}\le p(x)\le p_{\max}<\infty$
with  $p(x)$ Lipschitz continuous in $\Omega$  and  $f\in L^{\infty}(\Omega)$.

We say that $u$
is a weak solution to $\Delta_{p(x)}u=f$ in $\Omega$ if $u\in W^{1,p(\cdot)}(\Omega)$ and,  for every  $\varphi \in
C_0^{\infty}(\Omega)$, there holds that
$$
-\int_{\Omega} |\nabla u(x)|^{p(x)-2}\nabla u \cdot \nabla
\varphi\, dx =\int_{\Omega} \varphi\, f(x)\, dx.
$$
\end{defn}

We recall the following result we proved in \cite{FL} (see \cite{FL}, Theorem 3.2)
\begin{thm} \label{weak-is-visc} Let $p$ and $f$ be as in Definition \ref{defnweak}. Assume moreover that $f\in C(\Omega)$ and $p\in C^1(\Omega)$. 

 Let $u\in W^{1,p(\cdot)}(\Omega)\cap C(\Omega)$ be a weak solution to $\Delta_{p(x)}u=f$ in $\Omega$. Then $u$ is a viscosity solution to $\Delta_{p(x)}u=f$ in $\Omega$.
\end{thm}

\begin{rem} \label{equiv-not} We point out that the equivalence between weak and viscosity solutions to the $p(x)$-Laplacian with right hand side $f\equiv 0$ was proved in \cite{JLP}. On the other hand, this equivalence, in case $p(x)\equiv p$ and $f\not\equiv 0$ was dealt with in \cite{JJ} and \cite{MO}. See also \cite{JLM} for the case $p(x)\equiv p$ and $f\equiv 0$.
\end{rem}

We need the following standard notion.

\begin{defn}Given $u, \varphi \in C(\Omega)$, we say that $\varphi$
touches $u$ from below (resp. above) at $x_0 \in \Omega$ if $u(x_0)=
\varphi(x_0),$ and
$$u(x) \geq \varphi(x) \quad (\text{resp. $u(x) \leq
\varphi(x)$}) \quad \text{in a neighborhood $O$ of $x_0$.}$$ If
this inequality is strict in $O \setminus \{x_0\}$, we say that
$\varphi$ touches $u$ strictly from below (resp. above).
\end{defn}

\begin{defn}\label{defnhsol1} Let $u$ be a continuous nonnegative function in
$\Omega$. We say that $u$ is a viscosity solution to (\ref{fb}) in
$\Omega$, if the following conditions are satisfied:
\begin{enumerate}
\item $ \Delta_{p(x)} u = f$ in $\Omega^+(u)$ 
in the weak sense of Definition \ref{defnweak}.
\item For every $\varphi \in C(\Omega)$, $\varphi \in C^2(\overline{\Omega^+(\varphi)})$. If $\varphi^+$ touches $u$ from below (resp.  above) at $x_0 \in F(u)$ and $\nabla \varphi(x_0)\not=0$, then 
$$|\nabla \varphi(x_0)| \leq g(x_0)\quad (\text{resp. $ \geq g(x_0)$)}.$$
\end{enumerate}
\end{defn}

\smallskip

Next theorem follows as a consequence of our Theorem \ref{weak-is-visc}. 

\begin{thm}\label{defnhsol2} Let $u$ be a viscosity solution to (\ref{fb}) in
$\Omega.$  Then the following conditions are satisfied:
\begin{enumerate}
\item $ \Delta_{p(x)} u = f$ in $\Omega^+(u)$ in the
viscosity sense, 
that is:
\begin{itemize}
\item[(ia)] for every $\varphi\in C^2(\Omega^+(u))$ and for every $x_0\in \Omega^+(u),$ if $\varphi$ touches $u$ from above at $x_0$ and $\nabla \varphi(x_0)\not=0,$ then $\Delta_{p(x_0)}\varphi(x_0)\geq f(x_0),$ that is, $u$ is a viscosity subsolution;
\item[(ib)] for every $\varphi\in C^2(\Omega^+(u))$ and for every $x_0\in \Omega^+(u),$ if $\varphi$ touches $u$ from below at $x_0$ and $\nabla \varphi(x_0)\not=0,$ then $\Delta_{p(x_0)}\varphi(x_0)\leq f(x_0),$ that is, $u$ is a viscosity supersolution.
\end{itemize}
\item For every $\varphi \in C(\Omega)$, $\varphi \in C^2(\overline{\Omega^+(\varphi)})$. If $\varphi^+$ touches $u$ from below (resp.  above) at $x_0 \in F(u)$ and $\nabla \varphi(x_0)\not=0$, then 
$$|\nabla \varphi(x_0)| \leq g(x_0)\quad (\text{resp. $ \geq g(x_0)$)}.$$
\end{enumerate}
\end{thm}

\smallskip

\begin{rem} If  $p(x)\equiv p$ or $f\equiv 0$, then any function satisfying the conditions of Theorem \ref{defnhsol2} is a solution to \eqref{fb} in the sense of Definition \ref{defnhsol1} (see Remark \ref{equiv-not}).
\end{rem}

We introduce also the notion of comparison sub/supersolution.

\begin{defn}\label{defsub}
We say that $v \in C(\Omega)$ is a strict (comparison) subsolution (resp.
supersolution) to (\ref{fb}) in $\Omega$ if $v \in C^2(\overline{\Omega^+(v) })$,  $\nabla v\not=0$ in $\overline{\Omega^+(v) }$ and the
following conditions are satisfied:
\begin{enumerate}
\item $ \Delta_{p(x)} v  > f $ (resp. $< f $) in $\Omega^+(v)$; 
\item If $x_0 \in F(v)$, then $$|\nabla v(x_0)|>g(x_0)\quad (\text{resp. $|\nabla v(x_0)| <g(x_0)$}). $$
\end{enumerate}
\end{defn}

\smallskip

Notice that by the implicit function theorem, according to our definition, the free boundary of a comparison sub/supersolution is $C^2$.

\smallskip

As a consequence of the previous discussion we have

\begin{lem}\label{compare} Let $u$ be a viscosity solution to \eqref{fb} in $\Omega$. If $v$ is a strict  (comparison) subsolution to \eqref{fb} in $\Omega$ and  $u\ge v^+$ in $\Omega$ then $u>v$ in $\Omega^+(v)\cup F(v)$. Analogously, 
if  $v$ is a strict  (comparison) supersolution to \eqref{fb} in $\Omega$ and $v\ge u$ in $\Omega$ then $v>u$ in $\Omega^+(u)\cup F(u)$.
\end{lem}

\medskip

\noindent {\bf Notation.} From now on $B_{\rho}(x_0)\subset {\R}^n$ will denote the open ball of radius $\rho$ centered at $x_0$, and 
$B_{\rho}=B_{\rho}(0)$. A positive constant depending only on the dimension $n$, $p_{\min}$, $p_{\max}$ will be called a universal constant. We will use $c$, $c_i$ to denote small universal constants and $C$, $C_i$ to denote large universal constants.

\section{A Harnack inequality for $\Delta_{p(x)}u=f$}\label{section4}

In this section we prove a Harnack inequality for $\Delta_{p(x)}u=f$, under a small perturbation assumption (Theorem \ref{harnack-small-perturb}).

We first prove

\begin{lem}\label{approx-by-p0} Assume that $1<p_{\min}\le p(x)\le p_{\max}<\infty$
with  $p(x)$ Lipschitz continuous in $B_1$ and $\|\nabla p\|_{L^{\infty}}\leq L$, for some $L>0$.
Let $p_0$ be such that $p_{\min}\le p_0\le p_{\max}$ and  $f\in L^{\infty}(B_1)$.

Let $u\in W^{1,p(\cdot)}(B_1)\cap L^{\infty}(B_1)$
be a nonnegative solution to
\begin{equation*}
\Delta_{p(x)}u =f \quad\mbox{in } B_1,
\end{equation*}
with $||u||_{L^{\infty}(B_{1})}\le M$, for some $M>0$.

Given $\eta>0$, there exists $\ep_0=\ep_0(\eta, n,p_{\min},p_{\max}, M,L)>0$ 
such that if 
\begin{equation*}
||f||_{L^{\infty}(B_{1})}\le \ep,\qquad ||p-p_0||_{L^{\infty}(B_{1})}\le \ep,
\end{equation*}
with $\ep\le\ep_0$, then 
\begin{equation}\label{eq-u_0-b}
||u-u_0||_{L^{\infty}(B_{3/4})}\le\eta,
\end{equation}
for a suitable  $u_0\in W^{1,\infty}(B_{3/4})$ nonnegative solution to
\begin{equation}\label{eq-u_0-a}
\Delta_{p_0}u_0 =0 \quad\mbox{in }B_{3/4}.
\end{equation}
\end{lem}
\begin{proof} Let us suppose by contradiction that there exist  $\eta_0>0$ and a  sequence of nonnegative
functions   $u_k\in W^{1,p_k(\cdot)}(B_1)\cap L^{\infty}(B_1)$ with  $p_{\min}\leq p_k(x)\leq p_{\max}$, $\|\nabla
p_k\|_{L^{\infty}}\leq L$,  $||u_k||_{L^{\infty}(B_1)}\le M$, such that
\begin{equation*}
||f_k||_{L^{\infty}(B_{1})}\le \frac1k,  \qquad ||p_k-p_0||_{L^{\infty}(B_{1})}\le\frac1k,
\end{equation*} 
\begin{equation*}
\Delta_{p_k(x)}u_k =f_k \quad\mbox{in } B_1,
\end{equation*}
and such that
\begin{equation*}
||u_k-v||_{L^{\infty}(B_{3/4})}\ge\eta_0,
\end{equation*}
for every  $v\in W^{1,\infty}(B_{3/4})$ nonnegative solution to  $\Delta_{p_0}v =0$ in $B_{3/4}$.

Then, by  Theorem 1.1 in \cite{Fan} we obtain  that
\begin{equation*}\label{cota-vk-grad}
||u_k||_{C^{1,\alpha}(\overline{B_{3/4}})}\le  C \quad\mbox{ with }\quad 0<\alpha<1,
\end{equation*}
where $ C$ and $\alpha$ depend only on $n$,  $p_{\min}$, $p_{\max}$,
	$L$ and $M$.  Therefore, there is a function
$u_0\in C^{1,\alpha}(\overline{B_{3/4}})$ such that, for a subsequence,
\begin{equation*}
u_k\rightarrow u_0    \quad\mbox{and}\quad \nabla u_k\rightarrow \nabla u_0\quad\mbox {uniformly in }\overline{B_{3/4}}.
\end{equation*}
Since 
\begin{equation*}
f_k\rightarrow 0    \quad\mbox{and}\quad  p_k\rightarrow  p_0\quad\mbox {uniformly in }{B_{1}},
\end{equation*}
it follows that  $u_0\in W^{1,\infty}(B_{3/4})$ is a nonnegative solution to 
\begin{equation*}
\Delta_{p_0}u_0 =0 \quad\mbox{in }B_{3/4}
\end{equation*}
and thus,
\begin{equation*}
0<\eta_0\le||u_k-u_0||_{L^{\infty}(B_{3/4})}\to 0,
\end{equation*}
which gives a contradiction and concludes the proof. 
\end{proof}

As a consequence we get 

\begin{thm}
\label{harnack-small-perturb} Assume that $1<p_{\min}\le p(x)\le p_{\max}<\infty$
with  $p(x)$ Lipschitz continuous in $\Omega$ and $\|\nabla p\|_{L^{\infty}}\leq L$, for some $L>0$.
Let $p_0$ be such that $p_{\min}\le p_0\le p_{\max}$ and  $f\in L^{\infty}(\Omega)$. Let $x_0\in \Omega$ and $0<R_1\le R\le R_2 $ such
that $B_R(x_0)\subset\Omega$.

Let $u\in W^{1,p(\cdot)}(B_R(x_0))\cap L^{\infty}(B_R(x_0))$
be a nonnegative solution to
\begin{equation*}
\Delta_{p(x)}u =f \quad\mbox{in } B_R(x_0),
\end{equation*}
with $||u||_{L^{\infty}(B_R(x_0))}\le M$, for some $M>0$.

Given $\sigma>0$, there exist positive constants $\ep_1=\ep_1(\sigma, n,p_{\min},p_{\max}, M,L, R_1, R_2)$ and $C=C(n, p_{\min},p_{\max})$ such that if 
\begin{equation}\label{smaller-than-ep}
||f||_{L^{\infty}(B_R(x_0))}\le \ep,\qquad ||p-p_0||_{L^{\infty}(B_R(x_0))}\le \ep,
\end{equation}
with $\ep\le\ep_1$, then 
\begin{equation}\label{harnack-ineq}
\sup_{{B_{R/2}(x_0)}}u\leq C\inf_{B_{R/2}(x_0)}u+ \sigma.
\end{equation}
\end{thm}
\begin{proof} We assume without loss of generality that $x_0=0$.

{\it Case I.} Suppose first that $R=1$.

We let $\eta>0$ to be precised later. We now take $\ep_0=\ep_0(\eta, n,p_{\min},p_{\max}, M,L)$ 
given by Lemma \ref{approx-by-p0}.
 Then, 
if \eqref{smaller-than-ep} is satisfied with $\ep\le \ep_0$, there holds \eqref{eq-u_0-b}, for a suitable 
 $u_0\in W^{1,\infty}(B_{3/4})$ nonnegative solution to \eqref{eq-u_0-a}.

By Harnack's inequality (Theorem 1.1 in \cite{T}), there exists a positive constant $C=C(n, p_{\min},p_{\max})$ such that
$$\sup_{{B_{1/2}}}u_0\leq C\inf_{B_{1/2}}u_0.$$
Since $||u-u_0||_{L^{\infty}(B_{3/4})}\le\eta$, we obtain
\begin{equation*}
\begin{aligned}
\sup_{{B_{1/2}}}u&\leq \sup_{{B_{1/2}}}u_0+\eta\le   C\inf_{B_{1/2}}u_0+\eta \\
& \le  C\inf_{B_{1/2}}u +(C+1)\eta\le C\inf_{B_{1/2}}u +\sigma,
		\end{aligned}
\end{equation*}
if we choose $\eta$ such that $(C+1)\eta<\sigma$. So \eqref{harnack-ineq} follows if  \eqref{smaller-than-ep} is satisfied with
$\ep\le \tilde{\ep}_0$, where $\tilde{\ep}_0=\tilde{\ep}_0(\sigma, n,p_{\min},p_{\max}, M,L)$.

\smallskip

{\it Case II.} We now assume that $0<R_1\le R\le 1 $ and consider $\bar{u}(x)=\frac{u(Rx)}{R}$. Then, 
$\bar{u}\in W^{1,\bar{p}(\cdot)}(B_1)\cap L^{\infty}(B_1)$ is a nonnegative solution to
\begin{equation*}
\Delta_{\bar{p}(x)}\bar{u} =\bar{f} \quad\mbox{in } B_1,
\end{equation*}
with $||\bar{u}||_{L^{\infty}(B_1)}\le \overline{M}$, where $\bar{p}(x)=p(Rx)$, $\|\nabla \bar{p}\|_{L^{\infty}}\leq L$, $\bar{f}(x)=R f(Rx)$ and $\overline{M}=\frac{M}{R_1}$.

Then, if 
\begin{equation*}
||\bar{f}||_{L^{\infty}(B_1)}\le ||{f}||_{L^{\infty}(B_R)}\le \ep,\qquad 
||\bar{p}-p_0||_{L^{\infty}(B_1)}=||{p}-p_0||_{L^{\infty}(B_R)}\le \ep,
\end{equation*}
for $\ep\le \tilde{\ep}_0$, with $\tilde{\ep}_0=\tilde{\ep}_0(\sigma, n,p_{\min},p_{\max}, \overline{M},L)$ chosen as in {\it Case I}, we get
\begin{equation*}
\sup_{{B_{1/2}}}\bar{u}\leq C\inf_{B_{1/2}}\bar{u}+ \sigma.
\end{equation*}
That is,
\begin{equation*}
\sup_{{B_{R/2}}}{u}\leq C\inf_{B_{R/2}}{u}+ R\sigma\le C\inf_{B_{R/2}}{u}+ \sigma.
\end{equation*}
Hence we get \eqref{harnack-ineq}, if  \eqref{smaller-than-ep} is satisfied with $\ep\le\ep_1(\sigma, n,p_{\min},p_{\max}, M,L, R_1)$.

\smallskip

{\it Case III.} Finally, if we assume that $0<R_1\le R\le R_2$, we proceed as in {\it Case II} and we obtain the desired result with 
$\ep_1=\ep_1(\sigma, n,p_{\min},p_{\max}, M,L, R_1, R_2)$.
\end{proof}

\section{Lipschitz continuity and nondegeneracy}\label{section6}

In this section we prove Theorem \ref{Lip-contin}, which gives the optimal regularity for viscosity solutions to \eqref{fb}, i.e., the local Lipschitz continuity. We also prove that if $F(u)$ is a Lipschitz graph, then viscosity solutions to \eqref{fb} are nondegenerate.

We recall the following result we proved in \cite{FL}

\begin{lem}\label{barry} Let $x_0\in B_1$ and $0<\bar{r}_1<\bar{r}_2\le 1$. 
Assume that $1<p_{\min}\le p(x)\le p_{\max}<\infty$
 and $\|\nabla p\|_{L^{\infty}}\leq \ep^{1+\theta}$, for some $0<\theta\le 1$. Let $c_1$ and $c_2$ be positive constants.

There exist positive constants $\gamma\ge 1$, $\bar{c}$ and $\ep_0$ such that the function
\begin{equation*}
w(x)=c_1|x-x_0|^{-\gamma}-c_2, 
\end{equation*}
satisfies, for $\bar{r}_1\le |x-x_0|\le \bar{r}_2$,
\begin{equation*}
\Delta_{p(x)}w\geq \bar{c}, \quad \text{for }\, 0<\ep\le \ep_0.
\end{equation*}
Here $\gamma=\gamma(n,p_{\min}, p_{\max})$, $\bar{c}=\bar{c}(p_{\min}, p_{\max},  c_1)$ and $\ep_0=\ep_0(n,p_{\min}, p_{\max}, \bar{r}_1,  c_1)$.
\end{lem}
\begin{proof} See Lemma 4.2 in \cite{FL}.
\end{proof}

We will now prove two key estimates for  viscosity solutions to \eqref{fb}. Estimate \eqref{bound-lip} will imply that viscosity solutions are locally Lipschitz continuous (see Theorem \ref{Lip-contin}). If $F(u)$ is a Lipschitz continuous graph, we also obtain estimate \eqref{bound-nondeg}, which gives the nondegeneracy of $u$ close to $F(u)$.

We will use the notation $p_{+}^{r}=\sup_{B_{r}}p$ and $p_{-}^{r}=\inf_{B_{r}}p$, for $r>0$ (see \cite{Wo}).

\begin{prop}\label{lemma7-1} Let $p_{\min}\le p_0\le p_{\max}$ and $0<\gamma_0\le g_0\le \|g\|_{L^\infty(B_2)}$.
Let $u$ be a viscosity solution to \eqref{fb} in $B_2$
 such that $0\in F(u)$.
 There exists a constant $0<\tilde{\varepsilon}<1$ such that if
\begin{equation}\label{cindinfy}
 \|f\|_{L^\infty(B_2)}\le \tilde{\varepsilon},\quad ||g-g_0||_{L^{\infty}(B_2)}\le \tilde{\varepsilon}, \quad  
||\nabla p||_{L^{\infty}(B_2)}\le \tilde{\varepsilon},
\quad ||p-p_0||_{L^{\infty}(B_2)}\le \tilde{\varepsilon},
\end{equation}
then   
\begin{equation}\label{bound-lip}
u(x)\leq C_0 \mbox{dist}(x,F(u)),\quad x\in B^+_{1/2}(u).
\end{equation}

Assume moreover that  $F(u)$ is a Lipschitz continuous graph in $B_2$. Then 
\begin{equation}\label{bound-nondeg}
c_0\mbox{dist}(x,F(u))\leq u(x),\quad x\in B^+_{\rho_0}(u).
\end{equation}
The constants $\tilde{\ep}$,   $c_0$ and $C_0$ depend only on $n$, $p_{\min}$, $p_{\max}$, $\|g\|_{L^\infty(B_2)}$ and
$||u||_{L^{\infty}(B_{3/2})}^{p_{+}^{3/2}-p_{-}^{3/2}}$, where $p_{+}^{3/2}=\sup_{B_{3/2}}p$ and $p_{-}^{3/2}=\inf_{B_{3/2}}p$. The constants
$\tilde{\ep}$ and   $c_0$ depend also on the Lipschitz constant of $F(u)$ and on $\gamma_0$, and the constant $\rho_0$ depends only on the Lipschitz constant of $F(u)$.
\end{prop}
\begin{proof}
 Without loss of generality we will assume that $g_0=1$. We let $x_0\in B_{1/2}^+(u)$  and we denote $d=\mbox{dist}(x_0, F(u)).$ We consider the 
rescaled function
\begin{equation}\label{tilde-u}
\tilde{u}(x)=\frac{u(x_0+d x)}{d}.
\end{equation} 
Then $\tilde{u}$ is a viscosity solution to \eqref{fb} with right hand side  $\tilde{f}(x)=df(x_0 + dx)$, exponent $\tilde{p}(x)=p(x_0 + d x)$ and free boundary condition $\tilde{g}(x)=g(x_0+dx).$ Since $d\le 1$, the assumptions \eqref{cindinfy} hold for the rescaled functions in $B_{3/2}$.

In particular, $\tilde{u}$ is well defined in the ball $\overline{B_1}$, with $\tilde u>0$ in $B_1$,  and it satisfies the equation 
\begin{equation}\label{eq-tildeu}
\Delta_{\tilde{p}(x)} \tilde{u}=\tilde{f}\quad \text{ in } B_1.
\end{equation} 
We will show that 
\begin{equation}\label{upp-low-bound}
c_0\leq\tilde{u}(0)\leq C_0,
\end{equation}
for suitable universal constants $C_0,c_0>0$.

\medskip

{\it Step I: Upper bound.} Let us prove the upper bound in \eqref{upp-low-bound}. We  will argue by contradiction, assuming that $\tilde{u}(0)>C_0$, with $C_0\ge 1$ to be  precised later.

We will use  a barrier like the one considered in Lemma \ref{barry}, in the annulus $B_1\setminus \overline{B}_{r}$, with $r$ suitably chosen.

 We are going to fix $0<r<1$ in a universal way, keeping in mind  the particular form of Harnack's inequality for the $p(x)$-Laplacian
(see Theorem 2.1 \cite{Wo}). In fact,  since there holds \eqref{eq-tildeu},  it follows from  \cite{Wo} that there exists a positive constant $C_H$ such that 
\begin{equation}\label{harn-tildeu}
\sup_{B_{r}}\tilde u\leq C_H(\inf_{B_{r}}\tilde u+r({{||\tilde{f}||}_{L^\infty}}^{\frac{1}{p_{\max}-1}}+1)),
\end{equation} 
if $r< \frac{1}{4}$. Using that ${||\tilde{f}||}_{L^\infty}\le 1$ and ${||\nabla \tilde{p}||}_{L^\infty}\le 1$, we obtain  that
the constant $C_H$ depends only on $n$, $p_{\min}$, $p_{\max}$ and 
$||\tilde{u}||_{L^{\infty}(B_{4r})}^{\tilde{p}_{+}^{4r}-\tilde{p}_{-}^{4r}}$, where $\tilde{p}_{+}^{4r}=\sup_{B_{4r}}\tilde{p}$ and 
$\tilde{p}_{-}^{4r}=\inf_{B_{4r}}\tilde{p}$.

We now notice that 
\begin{equation}\label{bound-norm-u-1}
{||\tilde{u}||}_{L^{\infty}(B_{4r})}^{\tilde{p}_{+}^{4r}-\tilde{p}_{-}^{4r}}\le ||u||_{L^{\infty}(B_{d}(x_0))}^{\tilde{p}_{+}^{4r}-\tilde{p}_{-}^{4r}} \Big(\frac{1}{d}\Big)^{\tilde{p}_{+}^{4r}-\tilde{p}_{-}^{4r}},
\end{equation}
\begin{equation}\label{bound-norm-u-2}
{\tilde{p}_{+}^{4r}-\tilde{p}_{-}^{4r}}\le {||\nabla \tilde{p}||}_{L^{\infty}(B_{4r})} 8r\le d{||\nabla {p}||}_{L^{\infty}(B_{d}(x_0))} 2\le 2d,
\end{equation}
and also
\begin{equation}\label{bound-norm-u-3}
{\tilde{p}_{+}^{4r}-\tilde{p}_{-}^{4r}} = \sup_{x\in B_{4r}}{p(x_0+dx)}-\inf_{x\in B_{4r}}{p(x_0+dx)}\le 
\sup_{B_{d}(x_0)}{p}-\inf_{B_{d}(x_0)}{p}.
\end{equation}
Then, from \eqref{bound-norm-u-1}, \eqref{bound-norm-u-2} and \eqref{bound-norm-u-3} and using that $B_{d}(x_0)\subset B_{3/2}$,
 we conclude that
\begin{equation*}
{||\tilde{u}||}_{L^{\infty}(B_{4r})}^{\tilde{p}_{+}^{4r}-\tilde{p}_{-}^{4r}}\le 
c\max\Big\{ 1, \ ||u||_{L^{\infty}(B_{3/2})}^{{p}_{+}^{3/2}-{p}_{-}^{3/2}}\Big\},
\end{equation*}
where $c=\sup_{x\in (0,1)} \big(\frac{1}{x}\big)^{2x}$,  $p_{+}^{3/2}=\sup_{B_{3/2}}p$ and $p_{-}^{3/2}=\inf_{B_{3/2}}p$.

Hence, from \eqref{harn-tildeu} and the fact that ${||\tilde{f}||}_{L^\infty}\le 1$,  we deduce that for every $x\in B_r,$ 
\begin{equation}\label{stimaineq}
\frac{\tilde{u}(0)}{C_H}- 2r\leq  \inf_{B_{r}}\tilde u\leq \tilde{u}(x).
\end{equation} 
We now fix $r=\min\{\frac{1}{8}, \frac{1}{4C_H}\}$, and using that $\tilde{u}(0)>C_0\geq 1,$ we get from \eqref{stimaineq}
\begin{equation*}
\tilde{u}(x)\ge \frac{\tilde{u}(0)}{2C_H}, \quad x\in \overline{B}_r.
\end{equation*}

Next let  
$$w(x)=|x|^{-\gamma}-1,$$
where we fix $\gamma=\gamma(n,p_{\min},p_{\max})\ge 1$ given in  Lemma \ref{barry}.

We  denote
\begin{equation}\label{def-Gbar}
G(x)=\bar{C}w(x)=\bar{C}\big(|x|^{-\gamma}-1\big)
\end{equation}
in $B_1\setminus \overline{B}_{r},$ where we fix $\bar{C}=\bar{C}(r,\gamma)>0$ in such a way that $G=1$ on $\partial B_{r}.$

Let 
$$
\bar{G}(x)=kG(x)=k\bar{C}\big(|x|^{-\gamma}-1\big),\quad \text{where} \ k=\frac{C_0}{2C_H}.
$$  

Recalling that $\tilde{u}(x)\geq\frac{\tilde{u}(0)}{2C_H}>\frac{C_0}{2C_H}$ in $\bar{B}_r$, we get
\begin{equation}\label{boundary-annulus}
\begin{aligned} 
&\tilde{u}\ge 0 =\bar{G},\quad\mbox{ on }\partial B_1,\\
&\tilde{u}\geq k=\bar{G},\quad\mbox{ on } \partial B_r.
\end{aligned}
\end{equation}

We claim that
\begin{equation}\label{barGsub}
\Delta_{\tilde{p}(x)}\bar{G}\geq \tilde{f}\quad \text {in } B_1\setminus \bar{B}_r,
\end{equation}
if $\tilde{\ep}$ is suitably chosen.

In fact, by Lemma \ref{barry}, we know that 
\begin{equation*}
\Delta_{\tilde{p}(x)}\bar{G}\geq \bar{c}\quad \text {in } B_1\setminus \bar{B}_r,
\end{equation*}
with $\bar{c}=\bar{c}(p_{\min}, p_{\max},  k, \bar{C})$, if $\tilde{\ep}\le \bar{\ep}_0(n,p_{\min}, p_{\max}, {r}, k, \bar{C})$, since
$||\nabla \tilde{p}||_{L^{\infty}}\le \tilde{\varepsilon}$. 
So, if we let $\tilde{\ep}\le \bar{c}$, then $||\tilde{f}||_{L^{\infty}}\le \bar{c}$.
That is, \eqref{barGsub} holds.

Then, from \eqref{eq-tildeu}, \eqref{barGsub} and \eqref{boundary-annulus}, we conclude that $\tilde{u}\ge \bar{G}$ in 
$\overline{B_1\setminus B_r}$, with $\bar{G}\in C^2$ and $\nabla\bar{G}\neq 0$ in that set, and $\bar{G}$ touches $\tilde{u}$ from below at some $z\in\partial B_1\cap F(\tilde{u}).$ Then
$$
2> 1+\tilde{\varepsilon}\geq \tilde{g}(z) \ge |\nabla \bar{G}(z)|=\frac{C_0}{2C_H}|\nabla G(z)|=\frac{C_0\gamma \bar{C}}{2C_H},
$$
so we obtain a contradiction if we choose $C_0=\max\big\{1,  \frac{8C_H}{\gamma\bar{C}}\big\}$. Hence \eqref{bound-lip} follows.

\medskip

{\it Step II: Lipschitz estimate.}  From \eqref{bound-lip} we deduce that $u$ is Lipschitz continuous in $B_{1/4}$, with a Lipschitz constant depending only on $n$, $p_{\min}$, $p_{\max}$ and $C_0$. In fact, this can be seen with similar arguments as those in Theorem \ref{Lip-contin}, {\it Step III}. When estimating the Lipschitz constant, we use that, in the present case, 
$\|f\|_{L^\infty(B_2)}\le \tilde{\varepsilon}<1$ and $||\nabla p||_{L^{\infty}(B_2)}\le \tilde{\varepsilon}<1$.

\medskip

{\it Step III: Lower bound.} Now we assume that  $F(u)$ is a Lipschitz continuous graph in $B_2$. 
Without loss of generality we  assume that $F({u})$ is a Lipschitz graph in the direction $e_n$ with Lipschitz constant $1$.
We want to prove that $\tilde{u}$ given by \eqref{tilde-u} satisfies the lower bound in  \eqref{upp-low-bound}.

We assume moreover that our point $x_0\in B_{1/2}^+(u)$  belongs to $B_{\rho_0}$, with ${\rho_0}<1/5$. Then, $d=\mbox{dist}(x_0, F(u))<\rho_0<1/5$ so $\tilde{u}$ is well defined in the ball $\overline{B_5}$. 

Taking additionally $\rho_0<1/24$, we also obtain from the previous step that $\tilde{u}$ is Lipschitz in ${B_5}$, with Lipschitz constant depending only on $n$, $p_{\min}$, $p_{\max}$ and $C_0$. Moreover, since there exists $\bar{x}\in \partial B_1\cap F(\tilde{u})$, $||\tilde{u}||_{L^{\infty}(B_5)}$ depends only on the Lipschitz constant of $\tilde{u}$ in $B_5$.

Let us point out that also in this part of the proof we need to use more delicate arguments than those in \cite{D}. Thus, we first remark what does not change. Since $F(\tilde{u})$ is a Lipschitz continuous graph, then $\{\tilde{u}>0\}$ is a NTA domain, see \cite{JK}. This fact implies that for every couple of points $\delta$-away from $F(\tilde{u})$ in $\{\tilde{u}>0\}$  such that they are contained in a ball of size 
$\bar{M}\delta,$ there exists a Harnack chain of balls, whose length is of order $\bar{M}$, contained in the domain, connecting the two points. In other words, there exist $k$   balls in  $\{\tilde{u}>0\}$ of radius comparable to $\delta$ ($k$ depending only on $\bar{M}$), such that consecutive balls intersect,  connecting the two points.

As a consequence we will show that, in the present case, we can apply a suitable  Harnack inequality (Theorem \ref{harnack-small-perturb})  at each ball, and this will allow us to estimate the value of $\tilde{u}$ at the first point with the value of $\tilde{u}$ at the last one, times a universal constant, provided  \eqref{cindinfy} holds, for appropriate $\tilde{\ep}$.

We start by considering, for  $\eta>0$,
$$
\widetilde{G}(x)=\eta(1-{G}(x)), \quad \text{ in } B_1\setminus \bar{B_r}
$$
where ${G}$, as well as the constants   $r$, $\gamma$ and $\bar{C}$, are defined as in \eqref{def-Gbar}.

We observe that, $\nabla \widetilde{G}\neq 0$ in  $\overline{B_1\setminus {B_r}}$ and, on $\partial B_r$, 
$$
|\nabla\widetilde{G}|=\eta\bar{C} |\nabla w|=\eta\bar{C}\gamma r^{-1-\gamma},
$$
then we can choose $\eta=\eta(r,\gamma),$ so that  
$$
|\nabla\widetilde{G}|<\frac{1}{2}<1-\tilde{\varepsilon} \quad \text{ on }\partial B_r,
$$
 if  $\tilde{\ep}<\frac{1}{2}$. Now, since  
\begin{equation}\label{strict1}
\Delta_{\tilde{p}(x)}\widetilde{G}= -\Delta_{\tilde{p}(x)}\big({\eta{G}}\big),\qquad \eta {G(x)}= \eta\bar{C}\big(|x|^{-\gamma}-1\big),
\end{equation}
we can apply Lemma \ref{barry} once more and deduce that 
\begin{equation}\label{strict2}
\Delta_{\tilde{p}(x)}\big({\eta{G}}\big)\geq \hat{c}\quad\text {in } B_1\setminus \bar{B_r},
\end{equation}
with $\hat{c}=\hat{c}(p_{\min}, p_{\max},  \eta, \bar{C})$, if $\tilde{\ep}\le \hat{\ep}_0(n,p_{\min}, p_{\max}, {r}, \eta, \bar{C})$, since
$||\nabla \tilde{p}||_{L^{\infty}}\le \tilde{\varepsilon}$. 
So, if we let $\tilde{\ep}<\hat{c}$, then $||\tilde{f}||_{L^{\infty}(B_5)}< \hat{c}$ and therefore, from \eqref{strict1} and \eqref{strict2}, we get
$$\Delta_{\tilde{p}(x)}\widetilde{G}<-||\tilde{f}||_{L^{\infty}(B_5)} \quad\text {in } B_1\setminus \bar{B}_r.$$

That is, $\widetilde G$ is a strict supersolution to the rescaled free boundary problem in $B_1\setminus \bar{B}_r$.

\smallskip

Next, observe that from the assumptions we made, $F(\tilde{u})$ is a Lipschitz graph in the direction $e_n$ with Lipschitz constant $1$ and  consider the function  
$$
\widetilde{G}(x+4e_n)
$$
in $B_1(-4e_n)\setminus \overline{B_r}(-4e_n)$, which is a strict supersolution of our rescaled free boundary problem. There holds that 
$\widetilde{G}(x+4e_n)\geq 0$  as well as  $\widetilde{G}(x+4e_n)\geq \tilde{u}(x)$ in  $B(-4e_n)\setminus \overline{B_r}(-4e_n)$, since
$\tilde{u}\equiv 0$ in  $B_1(-4e_n)$.

Now we move back the graph, by a translation depending on $t>0$, until the graph of the function 
$$
\widetilde{G}(x+(4-t)e_n):  -(4-t)e_n+{(B_1\setminus \overline{B_{r}})}\to \mathbb{R}
$$
touches the graph of $\tilde{u}.$ Let say that the contact happens when $t=t^*$ at a point $\tilde{z}$ such that 
$\tilde{u}(\tilde{z})=\widetilde{G}(\tilde{z}+(4-t^*)e_n).$

Since $\widetilde{G}(x+(4-t^*)e_n))$ is a strict supersolution to the rescaled free boundary problem, recalling the comparison result (see Lemma \ref{compare}),  we conclude that  $\widetilde{G}(x+(4-t^*)e_n)$ cannot touch $\tilde{u}$ from above at the common free boundary sets, neither at interior of the annulus.

Then the contact point $\tilde{z}$ belongs to $-(4-t^*)e_n+\partial B_1.$ As a consequence $\eta=\tilde{u}(\tilde{z})$ and $\tilde{d}=\mbox{dist}(\tilde{z},F(\tilde{u}))\leq 1.$    Since $\tilde{u}$ is Lipschitz continuous with universal constant, then
$\eta=\tilde{u}(\tilde{z})\leq C\tilde{d}$ so that
\begin{equation}\label{NTAineq}
C^{-1}\eta\leq \tilde{d}\leq 1.
\end{equation}
Hence, from \eqref{NTAineq} and by applying the cited result on NTA domains, we know 
that we can construct a Harnack chain connecting $0$ and $\tilde z$, and
the length of this chain, let us say $m$, is bounded by a universal constant.

That is, we have balls $B_{r_i}(x_i)$ with radius $r_i$ comparable to $1$, $B_{2r_i}(x_i)\subset\{\tilde{u}>0\}$, $0\le i\le m$,  $x_0=0$, 
$x_m=\tilde z$ and $y_i\in B_{r_{i-1}}(x_{i-1})\cap B_{r_i}(x_i)$, for $1\le i\le m$.

We can now apply Theorem \ref{harnack-small-perturb} to $\tilde{u}$ at every ball $B_{2r_i}(x_i)$. That is, given $\sigma>0$ there exist $\varepsilon_1=\varepsilon_1(\sigma)$ and $C^*$ universal such that if $||\tilde{p}-p_0||_{L^\infty}\leq \varepsilon$ and $||\tilde{f}||_{L^\infty}\leq\varepsilon,$ with
$\ep\le\ep_1$, then
\begin{equation}\label{harn-small-tildeu}
\sup_{B_{r_i}(x_i)}\tilde{u}\leq C^*  \inf_{B_{r_i}(x_i)}\tilde{u} +\sigma.
\end{equation}
For the application of Theorem \ref{harnack-small-perturb} we need to recall that ${||\tilde{u}||}_{L^\infty(B_5)}\le M$, with $M$ universal.

Now, \eqref{harn-small-tildeu} implies that for any $x,y\in B_{r_i}(x_i)$,
$$
c\tilde{u}(x)-c\sigma\leq \tilde{u}(y),
$$
where we have denoted $c=\frac{1}{C^*}$. Then, we obtain
$$
c\tilde{u}(y_{1})-c\sigma\leq \tilde{u}(x_{0}),
$$
$$
c\tilde{u}(y_{i+1})-c\sigma\leq \tilde{u}(y_{i}),\qquad 1\le i\le m-1,
$$
and
$$
c\tilde{u}(x_{m})-c\sigma\leq \tilde{u}(y_{m}).
$$
Then, iterating we deduce
\begin{equation*}
c^{m+1}\tilde{u}(x_{m})-\sigma\sum_{j=1}^{m+1}c^j\leq \tilde{u}(x_{0}).
\end{equation*}
Thus, since $x_0=0$ and $x_m=\tilde{z},$ we have 
\begin{equation*}
c^{m+1}\tilde{u}(\tilde{z})-\sigma c\,\frac{1-c^{m+1}}{1-c}=c^{m+1}\tilde{u}(\tilde{z})-\sigma\sum_{j=1}^{m+1}c^j\leq \tilde{u}(0).
\end{equation*}
Hence denoting $c_1=c^{m+1}$ and $c_2=c\, \frac{1-c^{m+1}}{1-c},$ we obtain 
\begin{equation*}
c_1\eta-\sigma c_2=c_1\tilde{u}(\tilde{z})-\sigma c_2\leq \tilde{u}(0),
\end{equation*}
where $c_1$ and $c_2$ are universal constants. Now we fix $\sigma$ universal, $$\sigma=\frac{c_1\eta}{2c_2}.$$ 
In this way we conclude that
$$
\tilde{u}(0)\geq c_1\eta-\frac{c_1\eta}{2}=\frac{c_1\eta}{2},
$$
if $||\tilde{p}-p_0||_{L^\infty}\leq||p-p_0||_{L^\infty(B_2)}\leq \tilde{\varepsilon}$ and $||\tilde{f}||_{L^\infty}\leq||f||_{L^\infty(B_2)}\leq\tilde{\varepsilon},$ with
$\tilde{\ep}\le\ep_1(\sigma)$. Since $\eta$ is universal as well, we have finished the proof.\end{proof}

{}From Proposition \ref{lemma7-1}, we can now obtain the proof of Theorem \ref{Lip-contin}.

We  recall again the notation we use: $p_{+}^{r}=\sup_{B_{r}}p$ and $p_{-}^{r}=\inf_{B_{r}}p$, for $r>0$.

\medskip

\begin{proof}[\bf Proof of Theorem \ref{Lip-contin}]

Let $u$ be a viscosity solution to \eqref{fb} in $B_1$. We will divide the proof into several steps.

{\it Step I}. Let us fix $z_0\in B_{5/8}\cap F(u)$.   
For $0<\rho\le \frac{1}{16}$, we consider the function  
$$\bar{u}(x) = \frac{1}{\rho} u(z_0+\rho x), \quad x\in B_2.$$

Then $\bar{u}$ is a viscosity solution to \eqref{fb} in $B_2$,  with right hand side $\bar{f}(x) = \rho f(z_0+\rho x)$, exponent $\bar{p}(x) =  p(z_0+\rho x)$ and 
free boundary condition  $\bar{g}(x) =  g(z_0+\rho x)$. Moreover, $0\in F(\bar{u})$.

Let us see that we can apply the first part of Proposition \ref{lemma7-1} to ${\bar u}$, if $\rho$ is suitably chosen.

For that purpose, let us first show that the constants appearing in that proposition can be taken independent of $\rho$. More precisely, we want to 
find a bound independent of $\rho$ for
$$
||\bar{u}||_{L^{\infty}(B_{3/2})}^{\bar{p}_{+}^{3/2}-\bar{p}_{-}^{3/2}},\qquad \text{where} \quad \bar{p}_{+}^{3/2}=\sup_{B_{3/2}}\bar{p}, \ \ \bar{p}_{-}^{3/2}=\inf_{B_{3/2}}\bar{p}.
$$
In fact, we have
\begin{equation}\label{bound-norm-bu-1-a}
{||\bar{u}||}_{L^{\infty}(B_{3/2})}^{\bar{p}_{+}^{3/2}-\bar{p}_{-}^{3/2}}\le ||u||_{L^{\infty}(B_{1/8}(z_0))}^{\bar{p}_{+}^{3/2}-
\bar{p}_{-}^{3/2}} \Big(\frac{1}{\rho}\Big)^{\bar{p}_{+}^{3/2}-\bar{p}_{-}^{3/2}},
\end{equation}
and
\begin{equation}\label{bound-norm-bu-2-a}
{\bar{p}_{+}^{3/2}-\bar{p}_{-}^{3/2}}\le 3 {||\nabla \bar{p}||}_{L^{\infty}(B_{3/2})} \le  3\rho{||\nabla {p}||}_{L^{\infty}(B_{1/8}(z_0))}.
\end{equation}
Then, from \eqref{bound-norm-bu-1-a} and \eqref{bound-norm-bu-2-a},  we conclude that
\begin{equation*}
||\bar{u}||_{L^{\infty}(B_{3/2})}^{\bar{p}_{+}^{3/2}-\bar{p}_{-}^{3/2}}\le 
C=C\big(||u||_{L^{\infty}(B_{1/8}(z_0))}, {||\nabla {p}||}_{L^{\infty}(B_{1/8}(z_0))}\big).
\end{equation*}

It follows that in order to apply the first part of Proposition \ref{lemma7-1} to $\bar{u}$ we can take the constants $\tilde{\ep}$ and $C_0$ in that proposition depending only on
 $n$, $p_{\min}$, $p_{\max}$,  $||u||_{L^{\infty}(B_{1/8}(z_0))}$, ${||\nabla {p}||}_{L^{\infty}(B_{1/8}(z_0))}$ and $||g||_{L^{\infty}(B_{1/8}(z_0))}$.

Then, if $\rho$ is small enough, there holds in $B_2$
\begin{equation*}
\begin{aligned}
|\bar{f}(x)|&\le ||f||_{L^{\infty}(B_{1/8}(z_0))} \,\rho\le \tilde{\ep},\\
|\bar{g}(x)-g(z_0)|&=|g(z_0+\rho x)-g(z_0)|\le 2 [g]_{C^{0,\beta}(B_{1/8}(z_0))}\, {\rho}^{\beta}\le \tilde{\ep},\\
|\nabla \bar{p}(x)|&\le ||\nabla p||_{L^{\infty}(B_{1/8}(z_0))}\,\rho\le \tilde{\ep},\\
|\bar{p}(x)-p(z_0)|&=|p(z_0+\rho x)-p(z_0)|\le 2 ||\nabla p||_{L^{\infty}(B_{1/8}(z_0))}\,\rho\le \tilde{\ep}.
\end{aligned}
\end{equation*}
Hence, if $\rho\le \rho_0$, $\rho_0$ depending only on $\tilde{\ep}$, $||f||_{L^{\infty}(B_{1/8}(z_0))}$, $[g]_{C^{0,\beta}(B_{1/8}(z_0))}$, 
$\beta$ and $||\nabla p||_{L^{\infty}(B_{1/8}(z_0))}$, 
then $\bar{u}$ satisfies
$$ \bar{u}(x)\leq C_0 \mbox{dist}(x,F(\bar{u})),\quad x\in B^+_{1/2}(\bar{u}).$$

\medskip
{\it Step II}. We deduce from the previous step that for every $z_0\in B_{5/8}\cap F(u)$  there holds 
\begin{equation}\label{bound-dist}
{u}(x)\leq C_0 \mbox{dist}(x,F({u})), \quad  x\in B_{\rho_1}(z_0)\cap\{u>0\},
\end{equation}
for $C_0>0$ and $0<\rho_1<\frac{1}{32}$ constants depending only on 
 $n$, $p_{\min}$, $p_{\max}$,  $||u||_{L^{\infty}(B_{3/4})}$, ${||\nabla {p}||}_{L^{\infty}(B_{3/4})}$,
$||f||_{L^{\infty}(B_{3/4})}$, $||g||_{C^{0,\beta}(\overline{B_{3/4}})}$ and $\beta$ (here we have used that $B_{1/8}(z_0)\subset B_{3/4}$ for every $z_0\in B_{5/8}\cap F(u)$).

\medskip
{\it Step III}. Let $x_0\in B^+_{1/2}({u})$ such that $\mbox{dist}(x_0,F({u}))\le {\rho_1/2}.$ We will show that 
\begin{equation}\label{bound-fb}
|\nabla u(x_0)|\le C_1,
\end{equation}
for $C_1>0$ universal.

In fact, we denote $d_0=\mbox{dist}(x_0,F({u}))$ and we define $\tilde{u}(x)=\frac{1}{d_0}{u(x_0+d_0 x)}$. Then, since
$B_{d_0}(x_0)\subset\{u>0\}$, 
$$\Delta_{\tilde{p}(x)} \tilde{u} = \tilde{f} \text{ in } B_{1},$$
with  $\tilde{f}(x) = d_0 f(x_0+d_0 x)$ and $\tilde{p}(x) =  p(x_0+d_0 x)$ and therefore,
\begin{equation}\label{bound-tilde1}
{||\nabla {\tilde{p}}||}_{L^{\infty}(B_{1})}\le {||\nabla {p}||}_{L^{\infty}(B_{d_0}(x_0))}, \qquad 
{||{\tilde{f}}||}_{L^{\infty}(B_{1})}\le {|| {f}||}_{L^{\infty}(B_{d_0}(x_0))}.
\end{equation}
Since $d_0=\mbox{dist}(x_0,F({u}))$, there exists $z_0\in F(u)$ such that $|x_0-z_0|=d_0$ and recalling that $d_0<1/8$ we see that 
$z_0\in B_{5/8}\cap F(u)$. 

Also $B_{d_0}(x_0)\subset B_{2d_0}(z_0)\subset B_{\rho_1}(z_0)$. Then \eqref{bound-dist} yields
 \begin{equation*}
{u}(x)\leq C_0 \mbox{dist}(x,F({u})) \quad \text{ in } B_{d_0}(x_0).
\end{equation*}

Moreover, if $x\in B_{d_0}(x_0)$, 
$$\mbox{dist}(x,F({u}))\le |x-z_0|< 2 d_0$$
and then,
 \begin{equation*}
{u}(x)\leq C_0 \mbox{dist}(x,F({u}))\le C_0 2d_0\quad \text{ in } B_{d_0}(x_0)
\end{equation*}
which implies
\begin{equation}\label{bound-tilde2}
{||{\tilde{u}}||}_{L^{\infty}(B_{1})}= \frac{1}{d_0}{||{u}||}_{L^{\infty}(B_{d_0}(x_0))}\le 2C_0.
\end{equation} 
Hence, from Theorem 1.1 in \cite{Fan} we deduce  that $\tilde{u}\in C^{1,\alpha}(\overline{B_{1/2}})$ and $||\nabla \tilde{u}||_{L^{\infty}(B_{1/2})}\le C_1$. Taking into account \eqref{bound-tilde1} and \eqref{bound-tilde2}, we obtain that the constant
 $C_1>0$  can be taken depending  only on  $n$, $p_{\min}$, $p_{\max}$,   ${||\nabla {p}||}_{L^{\infty}(B_{3/4})}$, $||f||_{L^{\infty}(B_{3/4})}$ and $C_0$. 
It follows that
$$|\nabla u(x_0)|=|\nabla \tilde{u}(0)|\le C_1,
$$
which proves \eqref{bound-fb}.

\medskip
{\it Step IV}. Let $x_0\in B^+_{1/2}({u})$ such that $\mbox{dist}(x_0,F({u}))> {\rho_1/2}.$ We will show that 
\begin{equation}\label{bound-not-fb}
|\nabla u(x_0)|\le C_2,
\end{equation}
for $C_2>0$ universal.

In fact, there holds that $B_{\rho_1/2}(x_0)\subset\{u>0\}$ and then, 
$$\Delta_{p(x)} u = f \text{ in } B_{\rho_1/2}(x_0).$$
Now Theorem 1.1 in \cite{Fan} implies that $u\in C^{1,\alpha}(\overline{B_{\rho_1/4}(x_0)})$ and $||\nabla u||_{L^{\infty}(B_{\rho_1/4}(x_0))}\le C_2,$
where $C_2>0$ is a constant that can be taken depending only on  $n$, $p_{\min}$, $p_{\max}$,  $||u||_{L^{\infty}(B_{3/4})}$, ${||\nabla {p}||}_{L^{\infty}(B_{3/4})}$, $||f||_{L^{\infty}(B_{3/4})}$ and $\rho_1$. This proves \eqref{bound-not-fb} and completes the proof.
\end{proof}

\section{Asymptotic expansions}\label{asymptoticI}
In this section we revisit some  lemmas  that are well known in the linear setting (see  \cite{CS} and the Appendix in \cite{C2}), for the case of $p_0$-harmonic functions  (i.e., $\Delta_{p_0} u=0$, $p_0\in (1,\infty)$). Our results ---that are used in 
Theorem \ref{Lipmain} and Section \ref{conseq}---   concern the existence of  first order expansions at one side regular boundary points of positive Lipschitz $p_0$-harmonic functions, vanishing at the boundary of a domain. The proof can be applied to a general class of fully nonlinear degenerate elliptic operators (see Remark \ref{extens-asympt}).

For the notion of solution we refer to Definition \ref{defnweak} and Remark \ref{equiv-not}. Our result is the following
\begin{lem}\label{lemma51} Let $1<p_0<\infty$ and let $u$ be a positive Lipschitz $p_0$-harmonic function in a domain $\Omega\subset \mathbb{R}^n.$  Let $x_0\in\partial \Omega$ and assume that $u$ vanishes continuously on $\partial \Omega\cap B_\rho(x_0),$ for some $\rho>0.$
\begin{itemize}
\item[(a)] If there exists $B_r(y)\subset \Omega$ such that $x_0\in\partial B_r(y)$, then
\begin{equation*}
u(x)= \alpha \langle x-x_0,\nu\rangle^+ +o(|x-x_0|),
\end{equation*}
 in the ball $B_r(y),$ with $\alpha>0$ and $\nu=\frac{y-x_0}{|y-x_0|}.$
\item[(b)] If there exists a ball $B_r(y)\subset \Omega^c$ such that $x_0\in\partial B_r(y)$, then
\begin{equation*}
u(x)= \beta \langle x-x_0,\nu\rangle^+ +o(|x-x_0|),
\end{equation*}
 with $\beta\geq 0$ and  $\nu=\frac{x_0-y}{|x_0-y|}$.  
In addition, if $\beta>0,$ then $B_r(y)$ is tangent to $\partial \Omega$ at $x_0.$
\end{itemize}
\end{lem}
\begin{proof} We  will assume, without loss of generality, that $x_0=0$, $\nu=e_n$ and $\rho>1$. 
We will let $\lambda_0:=\min\{1,p_{0}-1\}$ and $\Lambda_0:=\max\{1,p_{0}-1\}$.

We define
\begin{equation}\label{extend}
\tilde{u}=\left\{
\begin{array}{l}
u\quad x\in \bar{\Omega}\cap \overline{B_1},\\
0\quad x\in   \bar{\Omega}^c\cap \overline{B_1}. \\
\end{array}
\right.
\end{equation}
Hence $\tilde{u}$ is Lipschitz in $\overline{B_1}$. To simplify the  notation we will denote $\tilde{u}$ as $u.$

\smallskip

{\bf Case (a).} Let $h$ be the solution of 
\begin{equation*}
\left\{
\begin{array}{l}
\mathcal{M}_{\lambda_0, \Lambda_0}^- (D^2h)=0,\quad B_{r}(y)\setminus \overline{B_{r/2}}(y)\\
h=0,\quad\mbox{on}\:\: \partial B_r(y)\\
h=\min_{\overline{B_{r/2}}(y)}{u},\quad\mbox{on}\:\: \partial B_{r/2}(y).
\end{array}
\right.
\end{equation*}
Let $h\equiv \min_{\overline{B_{r/2}}(y)}u$ in $B_{r/2}(y)$ and $h\equiv 0$ in $B^c_r(y)$.
Then, $h\ge 0$, $h\in C^{2}(\overline{B_{r}(y)\setminus B_{r/2}(y)})$, see \cite{CC}, and 
\begin{equation}\label{devel-h-a}
h(x)=cx_n^++o(|x|),\quad c>0.
\end{equation}

 In addition, recalling \eqref{p(x)-vs-pucci}, we have in $B_r(y)$, in the viscosity sense,
\begin{equation*}
\begin{split}
0= \Delta_{p_0} {u}(x)&=|\nabla{u}(x)|^{p_0-2}\left(\Delta {u}+(p_0-2)\langle D^2{u}(x)\frac{\nabla {u}(x)}{|\nabla {u}(x)|},\frac{\nabla {u}(x)}{|\nabla {u}(x)|}\rangle\right)\\
&\geq |\nabla{u}(x)|^{p_0-2}\mathcal{M}^-_{\lambda_0,\Lambda_0}(D^2u(x)).
\end{split}
\end{equation*}
Hence, applying Lemma 6 in \cite{IS}, we conclude that $\mathcal{M}^-_{\lambda_0,\Lambda_0}(D^2u(x))\leq 0$, in the viscosity sense, in $B_r(y).$
 Since $u\ge0$ in $B_r(y)$,  we deduce that
$u\geq h$ in $B_{r}(y).$
We define now
$$
\alpha_0=\sup\{m:\quad  {u}(x)\geq mh(x) ,\quad  B_1\cap B_r(y)\}
$$
and for  $k\in \mathbb{N}$
$$
\alpha_k=\sup\{m:\quad {u}(x)\geq mh(x),\quad B_{2^{-k}}\cap B_r(y) \}.
$$
In particular these sets are well defined and not empty, since $m=1$ belongs to all of them. 
The sequence $\{\alpha_k\}_{k\in \mathbb{N}}$ is increasing and bounded because $u$ is Lipschitz. Let
$$
\tilde{\alpha}=\lim_{k\to \infty}\alpha_k.
$$
{}From the definition of $\tilde{\alpha}$ there holds that $\tilde{\alpha}>0$ and
\begin{equation}\label{inf-alpha}
\liminf_{x\to 0,\:\:x\in B_r(y)} \frac{u(x)-\tilde{\alpha}h(x)}{|x|}\ge 0.
\end{equation}
Let us show that
\begin{equation}\label{sup-alpha}
\limsup_{x\to 0,\:\:x\in B_r(y)} \frac{u(x)-\tilde{\alpha}h(x)}{|x|}\le 0.
\end{equation}
Then, \eqref{devel-h-a}, \eqref{inf-alpha} and \eqref{sup-alpha} will give the desired result.

We argue by contradiction assuming that
$$
\limsup_{x\to 0,\:\:x\in B_r(y)} \frac{u(x)-\tilde{\alpha}h(x)}{|x|}=2\delta>0.
$$
Hence, there exists a sequence  $x^k\in B_r(y),$ $|x^{k}|=r_k\to 0$   such that for every $k$
\begin{equation*}
 \frac{u(x^{k})-\tilde{\alpha}h(x^{k})}{|x^{k}|}\geq \delta.
\end{equation*}

We define $r_k=|x^{k}|,$ $y^{k}=\frac{x^{k}}{r_k},$ so that $r_k\to 0$ and $|y^{k}|=1.$ Moreover, we denote
$$
u_k(x):=\frac{u(r_kx)}{r_k},\quad h_k(x):=\frac{h(r_kx)}{r_k}.
$$
Since $u$ and $h$ are Lipschitz in $\overline B_1$ and $u(0)=h(0)=0$ then, there exists $v$ Lipschitz continuous  in $\mathbb{R}^n$ such that, for a subsequence,
$$
{u}_k-\tilde{\alpha}h_k\to v
$$
uniformly on compact sets and such  that $y^{k}\to y^0,$ $ |y^{0}|=1,$ $y^{0}_n\geq 0.$ Since
$$
{u}_k(y^{k})-\tilde{\alpha}h_k(y^{k})\geq \delta,
$$
as a consequence  $v(y^{0})\geq \delta.$ Then there exists $z^0$ with $|z^0|=1$,  $z^0_n>0$ and $\overline{B_\varepsilon(z^0)}\subset \{x_n>0\}$ such that 
$$
v(x)\geq \frac{\delta}{2}\qquad \text{ in } B_\varepsilon(z^0)
$$
and 
\begin{equation}\label{bound-z0-a}
\begin{split}
 {u}_{k}(x)-\tilde{\alpha}h_k(x)\geq \frac{\delta}{2}\qquad \text{ in } B_\varepsilon(z^0).
\end{split}
\end{equation}

We know that in $B_r(y)\cap B_{2^{-k}}$
$$
u(x)\geq \alpha_k h(x),
$$
and $r_k\to 0.$ We take a sequence $j_k\to +\infty$ such that $r_k<2^{-j_k}$
 and then
$$
{u}(x)\geq \alpha_{j_k}h(x)\quad \text{ in }B_r(y)\cap B_{2^{-{j_k}}}.
$$
Hence
$$
{u}(x)\geq \alpha_{j_k}h(x)\quad\text{ if }|x|<r_k, \  x\in B_r(y).
$$
 As a consequence, 
$$
{u}_k(x)\geq \alpha_{j_k}h_k(x)\quad \text{ if }|x|<1, \ |x-\frac{y}{r_k}|<\frac{r}{r_k},
$$
and, recalling \eqref{bound-z0-a}, we have ${u}_k(x)- \alpha_{j_k}h_k(x)\geq \frac{\delta}{2}$ in $B_\varepsilon(z^0).$
We also observe that 
\begin{equation}\label{ccrefr-a}
\begin{aligned}
\mathcal{M}_{\lambda_0,\Lambda_0}^-(D^2h_k)&=0\qquad\mbox{ in }B_{\frac{r}{r_k}}(\frac{y}{r_k})\cap \overline{B}_{\frac{r}{2r_k}}^c(\frac{y}{r_k}),\\ 
\mathcal{M}_{\lambda_0,\Lambda_0}^-(D^2{u}_k)&\leq 0\qquad\mbox{ in }B_1\cap  {B}_{\frac{r}{r_k}}(\frac{y}{r_k}).
\end{aligned}
\end{equation}
Hence, in the viscosity sense, if $k$ is large,
$$
\mathcal{M}_{\lambda_0,\Lambda_0}^-(D^2({u}_k-\alpha_{j_k}h_k))\leq 0\quad\mbox{in}\:\:\:B_1\cap  {B}_{\frac{r}{r_k}}(\frac{y}{r_k}).
$$
In fact, since $\alpha_{j_k}h_k\in C^2,$ we get from \eqref{ccrefr-a}, reasoning as in  Proposition 2.13 in \cite{CC},  
$$
\mathcal{M}_{\lambda_0,\Lambda_0}^-(D^2({u}_k-\alpha_{j_k}h_k))\leq -\mathcal{M}_{\lambda_0,\Lambda_0}^-(D^2(\alpha_{j_k}h_k))=0
$$
in $B_1\cap {B}_{\frac{r}{r_k}}(\frac{y}{r_k})$. 
We now consider, for large $k$, $w_k$ satisfying
\begin{equation*}
\begin{split}
\left\{\begin{array}{l}
\mathcal{M}_{\lambda_0,\Lambda_0}^-(D^2 w_k)=0\quad \text{ in }D_k:=B_1\cap {B}_{\frac{r}{r_k}}(\frac{y}{r_k}),\\
w_k=\frac{\delta}{4}\varphi\quad \text{ on } B_{\varepsilon}(z^0)\cap \partial D_k,\\
w_k=0\quad \text{ on } \partial D_k\setminus B_\varepsilon (z^0),
\end{array}
\right.
\end{split}
\end{equation*}
\begin{equation*}
\varphi\in C_0^{\infty}(B_{\varepsilon} (z^0)), \quad 0\le \varphi\le 1, \quad \varphi\equiv 1 \text{ in }B_{\varepsilon/2} (z^0).
\end{equation*}
 Then $w_k\in C(\overline{D_k})\cap C^{2}(\overline{{D}_k\cap B_{1/2}}),$ $w_k\geq 0$ in 
$\overline{D_k}$ and 
$$
{u}_k-\alpha_{j_k}h_k\geq w_k\quad \mbox{ in } \overline{D_k}.
$$

We claim that there exist $\mu>0$ and $\tilde{\rho}_0>0$ such that,  for large  $k$,
\begin{equation}\label{bound-mu-a}
\frac{w_k(x)}{h_k(x)}\geq \mu\quad \text{ in }B_{\tilde{\rho}_0}\cap {B}_{\frac{r}{r_k}}(\frac{y}{r_k}).
\end{equation}
In fact, we consider $\varphi_k$ a $C^2$ diffeomorphism which maps, for $\rho_0$ small, $B_{\rho_0}\cap {B}_{\frac{r}{r_k}}(\frac{y}{r_k})$ in $B_1^+:=B_1\cap\{x_n>0\}$, with $\varphi_k(B_{\rho_0}\cap \partial{B}_{\frac{r}{r_k}}(\frac{y}{r_k}))=B_1\cap\{x_n=0\}$ and 
$\varphi_k(0)=0$. We choose $\varphi_k$ with uniformly bounded $C^2$ norms. Then, we define
$$
\tilde{w}_k(x)=w_k(\varphi_k^{-1}(x)),\quad \tilde{h}_k(x)=h_{k}(\varphi_{k}^{-1}(x))\quad \text{ for }x\in B_{1}^+.
$$ 
We first observe that, for every $M,N\in\mathcal{S}^{n\times n}$, the following inequalities hold (see Lemma 2.10 in \cite{CC})
\begin{equation*}
\mathcal{M}_{\lambda_0,\Lambda_0}^+(M-N)\geq\mathcal{M}_{\lambda_0,\Lambda_0}^-(M)-\mathcal{M}_{\lambda_0,\Lambda_0}^-(N)\geq \mathcal{M}_{\lambda_0,\Lambda_0}^-(M-N).
\end{equation*}

Then, we can apply Proposition 2.1 in \cite{SS} with $F(M):=\mathcal{M}_{\lambda_0,\Lambda_0}^-(M)$ and we obtain that
\begin{equation*}
\tilde{F}_k(D^2\tilde{w}_k(x),D\tilde{w}_k(x),x)=0\quad\mbox{in}\:\:\:B_1^+,
\end{equation*}
where,  for  $M\in\mathcal{S}^{n\times n}$,  $q\in \mathbb{R}^n$ and $x\in B_1^+$,
\begin{equation*}
\tilde{F}_k(M,q,x):=\mathcal{M}_{\lambda_0,\Lambda_0}^-(D\varphi^T_k(\varphi^{-1}_k(x))MD\varphi_k(\varphi^{-1}_k(x))+qD^2\varphi_k(\varphi^{-1}_k(x))),
\end{equation*}
with $\tilde{F}_k$ satisfying for every $M,N\in\mathcal{S}^{n\times n}$,  $p,q\in \mathbb{R}^n$ and $x\in B_1^+$,
\begin{equation*}
\begin{split}
\mathcal{M}_{\lambda_0,\Lambda_0}^+(M-N)+K|p-q|&\geq\tilde{F}_k(M,p,x)-\tilde{F}_k(N,q,x)\\
&\geq \mathcal{M}_{\lambda_0,\Lambda_0}^-(M-N)-K|p-q|.
\end{split}
\end{equation*}
Here $K$ is a fixed constant depending only on the uniform bound of the $C^2$ norms of $\varphi_k.$

As a consequence, $\tilde{w}_k$  satisfy   in the viscosity sense, the following set of inequalities
\begin{equation}\label{condiSS-a}
\begin{split}
\left\{\begin{array}{l}
 \mathcal{M}_{\lambda_0,\Lambda_0}^+(D^2\tilde{w}_k)+K|\nabla\tilde{w}_k|\geq 0\quad \text{ in }B_1^+,\\
\mathcal{M}_{\lambda_0,\Lambda_0}^-(D^2\tilde{w}_k)-K|\nabla\tilde{w}_k|\leq 0\quad \text{ in }B_1^+.
\end{array}
\right.
\end{split}
\end{equation}

 With similar arguments we  obtain that 
$\tilde{h}_k$ satisfy in the viscosity sense  the inequalities in \eqref{condiSS-a} in $B_1^+$,  as well.

We also notice that, since $h_k(x)\to cx_n^+$ uniformly on compact sets of $\mathbb{R}^n,$ with $c>0,$ then, 
$\tilde{h}_k(\frac{1}{2}e_n)=h_k(\varphi_k^{-1}(\frac{1}{2}e_n))\to \tilde{c}>0$.

Hence, we can apply   Proposition 2.4 of \cite{SS} and we get
\begin{equation}\label{bound-hk-a}
\tilde{h}_k(x)\leq C\tilde{h}_k(\frac{1}{2}e_n)x_n\leq C_0x_n \quad\text{ in }B_{1/2}^+, 
\end{equation}
for a positive constant $C_0$ and large $k.$ 

On the other hand,  for $k_1$ large and fixed, there holds 
 $w_k\geq w_{k_1}$ in $D_{k_1}$, for $k\ge k_1$. We remark that $w_{k_1}>0$ in $D_{k_1}.$ Thus, for any $0<r_0<1$, 
\begin{equation}\label{finalc1-a}
\tilde{w}_k(\frac{r_0}{2}e_n)=w_k(\varphi^{-1}_k(\frac{r_0}{2}e_n))\geq w_{k_1}(\varphi^{-1}_k(\frac{r_0}{2}e_n))\to \tilde{c}_{r_0}>0.
\end{equation}
Now the application of Proposition 2.5 in \cite{SS} to $\tilde{w}_k^{r_0}(x):=\tilde{w}_k(r_0x)$, for $r_0>0$ universal and small, gives 
\begin{equation*}\label{finalc0}
\tilde{w}_k(x)\geq c_0\tilde{w}_k(\frac{r_0}{2}e_n)x_n\quad \mbox{ in }B^+_{\frac{r_0}{2}},
\end{equation*}
for a positive constant $c_0$.
 Hence, using \eqref{finalc1-a} with this choice of $r_0$,  we get 
\begin{equation}\label{finalc2-a}
\tilde{w}_k(x)\geq c_1 x_n\quad \text{ in }B^+_{\frac{r_0}{2}},
\end{equation}
for $c_1$ a positive constant and large $k$.
Thus, from \eqref{bound-hk-a} and \eqref{finalc2-a}, we obtain
$$
\frac{\tilde{w}_k(x)}{\tilde{h}_k(x)}\geq \frac{c_1}{C_0}:=\mu\quad \text{ in }B_{\rho_1}^+,
$$
 for $\rho_1>0$ small and large $k.$ Now, going back to the original variables, we conclude that
$$
\frac{w_k(x)}{h_k(x)}\geq \mu \quad \text{ in }B_{\tilde{\rho}_0}\cap {B}_{\frac{r}{r_k}}(\frac{y}{r_k}),
$$
for some constants $\tilde{\rho}_0>0$ and $\mu>0,$ and large $k.$ That is, \eqref{bound-mu-a} holds.

Finally, since 
$$
u_k -\alpha_{j_k}h_k\geq w_k \quad \text{ in   }B_{1}\cap {B}_{\frac{r}{r_k}}(\frac{y}{r_k}),
$$
 we get
$$
u_k -\alpha_{j_k}h_k\geq w_k=\frac{w_k}{h_k}h_k\geq \mu h_k\quad\text{ in }B_{\tilde{\rho}_0}\cap {B}_{\frac{r}{r_k}}(\frac{y}{r_k}).
$$
 As a consequence,
$$
u_k - (\alpha_{j_k}+\mu)h_k\geq 0\quad \text{ in }B_{\tilde{\rho}_0}\cap {B}_{\frac{r}{r_k}}(\frac{y}{r_k}).
$$
 Then, in the original variables,  we have
$$
u(r_kx)-(\alpha_{j_k}+\mu)h(r_kx)\geq 0\quad \text{ when }|x|\leq \tilde{\rho}_0, \ |x-\frac{y}{r_k}|<\frac{r}{r_k},
$$
or, equivalently, when
$|r_kx|\leq r_k \tilde{\rho}_0,$ $|r_kx-y|<r$.
Since $\alpha_{j_k}+\mu\to\tilde{\alpha}+\mu$, there holds that 
   $\alpha_{j_k}+\mu\geq \tilde{\alpha}+{\mu/2},$  if $k$ is large enough. Hence, 
$$
u - (\tilde{\alpha}+\mu/2)h\geq 0\quad \text{ in  }B_{r_{k_0}\tilde{\rho}_0}\cap B_{r}(y),
$$
 for some suitable $k_0$. As a consequence, if $2^{-k}\leq r_{k_0}\tilde{\rho}_0$,
$$
 u - (\alpha_{k}+\mu/2)h\geq u - (\tilde{\alpha}+\mu/2)h\geq 0 \quad \text{ in  }B_{2^{-k}}\cap B_{r}(y),
$$
 but this contradicts the definition of $\alpha_k$ and completes the proof.

\medskip

{\bf Case (b).} Recalling \eqref{extend}, we have that $\tilde{u}$ is Lipschitz in $\overline{B_1}$,  satisfies $\Delta_{p_0}\tilde{u}\geq 0$ in $B_1$ in the sense of Definition 2.2 in \cite{JLM} and, by Theorem 2.5 of that paper, in the viscosity sense. We again denote $\tilde u$ as $u$. 
Without loss of generality we may suppose that $B_{2r}(y)\subset B_1.$

Let $h$ be the solution of 
\begin{equation*}
\left\{
\begin{array}{l}
\mathcal{M}_{\lambda_0, \Lambda_0}^+ (D^2h)=0,\quad B_{2r}(y)\setminus \overline{B_r}(y)\\
h=0,\quad\mbox{on}\:\: \partial B_r(y)\\
h=\max_{\partial B_{2r}(y)}{u},\quad\mbox{on}\:\: \partial B_{2r}(y),
\end{array}
\right.
\end{equation*}
and define $h\equiv 0$ in $B_r(y)$. 
Then, $h\ge 0$, $h\in C^{2}(\overline{B_{2r}(y)\setminus B_r(y)}),$ see \cite{CC}, and 
\begin{equation}\label{devel-h}
h(x)=cx_n^++o(|x|),\quad c>0.
\end{equation}
 In addition, recalling \eqref{p(x)-vs-pucci}, we have in $B_1$, in the viscosity sense,
\begin{equation*}
\begin{split}
0\leq \Delta_{p_0} {u}(x)&=|\nabla{u}(x)|^{p_0-2}\left(\Delta {u}+(p_0-2)\langle D^2{u}(x)\frac{\nabla {u}(x)}{|\nabla {u}(x)|},\frac{\nabla {u}(x)}{|\nabla {u}(x)|}\rangle\right)\\
&\leq |\nabla{u}(x)|^{p_0-2}\mathcal{M}^+_{\lambda_0,\Lambda_0}(D^2u(x)).
\end{split}
\end{equation*}
Hence, applying Lemma 6 in \cite{IS}, we conclude that $\mathcal{M}^+_{\lambda_0,\Lambda_0}(D^2u(x))\geq 0$, in the viscosity sense, in $B_1.$
 Since $u=0$ on $\partial B_r(y)$, then $u\leq h$ on $\partial (B_{2r}(y)\setminus \overline{B}_r(y))$, thus we deduce that
$u\leq h$ in $B_{2r}(y)\setminus \overline{B}_r(y).$
We define now
$$
\beta_0=\inf\{m:\quad  mh(x)\geq {u}(x),\quad  B_1\cap B_r^c(y)\}
$$
and for  $k\in \mathbb{N}$
$$
\beta_k=\inf\{m:\quad mh(x)\geq {u}(x),\quad B_{2^{-k}}\cap B_r^c(y) \}.
$$
In particular these sets are well defined and not empty, for $k\ge k_0$, since $m=1$ belongs to all of them. The sequence $\{\beta_k\}_{k\in \mathbb{N}}$ is monotone decreasing, so that
$$
\tilde{\beta}:=\inf_{k\in\mathbb{N}}\beta_k\geq 0,
$$
because $\beta_k\geq 0$ for  $k\in \mathbb{N}.$ There holds that
\begin{equation}\label{limsup-beta}
\limsup_{x\to 0,\:\:x\in B_r^c(y)}\frac{{u}(x)-\tilde{\beta}h(x)}{|x|}\leq 0.
\end{equation}

We will show that 
\begin{equation}\label{liminf-beta}
\liminf_{x\to 0,\:\:x\in B_r^c(y)}\frac{{u}(x)-\tilde{\beta}h(x)}{|x|}\geq 0.
\end{equation}
Then, \eqref{devel-h}, \eqref{limsup-beta} and \eqref{liminf-beta} will give the desired result.

We will proceed by contradiction. In fact, assume that there exists $\delta>0$ such that $$
\liminf_{x\to 0,\:\:x\in B_r^c(y)}\frac{{u}(x)-\tilde{\beta}h(x)}{|x|}=-2\delta.
$$
Then, there exists a sequence $\{x^{k}\}_{k\in\mathbb{N}}\subset B_r^{c}(y)$, $x^k\to 0$, such that 
$$
\frac{{u}(x^{k})-\tilde{\beta}h(x^{k})}{|x^{k}|}\leq -\delta.
$$
We define $r_k=|x^{k}|,$ $y^{k}=\frac{x^{k}}{r_k},$ so that $r_k\to 0$ and $|y^{k}|=1.$ Moreover, we denote
$$
u_k(x):=\frac{u(r_kx)}{r_k},\quad h_k(x):=\frac{h(r_kx)}{r_k}.
$$
Since $u$ is Lipschitz in $B_1$, $h\in C^{2}(\overline{B_{2r}(y)\setminus B_r(y)})$ and $u(0)=h(0)=0$ then, there exists $v$ Lipschitz continuous  in $\mathbb{R}^n$ such that, for a subsequence,
$$
{u}_k-\tilde{\beta}h_k\to v
$$
uniformly on compact sets and such  that $y^{k}\to y^0,$ $ |y^{0}|=1,$ $y^{0}_n\geq 0.$ Since
$$
{u}_k(y^{k})-\tilde{\beta}h_k(y^{k})\leq -\delta,
$$
as a consequence  $v(y^{0})\leq -\delta.$ Then there exists $z^0$ with $|z^0|=1$,  $z^0_n>0$ and $\overline{B_\varepsilon(z^0)}\subset \{x_n>0\}$ such that 
$$
v(x)\leq -\frac{\delta}{2}\qquad \text{ in } B_\varepsilon(z^0)
$$
and 
\begin{equation}\label{bound-z0}
\begin{split}
 {u}_{k}(x)-\tilde{\beta}h_k(x)\leq -\frac{\delta}{2}\qquad \text{ in } B_\varepsilon(z^0).
\end{split}
\end{equation}

We know that in $B_r^c(y)\cap B_{2^{-k}}$
$$
u(x)\leq \beta_k h(x),
$$
and $r_k\to 0.$ We take a sequence $j_k\to +\infty$ such that $r_k<2^{-j_k}$
 and then
$$
{u}(x)\leq \beta_{j_k}h(x)\quad \text{ in }B_r^c(y)\cap B_{2^{-{j_k}}}.
$$
Hence
$$
{u}(x)\leq \beta_{j_k}h(x)\quad\text{ if }|x|<r_k, \  x\in B_r^c(y).
$$
 As a consequence, 
$$
{u}_k(x)\leq \beta_{j_k}h_k(x)\quad \text{ if }|x|<1, \ |x-\frac{y}{r_k}|>\frac{r}{r_k},
$$
and, recalling \eqref{bound-z0}, we have ${u}_k(x)- \beta_{j_k}h_k(x)\leq -\frac{\delta}{2}$ in $B_\varepsilon(z^0).$
We also observe that 
\begin{equation}\label{ccrefr}
\begin{aligned}
\mathcal{M}_{\lambda_0,\Lambda_0}^+(D^2h_k)&=0\qquad\mbox{ in }B_{\frac{2r}{r_k}}(\frac{y}{r_k})\cap \overline{B}_{\frac{r}{r_k}}^c(\frac{y}{r_k}),\\ 
\mathcal{M}_{\lambda_0,\Lambda_0}^+(D^2{u}_k)&\geq 0\qquad\mbox{ in }B_1\cap  \overline{B}_{\frac{r}{r_k}}^c(\frac{y}{r_k}).
\end{aligned}
\end{equation}
Hence, in the viscosity sense,
$$
\mathcal{M}_{\lambda_0,\Lambda_0}^+(D^2({u}_k-\beta_{j_k}h_k))\geq 0\quad\mbox{in}\:\:\:B_1\cap  \overline{B}_{\frac{r}{r_k}}^c(\frac{y}{r_k}).
$$
In fact, since $\beta_{j_k}h_k\in C^2,$ we get from \eqref{ccrefr}, reasoning as in  Proposition 2.13 in \cite{CC},  
$$
\mathcal{M}_{\lambda_0,\Lambda_0}^+(D^2({u}_k-\beta_{j_k}h_k))\geq -\mathcal{M}_{\lambda_0,\Lambda_0}^+(D^2(\beta_{j_k}h_k))=0
$$
in $B_1\cap \overline{B}_{\frac{r}{r_k}}^c(\frac{y}{r_k})$. Thus, we deduce that 
$$
\mathcal{M}_{\lambda_0,\Lambda_0}^-(D^2(\beta_{j_k}h_k-{u}_k))\leq 0\quad \mbox{ in }B_1\cap \overline{B}_{\frac{r}{r_k}}^c(\frac{y}{r_k}).
$$                                
We now consider $w_k$ satisfying
\begin{equation*}
\begin{split}
\left\{\begin{array}{l}
\mathcal{M}_{\lambda_0,\Lambda_0}^-(D^2 w_k)=0\quad \text{ in }D_k:=B_1\cap \overline{B}_{\frac{r}{r_k}}^c(\frac{y}{r_k}),\\
w_k=\frac{\delta}{4}\varphi\quad \text{ on } B_{\varepsilon}(z^0)\cap \partial D_k,\\
w_k=0\quad \text{ on } \partial D_k\setminus B_\varepsilon (z^0),
\end{array}
\right.
\end{split}
\end{equation*}
\begin{equation}\label{bound-data}
\varphi\in C_0^{\infty}(B_{\varepsilon} (z^0)), \quad 0\le \varphi\le 1, \quad \varphi\equiv 1 \text{ in }B_{\varepsilon/2} (z^0).
\end{equation}
 Then $w_k\in C(\overline{D_k})\cap C^{2}(\overline{{D}_k\cap B_{1/2}}),$ $w_k\geq 0$ in 
$\overline{D_k}$ and 
$$
\beta_{j_k}h_k-{u}_k\geq w_k\quad \mbox{ in } \overline{D_k}.
$$

We claim that there exist $\mu>0$ and $\tilde{\rho}_0>0$ such that,  for large  $k$,
\begin{equation}\label{bound-mu}
\frac{w_k(x)}{h_k(x)}\geq \mu\quad \text{ in }B_{\tilde{\rho}_0}\cap \overline{B}_{\frac{r}{r_k}}^c(\frac{y}{r_k}).
\end{equation}
In fact, we consider $\varphi_k$ a $C^2$ diffeomorphism which maps, for $\rho_0$ small, $B_{\rho_0}\cap \overline{B}_{\frac{r}{r_k}}^c(\frac{y}{r_k})$ in $B_1^+:=B_1\cap\{x_n>0\}$, with $\varphi_k(B_{\rho_0}\cap \partial{B}_{\frac{r}{r_k}}(\frac{y}{r_k}))=B_1\cap\{x_n=0\}$
and 
$\varphi_k(0)=0$. We choose $\varphi_k$ with uniformly bounded $C^2$ norms. Then, we define
$$
\tilde{w}_k(x)=w_k(\varphi_k^{-1}(x)),\quad \tilde{h}_k(x)=h_{k}(\varphi_{k}^{-1}(x))\quad \text{ for }x\in B_{1}^+.
$$ 
Reasoning as in Case a), we get that $\tilde{w}_k$  satisfy   in the viscosity sense, the following set of inequalities
\begin{equation}\label{condiSS}
\begin{split}
\left\{\begin{array}{l}
 \mathcal{M}_{\lambda_0,\Lambda_0}^+(D^2\tilde{w}_k)+K|\nabla\tilde{w}_k|\geq 0\quad \text{ in }B_1^+,\\
\mathcal{M}_{\lambda_0,\Lambda_0}^-(D^2\tilde{w}_k)-K|\nabla\tilde{w}_k|\leq 0\quad \text{ in }B_1^+,
\end{array}
\right.
\end{split}
\end{equation}
where $K$ is a fixed constant depending only on the uniform bound of the $C^2$ norms of $\varphi_k.$ 
 With similar arguments we  obtain that 
$\tilde{h}_k$ satisfy in the viscosity sense  the inequalities in \eqref{condiSS} in $B_1^+$,  as well.

We also notice that, since $h_k(x)\to cx_n^+$ uniformly on compact sets of $\mathbb{R}^n,$ with $c>0,$ then, 
$\tilde{h}_k(\frac{1}{2}e_n)=h_k(\varphi_k^{-1}(\frac{1}{2}e_n))\to \tilde{c}>0$.

Hence, we can apply   Proposition 2.4 of \cite{SS} and we get
\begin{equation}\label{bound-hk}
\tilde{h}_k(x)\leq C\tilde{h}_k(\frac{1}{2}e_n)x_n\leq C_0x_n \quad\text{ in }B_{1/2}^+, 
\end{equation}
for a positive constant $C_0$ and large $k.$ 

On the other hand,  let $w_0$
satisfying
\begin{equation*}
\begin{split}
\left\{\begin{array}{l}
\mathcal{M}_{\lambda_0,\Lambda_0}^-(D^2 w_0)=0\quad \text{ in } B_1^+,\\
w_0=\frac{\delta}{4}\varphi\quad \text{ on }  B_{\varepsilon}(z^0)\cap \partial B_1^+,\\
w_0=0\quad  \text{ on }\partial B_1^+\setminus B_\varepsilon (z^0),
\end{array}
\right.
\end{split}
\end{equation*}
with $\varphi$ as in \eqref{bound-data}. 
Then $w_k\geq w_0$ in $B_1^+.$ We remark that $w_0>0$ in $B_1^+.$ Thus, for any $0<r_0<1$, 
\begin{equation}\label{finalc1}
\tilde{w}_k(\frac{r_0}{2}e_n)=w_k(\varphi^{-1}_k(\frac{r_0}{2}e_n))\geq w_0(\varphi^{-1}_k(\frac{r_0}{2}e_n))\to\tilde{c}_{r_0}>0.
\end{equation}
Now the application of Proposition 2.5 in \cite{SS} to $\tilde{w}_k^{r_0}(x):=\tilde{w}_k(r_0x)$, for $r_0>0$ universal and small, gives 
\begin{equation*}\label{finalc0}
\tilde{w}_k(x)\geq c_0\tilde{w}_k(\frac{r_0}{2}e_n)x_n\quad \mbox{ in }B^+_{\frac{r_0}{2}},
\end{equation*}
for a positive constant $c_0$.
 Hence, using    \eqref{finalc1} with this choice of $r_0$, we get 
\begin{equation}\label{finalc2}
\tilde{w}_k(x)\geq c_1 x_n\quad \text{ in }B^+_{\frac{r_0}{2}},
\end{equation}
for $c_1$ a positive constant and large $k$.
Thus, from \eqref{bound-hk} and \eqref{finalc2}, we obtain
$$
\frac{\tilde{w}_k(x)}{\tilde{h}_k(x)}\geq \frac{c_1}{C_0}:=\mu\quad \text{ in }B_{\rho_1}^+,
$$
 for $\rho_1>0$ small and large $k.$ Now, going back to the original variables, we conclude that
$$
\frac{w_k(x)}{h_k(x)}\geq \mu \quad \text{ in }B_{\tilde{\rho}_0}\cap \overline{B}_{\frac{r}{r_k}}^c(\frac{y}{r_k}),
$$
for some constants $\tilde{\rho}_0>0$ and $\mu>0,$ and large $k.$ That is, \eqref{bound-mu} holds.

Finally, since 
$$
\beta_{j_k}h_k-u_k\geq w_k \quad \text{ in   }B_{1}\cap \overline{B}_{\frac{r}{r_k}}^c(\frac{y}{r_k}),
$$
 we get
$$
\beta_{j_k}h_k-u_k\geq w_k=\frac{w_k}{h_k}h_k\geq \mu h_k\quad\text{ in }B_{\tilde{\rho}_0}\cap \overline{B}_{\frac{r}{r_k}}^c(\frac{y}{r_k}).
$$
 As a consequence,
$$
(\beta_{j_k}-\mu)h_k-u_k\geq 0\quad \text{ in }B_{\tilde{\rho}_0}\cap \overline{B}_{\frac{r}{r_k}}^c(\frac{y}{r_k}).
$$
 Then, in the original variables,  we have
$$
(\beta_{j_k}-\mu)h(r_kx)-u(r_kx)\geq 0\quad \text{ when }|x|\leq \tilde{\rho}_0, \ |x-\frac{y}{r_k}|>\frac{r}{r_k},
$$
or, equivalently, when
$|r_kx|\leq r_k \tilde{\rho}_0,$ $|r_kx-y|>r$.
Since $\beta_{j_k}-\mu\to\tilde{\beta}-\mu$, there holds that 
   $\beta_{j_k}-\mu\leq \tilde{\beta}-\frac{\mu}{2},$  if $k$ is large enough. Hence, 
$$
(\tilde{\beta}-\mu/2)h-u\geq 0\quad \text{ in  }B_{r_{k_0}\tilde{\rho}_0}\cap B_{r}^c(y),
$$
 for some suitable $k_0$. As a consequence, if $2^{-k}\leq r_{k_0}\tilde{\rho}_0$,
$$
(\beta_{k}-\mu/2)h-u\geq (\tilde{\beta}-\mu/2)h-u\geq 0 \quad \text{ in  }B_{2^{-k}}\cap B_{r}^c(y),
$$
 but this contradicts the definition of $\beta_k$ and completes the proof. 
\end{proof}

\smallskip
\begin{rem}\label{extens-asympt} Lemma \ref{lemma51} also holds if we replace in the statement the $p_0$-Laplace operator by
a general class of fully nonlinear degenerate elliptic operators. More precisely, we can consider $u$ a Lipschitz viscosity solution of an equation
of the form 
\begin{equation*}
F(D^2 u (x),D u(x),x)=0 \quad \text{ in } \Omega,
\end{equation*} 
with ${F}$ satisfying for every $M\in\mathcal{S}^{n\times n}$,  $q\in \mathbb{R}^n$ and $x\in \Omega$,
\begin{equation*}
|q|^{\sigma}\mathcal{M}_{\lambda,\Lambda}^-(M)\le{F}(M,q,x)\le
 |q|^{\sigma}\mathcal{M}_{\lambda,\Lambda}^+(M),
\end{equation*}
for some $0<\lambda\le\Lambda$ and $\sigma\in \R$, and the same proof applies.
\end{rem}

\section{Regularity of the free boundary}\label{section7}

In this section we prove our main result, namely, Theorem \ref{Lipmain}.

Since we will apply a result of \cite{LN1}, we include first the definition of viscosity solution employed in that paper in case of nonnegative solutions. These are solutions of problem \eqref{fb} with $p(x)\equiv p_0$, $f\equiv 0$ and $g\equiv 1$.

\begin{defn}[Definition 1.4 in \cite{LN1}]\label{def-LN} Let $D\subset\R^n$ be a domain, $u\in C(D)$ be nonnegative and $1<p_0<\infty$. $u$ is a viscosity (or weak) solution of
\begin{equation}  \label{fb-LN}
\left\{
\begin{array}{ll}
\Delta_{p_0} u = 0 & \hbox{in $D^+(u):= \{x \in D : u(x)>0\}$}, \\
\  &  \\
|\nabla u|= 1 & \hbox{on $F(u):= \partial D^+(u) \cap
D,$} 
\end{array}
\right.
\end{equation}
if there holds that $u$ is $p_0$-harmonic in $D^+(u)$, in the sense that $u\in W^{1,p_0}(D^+(u))$ and 
$$
\int_{D^+(u)} |\nabla u|^{p_0-2}\nabla u \cdot \nabla
\varphi\, dx =0 \quad \text{for every } \varphi \in W_0^{1,p_0}(D^+(u)),
$$
and the free boundary condition in \eqref{fb-LN} is satisfied in the following sense. Assume that $x_0\in F(u)$ and there exists a ball
$B_r(y)\subset D$, with $x_0\in \partial B_r(y)$. If $\nu=\frac{y-x_0}{|y-x_0|}$, then the following holds, as $x\to x_0$ non-tangentially, for $\alpha =1$,
\begin{itemize}
\item[(i)] if $B_r(y)\subset D^+(u)$, then $u(x)= \alpha \langle x-x_0,\nu\rangle^+ +o(|x-x_0|)$,

\item[(ii)] if $B_r(y)\subset D^+(u)^c$, then $u(x)= \alpha \langle x_0-x,\nu\rangle^++o(|x-x_0|)$.
\end{itemize}
\end{defn}

\medskip

We next extend the result of Lemma 6.2 in \cite{DFS1} to the global homogenous $p_0$-Laplacian free boundary problem (i.e., to problem \eqref{fb} in $\Omega=\R^n$ with $p(x)\equiv p_0$, $f\equiv 0$ and $g\equiv 1$). This result is valid for globally Lipschitz continuous functions. The notion of viscosity solution we employ in Lemma \ref{fbliminftylinlemma} is the one in \cite{LN1} (see Definition \ref{def-LN} above). 
\begin{lem}\label{fbliminftylinlemma}
Let $1<p_0<\infty$.  Let  $v:\mathbb{R}^n\to\mathbb{R}$ be  a nonnegative Lipschitz viscosity solution (in the sense of Definition \ref{def-LN}) to 
\begin{equation}  \label{fbliminftylin}
\left\{
\begin{array}{ll}
\Delta_{p_0} v = 0, & \hbox{in $\{v>0\}$}, \\
\  &  \\
|\nabla v|= 1, &\mbox{on}\quad F(v):=\partial\{v>0\}.
\end{array}
\right.
\end{equation}
Assume that 
$$
\{v>0\}=\{(x',x_n)\in\mathbb{R}^n:\quad x'\in\mathbb{R}^{n-1},\quad x_n>h(x')\},
$$
with $h$ a Lipschitz continuous function, $h(0)=0$ and $\mbox{Lip}(h)\leq M.$ Then $h$ is linear and, after a rotation,
$$
v(x)=x_n^+.
$$
\end{lem}
\begin{proof} We will denote $B'_r$ the ball of radius $r$ centered at $0$ in $\mathbb{R}^{n-1}$.

We follow the idea of the proof in Lemma 6.2 in \cite{DFS1},  coupled with results about the regularity of the free boundary in the homogeneous two phase problem associated with the $p_0$-Laplace operator. In fact, if $v$ is a viscosity solution of \eqref{fbliminftylin} and its free boundary is a Lipschitz graph, from the regularity results in \cite{LN1}, we know that the free boundary $F(v)$ is $C^{1,\alpha}$ in $B_1,$  with a bound $C$ depending  only on $n, p_0$ and on the Lipschitz constant $M$ of $h.$  Then, 
\begin{equation}\label{calpha}
|h(x')-h(0)-\langle \nabla h(0),x'\rangle|\leq C|x'|^{1+\alpha}
\end{equation}
in $B'_1,$ where $C=C(n,p_0,M).$ Moreover, since $v$ is a global solution to \eqref{fbliminftylin}, considering the rescaled function $v_R(x)=\frac{v(Rx)}{R},$ we still obtain a solution to problem \eqref{fbliminftylin}  whose free boundary is the graph of the function $h_R(x')=\frac{h(Rx')}{R}$ for $x'\in \mathbb{R}^{n-1}.$ This function preserves the same Lipschitz constant  and then satisfies the inequality \eqref{calpha}. That is,
$$
|h_R(x')-h_R(0)-\langle \nabla h_R(0),x'\rangle|\leq C|x'|^{1+\alpha}
$$
for $x'\in B'_1. $ This fact can be read as
$$
|h(Rx')-h(0)-\langle \nabla h(0),Rx'\rangle|\leq CR|x'|^{1+\alpha}
$$
for $x'\in B'_1.$ Then,
$$
 |h(y')-h(0)-\langle \nabla h(0),y'\rangle|\leq C\frac{|y'|^{1+\alpha}}{R^\alpha}
$$
for $y' \in B'_R.$ Hence, passing to the limit $R\to \infty,$ we conclude 
 that $h$ is linear in $\mathbb{R}^{n-1}$. Since $v$ is Lipschitz, then Lemma \ref{lemma7.3} in Appendix \ref{app-liouv} applies and, up to a proper rotation, $v(x)=x_n^+.$ \end{proof}

For the sake of completeness we recall the following theorem we proved in \cite{FL}

\begin{thm}[Theorem 1.1 in \cite{FL}]
\label{flatmain1} Let
$u$ be a viscosity solution to \eqref{fb}
in $B_1$. Assume that  $0\in F(u),$ $g(0)=1$ and $p(0)=p_0.$   
There exists a universal constant $\bar{\varepsilon}>0$ such that, if the graph of $u$ is $\bar{\varepsilon}-$flat in $B_1,$  in the direction $e_n,$ that is
\begin{equation*}  
(x_n-\bar{\varepsilon})^+\leq u(x)\leq (x_n+\bar{\varepsilon})^+, \quad x\in B_1,
\end{equation*}
and
\begin{equation}  \label{pflat}
\|\nabla p\|_{L^{\infty}(B_1)}\leq \bar{\varepsilon},\quad \|f\|_{L^{\infty}(B_1)}\leq \bar{\varepsilon}, \quad [g]_{C^{0,\beta}(B_1)}\leq \bar{\varepsilon},
\end{equation}
 then $F(u)$ is $C^{1,\alpha}$ in $
B_{1/2}$.

The constants $\bar{\varepsilon}$ and $\alpha$ depend only on $p_{\min}$, $p_{\max}$, $\beta$  and $n$. 
\end{thm}

In the proof of Theorem \ref{Lipmain} we will also use

\begin{prop}\label{casealphageq1}
Let $u_k$ be a sequence of viscosity solutions to \eqref{fb} in $B_2$, with right hand side $f_k$, exponent $p_k$ and free boundary condition $g_k$, where $f_k$, $p_k$ and $g_k$ are as in Subsection \ref{assump}. Assume that $u_k$ are uniformly Lipschitz and that, for some $\alpha>0$ and $\nu\in\R^n$ with 
$|\nu|=1$,  $u_k\to u_0(x)=\alpha \langle x,\nu\rangle^+$,  $f_k\to 0$, $p_k\to p_0$, $\nabla p_k\to 0$ and $g_k\to 1$ uniformly in $B_2$. Assume moreover that 
$F(u_k)$ are uniform Lipschitz graphs and $F(u_k)\to F(u_0)$  in Hausdorff distance in $B_2$. 
Then $\alpha\geq 1$. 
\end{prop}
\begin{proof} Without loss of generality we assume that $\nu=e_n$.
Suppose by contradiction  that $0<\alpha<1$. We take $\varphi\in C^{\infty}(\mathbb{R}^n),$ with 
$0\leq \varphi\leq 1,$ $\varphi\equiv 0$ in ${B^c_{1/2}}$ and $\varphi\equiv 1$ in $B_{1/4}.$ For $0<\xi<1/4$ depending on $\alpha$, to be fixed later, and $0<\varepsilon<1,$ we define
$$
D_\varepsilon=D_\varepsilon^\xi=B_1\cap\{x_n>-\xi+\varepsilon\varphi (x)\}.
$$

Let $\lambda_0,\Lambda_0$ as in  \eqref{p(x)-vs-pucci}. For $\rho>0$ fixed and depending on $\alpha,$ to be precised later, we consider $v_\varepsilon$ such that
\begin{equation*}
\begin{split}
\left\{
\begin{array}{l}
\mathcal{M}^{+}_{\lambda_0,\Lambda_0}(D^2v_\varepsilon)=-\rho,\quad\mbox{in}\:\: D_\varepsilon\\
v_\varepsilon=\alpha (x_n+\xi),\quad\mbox{on}\:\:\partial B_1\cap\{x_n\geq -\xi\},\\
v_\varepsilon=0,\quad\mbox{on}\:\:B_1\cap\{x_n= -\xi+\varepsilon\varphi(x)\},
\end{array}
\right.
\end{split}
\end{equation*}
$v_\varepsilon\equiv0$ on $\overline{B_1}\setminus \overline{D_\varepsilon}$.

\bigskip

{\it Step I. } We will show that, in $B_{3/4}$, $v_\varepsilon$ is a strict supersolution to problem \eqref{fb} with   right hand side $f_k$, exponent $p_k$ and free boundary condition $g_k$, for  $\rho$, $\varepsilon$ and $\xi$ suitably chosen and large $k$.

\smallskip

We first observe that $v_\varepsilon>0$ in $D_\varepsilon.$ In addition, $v_\varepsilon\in C^{2,{\tilde{\alpha}}}(D_\varepsilon)$ and $v_\varepsilon\in C^{1,{\tilde{\alpha}}}(\overline{D_\varepsilon\cap B_{3/4}})$. We define 
$$
w_\varepsilon=v_\varepsilon-\alpha(x_n+\xi),
$$
that satisfies
\begin{equation*}
\begin{split}
\left\{
\begin{array}{l}
\mathcal{M}^{+}_{\lambda_0,\Lambda_0}(D^2w_\varepsilon)=-\rho,\quad\mbox{in}\:\: D_\varepsilon\\
w_\varepsilon=0,\quad\mbox{on}\:\:\partial B_1\cap\{x_n\geq -\xi\},\\
w_\varepsilon=-\alpha (x_n+\xi)=-\alpha\varepsilon\varphi(x),\quad\mbox{on}\:\: B_1\cap\{x_n= -\xi+\varepsilon\varphi(x)\}.\\
\end{array}
\right.
\end{split}
\end{equation*}
Hence, using ABP estimate (Theorem 3.6 in \cite{CC}), we obtain that $||w_\varepsilon||_{L^\infty(D_\varepsilon)}\leq C_0(\rho+\varepsilon)$,
 with $C_0>0$  independent of $\varepsilon$ and $\xi$ universal.
Then from the inner estimates in Corollary 5.7, \cite{CC} and the boundary estimates in Theorem 1.4, \cite{SS} we deduce that there exist positive constants $C_1$ and ${C_2}$ such that
\begin{equation}\label{wep-c1alpha}
\begin{split}
||w_\varepsilon||_{C^{1,{\tilde{\alpha}}}(\overline{D_\varepsilon\cap B_{3/4}})}&\leq C_1\left(||w_\varepsilon||_{{L^\infty}(D_\varepsilon)}+\rho+\alpha\varepsilon||\varphi||_{C^{1,{\tilde{\alpha}}}(B_1)}\right)\leq {C_2}(\rho+\varepsilon),
\end{split}
\end{equation}
 where $C_1$ and ${C_2}$ depend on  the maximal curvature of $\{x_n=-\xi+\varepsilon \varphi(x)\},$ see \cite{SS}, and can be chosen universal independent
of $\varepsilon$ and $\xi$. Then
$$
||\nabla w_\varepsilon||_{L^\infty(\overline{D_\varepsilon\cap B_{3/4}})}\leq {C_2}(\rho+\varepsilon)
$$
and
$$
||\nabla v_\varepsilon|-\alpha|\leq |\nabla v_\varepsilon-\alpha e_n|\leq {C_2}(\rho+\varepsilon) \qquad\mbox{ in } \overline{D_\varepsilon\cap B_{3/4}}.
$$
We now fix
\begin{equation}\label{choice-const}
\begin{array}{l}
\rho=\frac{c(\alpha)}{{C_2}},\qquad 0<\varepsilon\le \min\{\frac{c(\alpha)}{{C_2}},\frac{1}{4}\}, \qquad 2\xi=\min\{\frac{c(\alpha)}{{C_2}},\frac{1}{4}\},\\
\\
\text{with } \quad c(\alpha)=\frac{1}{2}\min\{\frac{1-\alpha}{2}, \frac{\alpha}{2}\} \quad \text{and } C_2 \quad \text{as in} \quad \eqref{wep-c1alpha},\\
\end{array}
\end{equation}
and then, 
\begin{equation}\label{bound-vep}
\frac{\alpha}{2}\le|\nabla v_\varepsilon|\le\frac{1+\alpha}{2}\qquad\mbox{ in } \overline{D_\varepsilon\cap B_{3/4}}.
\end{equation}
Recalling \eqref{p(x)-vs-pucci}, we obtain,  in $D_\varepsilon\cap B_{3/4},$
$$
\Delta_{p_k(x)}v_\varepsilon\leq  I+II
$$
where
\begin{equation*}
\begin{split}
I= |\nabla v_\varepsilon|^{p_k(x)-2}\mathcal{M}^+_{\lambda_0,\Lambda_0}(D^2v_\varepsilon)
\end{split}
\end{equation*}
and
\begin{equation*}
\begin{split}
II= |\nabla v_\varepsilon|^{p_k(x)-2}\langle \nabla p_k(x),\nabla v_\varepsilon\rangle\log |\nabla v_\varepsilon|.
\end{split}
\end{equation*}
Then
\begin{equation*}
\begin{split}
I= |\nabla v_\varepsilon|^{p_k(x)-2}(-\rho)\leq -c_1\rho,
\end{split}
\end{equation*}
since $ |\nabla v_\varepsilon|^{p_k(x)-2}\geq c_1=c_1(\alpha, p_{\max})$ because of \eqref{bound-vep}.
Moreover,
\begin{equation*}
\begin{split}
II\leq |\nabla v_\varepsilon|^{p_k(x)-1}| \nabla p_k(x)|\left|\log |\nabla v_\varepsilon|\right|\leq |\nabla p_k(x)| c_2,
\end{split}
\end{equation*}
where we have used that $|t|^{p_k(x)-1}\left|\log |t|\right|\leq c_2=c_2(p_{\min})$ for $|t|\leq 1$ and \eqref{bound-vep}.
Then,
\begin{equation}\label{bound-eq-vep}
\begin{split}
\Delta_{p_k(x)}v_\varepsilon\leq -c_1\rho+ {||\nabla p_k||}_{L^\infty}c_2\leq -c_1\rho+c_2\frac{c_1}{c_2}\frac{\rho}{2}=-\frac{c_1}{2}\rho
\end{split}
\end{equation}
if $k$ is large so that ${||\nabla p_k||}_{L^\infty}\leq \frac{c_1}{c_2}\frac{\rho}{2}.$

On the other hand, for $k$ large, we have ${||f_k||}_{L^\infty}<\frac{c_1}{2}\rho$  and ${g_k(x)}>\frac{1+\alpha}{2}$ in $B_1$ and then,
\begin{equation}\label{vep-supers}
\left\{\begin{array}{l}
\Delta_{p_k(x)}v_\varepsilon\leq -\frac{c_1}{2}\rho<-{||f_k||}_{L^\infty}\leq f_k,\quad \mbox{in}\:\:D_\varepsilon\cap B_{3/4}\\
|\nabla v_\varepsilon|\geq \frac{\alpha}{2},\quad \mbox{in}\:\:\overline{D_\varepsilon\cap B_{3/4}}\quad (\mbox{this implies}\:\:\:\nabla v_\varepsilon\not=0)\\
|\nabla v_\varepsilon|\le\frac{1+\alpha}{2}<g_k,\quad \mbox{in}\:\:\overline{D_\varepsilon\cap B_{3/4}}.
\end{array}
\right.
\end{equation}
Hence, for our choice of $\rho$, $\varepsilon$ and $\xi$ done in \eqref{choice-const}, $v_\varepsilon$ is a strict supersolution  to problem \eqref{fb} in  $B_{3/4}$ with   right hand side $f_k$, exponent $p_k$ and free boundary condition $g_k$, for   large $k$, as claimed.

\bigskip

{\it Step II. } We will now get some uniform bounds for the functions $u_k$. In fact, let $v$ be such that
\begin{equation*}
\begin{split}
\left\{
\begin{array}{l}
\mathcal{M}^{+}_{\lambda_0,\Lambda_0}(D^2v)=-\rho,\quad\mbox{in}\:\: \{x_n>-{\xi}, \ 5/8<|x|<1\},  \\
v=\alpha (x_n+\xi),\quad\mbox{on}\:\:\partial B_1\cap\{x_n\geq -\xi\},\\
v=0,\quad\mbox{on}\:\:\{x_n= -\xi, \ 5/8<|x|<1 \},\\
v=0,\quad\mbox{on}\:\:\partial B_{5/8}\cap\{x_n\geq -\xi\}.\\
\end{array}
\right.
\end{split}
\end{equation*}
Then, $v_{\varepsilon}>v>0$ in $\{x_n>-{\xi}, \ 5/8<|x|<1\}$. Moreover, there exists $c_3>0$ universal such that 
\begin{equation}\label{bound-v}
v>c_3\quad\text{ in }\{x_n\ge-\frac{\xi}{2}, \ 3/4\le|x|\le 1\}.
\end{equation}

Let us now fix $\delta$ universal such that
\begin{equation}\label{choice-delta-1}
0<\delta<\frac{\xi}{2} \quad \text{ and } \quad 2L\delta <c_3,
\end{equation} 
where $L$ is the uniform Lipschitz constant of the functions $u_k$.

We will first show that, if $k$ is large,
\begin{equation}\label{uk0}
u_k=0\qquad\text{ in } \overline{B_1}\cap\{x_n\le -\delta\},
\end{equation}
\begin{equation}\label{ukleCdelta}
u_k\le 2L\delta\qquad\text{ in } \overline{B_1}\cap\{|x_n|\le\delta\},
\end{equation}
\begin{equation}\label{gradukge}
\frac{\alpha}{2}\le|\nabla u_k|\le L\qquad\text{ in } \overline{B_1}\cap\{x_n\ge\delta\}.
\end{equation}

In fact, since $\partial \{u_k>0\}\to\partial \{u_0>0\}=\{x_n=0\}$ in the Hausdorff distance in $B_2$, then 
$\partial\{u_k>0\}\subset\{|x_n|<\delta\}$  in $B_2$, if $k$ is large. Hence \eqref{uk0} and  \eqref{ukleCdelta} follow.

Since $u_k\to u_0=\alpha x_n^+$ uniformly in $B_2$, then for large $k$,
\begin{equation}\label{equat-xndelta}
\left\{\begin{array}{l}
u_k\ge \alpha\frac{\delta}{4}\quad\text{in }B_2\cap\{x_n >\frac{\delta}{2}\},\\
\Delta_{p_k(x)}u_k =f_k \quad\text{in }B_2\cap\{x_n >\frac{\delta}{2}\},\\
\end{array}
\right.
\end{equation}
and then, by the $C^{1,\bar{\alpha}}$ estimates (Theorem 1.1 in \cite{Fan}),
\begin{equation*}
\nabla u_k\to \alpha e_n\quad \text{ uniformly in } \overline{B_1}\cap\{x_n\ge{\delta}\},
\end{equation*}
which gives \eqref{gradukge} for $k$ large.

We now observe that $v-u_0\geq \frac{\alpha}{2}\xi$ on  $\partial B_1\cap \{x_n\geq -\frac{\xi}{2}\}$ and
\begin{equation*}
v-u_0\geq \frac{\alpha\xi}{4}\quad\text{ in  }\{x_n\ge -\frac{\xi}{2}, \ 1-\sigma\le|x|\le 1\},
\end{equation*}
for some universal $0<\sigma<1/4$. Here we have used that $v\in C^{\bar{\alpha}}(\{x_n\ge-{\xi}, \ 5/8\le|x|\le 1\})$ (see, for
instance, Theorem 2 in \cite{Si}). Then, 
\begin{equation*}
v-u_k\geq \frac{\alpha\xi}{8}\quad\text{ in  }\{x_n\ge-\frac{\xi}{2}, \ 1-\sigma\le|x|\le 1\},
\end{equation*}
for large $k$. Recalling \eqref{bound-v},  \eqref{choice-delta-1} and \eqref{ukleCdelta}, we obtain
\begin{equation}\label{ukleCdelta-b}
v_{\varepsilon} > u_k\qquad\text{ in } \{|x_n|\le\delta,  \ 3/4 \le |x|\le 1 \},
\end{equation}
\begin{equation}\label{bound-choice-delta-c}
v_{\varepsilon}-u_k\geq \frac{\alpha\xi}{8}\quad\text{ in  }\{x_n\ge -\frac{\xi}{2}, \ 1-\sigma\le |x|\le 1\}.
\end{equation}

\bigskip

{\it Step III. } We will show that
\begin{equation}\label{compare-vep-uk}
v_{\varepsilon}\ge u_k\quad \text {in }\overline{B_1}, \quad \text{ for every }0<\varepsilon<\frac{\xi}{2},
\end{equation}  
if  $k$ is large.

If the result is not true, then
\begin{equation*}
\max_{\overline{B_1}} (u_k-v_{\varepsilon})= (u_k-v_{\varepsilon})(\tilde{x}_k)>0\quad \text{ for some }\tilde{x}_k\in \overline{B_1}.
\end{equation*} 
If $|\tilde{x}_k|\ge 3/4$, then \eqref{uk0}, \eqref{ukleCdelta-b} and \eqref{bound-choice-delta-c} imply that
$$\tilde{x}_k\in \{x_n >\delta,  \ 3/4 \le |x|\le 1-\sigma \}.$$
{}From \eqref{gradukge} and \eqref{equat-xndelta} we get
\begin{equation*}
\frac{\alpha}{2}\le|\nabla v_{\varepsilon}(\tilde{x}_k)|\le L
\end{equation*}
and 
\begin{equation}\label{equat-vep}
\Delta_{p_k(\tilde{x}_k)}v_{\varepsilon}(\tilde{x}_k) \ge f_k(\tilde{x}_k).
\end{equation}
Now, the uniform $C^{1, \bar{\alpha}}$ estimates for $v_{\varepsilon}$ in $\overline{B_1}\cap\{x_n\ge 0 \}$ give 
\begin{equation*}
\frac{\alpha}{4}\le|\nabla v_{\varepsilon}|\le 2L \quad\text{ in } B_{\mu}(\tilde{x}_k),
\end{equation*}
for some $\mu>0$ universal. Then, proceeding as in the computations leading to \eqref{bound-eq-vep}, we get
\begin{equation*}
\Delta_{p_k(x)}v_{\varepsilon}\le -\bar{c}\rho\quad\text{ in } B_{\mu}(\tilde{x}_k),
\end{equation*}
with $\bar{c}>0$ universal, if $k$ is large. Therefore,
\begin{equation*}
\Delta_{p_k(x)}v_{\varepsilon}<f_k \quad\text{ in } B_{\mu}(\tilde{x}_k),
\end{equation*}
for large $k$, which contradicts \eqref{equat-vep}. Then $\tilde{x}_k\in B_{3/4}$.

Since $\varepsilon<\frac{\xi}{2}$ and $\delta<\frac{\xi}{2}$, we have
$$\partial\{u_k>0\}\subset\{|x_n|<\delta\}\subseteq\{x_n>-\xi+\varepsilon \varphi(x)\},$$
 and then $\tilde{x}_k\in\{u_k>0\}\cap B_{3/4}\subset D_{\varepsilon}\cap B_{3/4}$.

So \eqref{vep-supers} implies that, for large $k$,
\begin{equation*}
\nabla v_{\varepsilon}(\tilde{x}_k)\neq 0
\end{equation*}
and 
\begin{equation*}
f_k(\tilde{x}_k)>\Delta_{p_k(\tilde{x}_k)}v_{\varepsilon}(\tilde{x}_k)\ge f_k(\tilde{x}_k),
\end{equation*}
a contradiction. This shows \eqref{compare-vep-uk}.

\bigskip

{\it Step IV. } We will finally show that, for some $\varepsilon_k>0$, we have $v_{\varepsilon_k}\ge u_k$ in $B_{3/4}$ and 
$F(u_k)\cap F(v_{\varepsilon_k})\cap B_{3/4}\neq\emptyset$, 
if  $k$ is large enough. This will contradict that $v_{\varepsilon}$ is a strict supersolution to  problem \eqref{fb} in  $B_{3/4}$ and
concludes the proof.

\smallskip

In fact, from \eqref{compare-vep-uk} we know that 
$$
v_\varepsilon\ge u_k\quad \mbox{in}\:\:\overline{\{u_k>0\}}\quad \text{ for } 0<\varepsilon<\frac{\xi}{2}.
$$
Let 
$$
\varepsilon_k=\sup\left\{\varepsilon > 0:\:\:v_\varepsilon \ge u_k\:\:\:\mbox{in}\:\:\overline{\{u_k>0\}}\right\}.
$$
Since $\xi\le 1/4$, if we consider $\varepsilon=2\xi$, then $B_\xi\subset \{x_n<-\xi+\varepsilon\varphi(x)\}$ because, in $B_{1/4}$, $x_n=-\xi+\varepsilon\varphi(x)=-\xi+2\xi=\xi.$

Moreover, $0\in \partial\{u_0>0\}$ and $\partial\{u_k>0\}\to \partial \{u_0>0\}$ in the Hausdorff distance, then for $k$ large, there exist $\hat{x}_k\in B_\xi\cap\partial\{u_k>0\}$  and $\bar{x}_k\in B_\xi$ such that $u_k(\bar{x}_k)>0=v_{\varepsilon}(\bar{x}_k)$,
with $\bar{x}_k\in\overline{\{u_k>0\}}$. Then, $0<\varepsilon_k <2\xi$. 

Therefore, there holds
$v_{\varepsilon_k}\geq u_k$ in 
$\overline{\{u_k>0\}}$ and then,
$$
v_{\varepsilon_k}\geq u_k\quad \text{ in }\overline{B_{1}},
$$ 
$$
v_{\varepsilon_k}(x_k)=u_k(x_k)\quad \text{ for some }x_k\in \overline{\{u_k>0\}}.
$$

Proceeding exactly as in {\it Step III} we obtain that 
$$x_k\in \overline{\{u_k>0\}}\cap B_{3/4}.$$

If $x_k\in \{u_k>0\}\cap B_{3/4},$ then  $v_{\varepsilon_k}(x_k)=u_k(x_k)>0.$ Since $v_{\varepsilon_k}\ge u_k>0$ in a neighborhood of $x_k$, this produces a contradiction because $v_{\varepsilon_k}$ is a strict supersolution  to  problem \eqref{fb} in $B_{3/4}$.

As a consequence $x_k\in \partial\{u_k>0\}\cap B_{3/4}$ and $v_{\varepsilon_k}(x_k)=u_k(x_k)=0$ and there exist $x_{k_j}\to x_k$ such that $u_k(x_{k_j})>0$. Then $v_{\varepsilon_k}(x_{k_j})\geq u_{k}(x_{k_j})>0$ and therefore $x_k\in \partial \{v_{\varepsilon_k}>0\}.$
 Hence $x_k\in F(u_k)\cap F(v_{\varepsilon_k})\cap B_{3/4},$ which gives a contradiction again.  This shows that $\alpha\geq 1$ and completes the proof.
\end{proof}

We will also need

\begin{prop}\label{alpha-leq-1}
Let $u_k$ be a sequence of viscosity solutions to \eqref{fb} in $B_2$, with right hand side $f_k$, exponent $p_k$ and free boundary condition $g_k$, where $f_k$, $p_k$ and $g_k$ are as in Subsection \ref{assump}. Assume that, for some $\alpha\ge 0$ and $\nu\in\R^n$ with 
$|\nu|=1$,  $u_k\to u_0(x)=\alpha \langle x,\nu\rangle^+$,  $f_k\to 0$, $p_k\to p_0$, $\nabla p_k\to 0$ and $g_k\to 1$ uniformly in $B_2$. 
Then $\alpha\le 1$. 
\end{prop}
\begin{proof} Without loss of generality we assume that $\nu=e_n$.
 Suppose by contradiction that $\alpha=1+\eta,$ with $\eta>0$, then
$$
u_0(x)= (1+\eta)x_n^+.
$$
For $\delta>0$ and $\ep>0$  small, to be precised later, we define
\begin{equation*}
{Q}(x): =(1+\frac{\eta}{2}) x_n +\delta x_n^2-\ep|x'|^2,
\end{equation*}
where we denote $x=(x',x_n)$, $x'\in \R^n$.

Let us show that 
\begin{equation}\label{u0smallerQ}
\left\{
\begin{aligned}
&u_0 > Q \quad \text{in } B_{\rho_0}\setminus\{0\},\\
&u_0(0)=Q(0),
\end{aligned}
\right.
\end{equation}
for some $\rho_0=\rho_0(\delta,\eta)>0$.
In fact, there holds that
\begin{equation*}
\begin{aligned}
&u_0(x)=(1+\eta)x_n^+>(1+\frac{\eta}{2}) x_n +\delta x_n^2\ge Q(x) \qquad \text{ for } \,0<|x_n|<\frac{\eta}{2\delta},\\
&u_0(x',0)=0>-\ep|x'|^2=Q(x',0)\qquad \text{ for } \, x'\neq 0,
\end{aligned}
\end{equation*}
so \eqref{u0smallerQ} follows for $\rho_0=\min\{1,\frac{\eta}{2\delta}\}$.

{\bf Claim.} We claim that, in $B_1$, ${Q}$ is a strict subsolution to problem \eqref{fb} with   right hand side $f_k$, exponent $p_k$ and free boundary condition $g_k$, for large $k$.

Indeed, we have
$$\nabla Q= (1+\frac{\eta}{2}) e_n +2Mx,\qquad D^2Q=2M,$$
where $M\in R^{n\times n}$ is given by
\begin{equation}\label{def-M}
M_{ij}=0 \, \text{ for } \, i\neq j\qquad M_{ii}=-\ep \, \text{ for } \, i\neq n, \qquad M_{nn}=\delta.
\end{equation}
Then,
\begin{equation}\label{bound-nabla-Q}
1+\frac{\eta}{4}\le|\nabla Q|\le 1+{\eta}\quad \text{ in } B_1,
\end{equation}
if $\delta\le {\eta}/{8}$ and $\ep\le {\eta}/{8}$.

 Moreover, applying the lower bound in \eqref{p(x)-vs-pucci}, we obtain
\begin{equation}\label{p(x)-Q}
\begin{split}
&\Delta_{p_k(x)}{Q}(x)\geq |\nabla Q|^{p_k(x)-2}\mathcal{M}_{\lambda_0,\Lambda_0}^-(D^2 Q)+|\nabla Q|^{p_k(x)-2}\langle \nabla Q,\nabla p_k(x)\rangle\log|\nabla Q|\\
&\ge |\nabla Q|^{p_k(x)-2} 2\mathcal{M}_{\lambda_0,\Lambda_0}^-(M) - |\nabla p_k(x)||\nabla Q|^{p_k(x)-1}\log|\nabla Q|.
\end{split}
\end{equation}
We also observe that \eqref{bound-nabla-Q} implies 
\begin{equation}\label{bound-nabla-Q2}
|\nabla Q|^{p_k(x)-2}\ge c_1 \qquad |\nabla Q|^{p_k(x)-1}\log|\nabla Q|\le c_2,
\end{equation}
in $B_1$, where $c_1=c_1(\eta, p_{\min})>0$ and $c_2=c_2(\eta,  p_{\max})>0$.

Now, from \eqref{def-M} and \eqref{def-pucci} it is not hard to see that  
\begin{equation}\label{pucci-M}
\mathcal{M}_{\lambda_0,\Lambda_0}^-(M)=-\Lambda_0 \ep (n-1)+\lambda_0 \delta\ge \frac{\lambda_0 \delta}{2},
\end{equation}
if $\ep\le \frac{\lambda_0 \delta}{2\Lambda_0 (n-1)}$. We next take $k$ large enough so that
\begin{equation}\label{bound-pk-fk}
|\nabla p_k|\le \frac{\lambda_0 \delta c_1}{2 c_2},\qquad |f_k|\le \frac{c_1\lambda_0 \delta}{4}\quad \text{ for }x\in B_1.
\end{equation}
Putting together \eqref{p(x)-Q}, \eqref{bound-nabla-Q2}, \eqref{pucci-M} and \eqref{bound-pk-fk}, we obtain in $B_1$
\begin{equation*}
\begin{split}
&\Delta_{p_k(x)}{Q}(x)\geq c_1 2\mathcal{M}_{\lambda_0,\Lambda_0}^-(M) - |\nabla p_k(x)|c_2\\
&\geq c_1 \lambda_0 \delta - \frac{\lambda_0 \delta c_1}{2 c_2}c_2>f_k.
\end{split}
\end{equation*}
If, additionally, $k$ is large so that
$$g_k\le 1+\frac{\eta}{8},\quad \text{ for }x\in B_1,$$
we obtain from \eqref{bound-nabla-Q} that
$$|\nabla Q|>g_k\quad \text{ in } B_1,$$
thus proving our claim.

We finally deduce from \eqref{u0smallerQ} that there exist a sequence $\sigma_k\to 0$ and points $x_k\in B_{\rho_0}$ such that, denoting $Q_k=Q+\sigma_k$, we get
\begin{equation*}
\left\{
\begin{aligned}
&u_k \ge Q_k \quad \text{in } B_{\rho_0},\\
&u_k(x_k)=Q(x_k),
\end{aligned}
\right.
\end{equation*}
if $k$ is large. We notice that if $u_k(x_k)>0$, then $Q_k(x_k)>0$. Otherwise $u_k(x_k)=0=Q_k(x_k)$, and since $\nabla Q_k(x_k)\neq 0$, then $x_k\in F(Q_k)$.

That is, for large $k$, ${Q}_k$ is a strict subsolution in $B_{\rho_0}$ to problem \eqref{fb}, with right hand side $f_k$, exponent $p_k$ and free boundary condition $g_k$, touching $u_k$ from  below at $x_k\in B^+_{\rho_0}(Q_k)\cup F(Q_k)$, a contradiction. Then $\alpha\leq 1.$ 
\end{proof}

We are now in a position to prove Theorem \ref{Lipmain}.

\medskip

\begin{proof}[\bf Proof of Theorem \ref{Lipmain}]
Let $u$ be a viscosity solution to \eqref{fb}
in $B_1$ such that $0\in F(u)$ and  such that $F(u)$ is a Lipschitz graph in $B_{r_0}$, for some $0<r_0\le 1$. Without loss of generality we assume that $g(0)=1$ and we denote $p(0)=p_0.$   

We will divide the proof into several steps.

\medskip

{\it Step I. Lipschitz continuity and nondegeneracy}. Let us first show that $u$ is Lipschitz and nondegenerate in a neighborhood of $0$.

In fact, for $0<r\le \frac{r_0}{2}\le \frac{1}{2}$, we consider the function  
$$\bar{u}(x) = \frac{1}{r} u(r x), \quad x\in B_2.$$

Then $\bar{u}$ is a viscosity solution to \eqref{fb} in $B_2$,  with right hand side $\bar{f}(x) = r f(r x)$, exponent $\bar{p}(x) =  p(r x)$ and 
free boundary condition  $\bar{g}(x) =  g(r x)$. Moreover, $0\in F(\bar{u})$.

{}From Theorem \ref{Lip-contin} we know that $\bar{u}$ is Lipschitz continuous in $B_{1/2}$ with a Lipschitz constant depending only on
 $n$, $p_{\min}$, $p_{\max},$  $\|\nabla p\|_{L^{\infty}(B_{3r_0/8})}$, $ \|f\|_{L^{\infty}(B_{3r_0/8})}$,  $\beta$, 
$\|g\|_{C^{0,\beta}(\overline{B_{3r_0/8}})}$ and $ \|u\|_{L^{\infty}(B_{3r_0/8})}$.

In order to prove the nondegeneracy, let us see that we can apply the second part of Proposition \ref{lemma7-1} to ${\bar u}$, if $r$ is suitably chosen.

For that purpose, let us first show that the constants appearing in that proposition can be taken independent of $r$. More precisely, we want to 
find a bound independent of $r$ for
$$
||\bar{u}||_{L^{\infty}(B_{3/2})}^{\bar{p}_{+}^{3/2}-\bar{p}_{-}^{3/2}},\qquad \text{where} \quad \bar{p}_{+}^{3/2}=\sup_{B_{3/2}}\bar{p}, \ \ \bar{p}_{-}^{3/2}=\inf_{B_{3/2}}\bar{p}.
$$
In fact, we have
\begin{equation}\label{bound-norm-bu-1}
{||\bar{u}||}_{L^{\infty}(B_{3/2})}^{\bar{p}_{+}^{3/2}-\bar{p}_{-}^{3/2}}\le ||u||_{L^{\infty}(B_{3r_0/4})}^{\bar{p}_{+}^{3/2}-
\bar{p}_{-}^{3/2}} \Big(\frac{1}{r}\Big)^{\bar{p}_{+}^{3/2}-\bar{p}_{-}^{3/2}},
\end{equation}
and
\begin{equation}\label{bound-norm-bu-2}
{\bar{p}_{+}^{3/2}-\bar{p}_{-}^{3/2}}\le 3 {||\nabla \bar{p}||}_{L^{\infty}(B_{3/2})} \le  3r{||\nabla {p}||}_{L^{\infty}(B_{3r_0/4})}.
\end{equation}
Then, from \eqref{bound-norm-bu-1} and \eqref{bound-norm-bu-2},  we conclude that
\begin{equation*}
||\bar{u}||_{L^{\infty}(B_{3/2})}^{\bar{p}_{+}^{3/2}-\bar{p}_{-}^{3/2}}\le 
C=C\big(||u||_{L^{\infty}(B_{3r_0/4})}, {||\nabla {p}||}_{L^{\infty}(B_{3r_0/4})}\big).
\end{equation*}

It follows that in order to apply the second part of  Proposition \ref{lemma7-1} to $\bar{u}$ we can take the constants $\tilde{\ep}$ and $c_0$  in that proposition depending only on
 $n$, $p_{\min}$, $p_{\max}$,  $||u||_{L^{\infty}(B_{3r_0/4})}$, ${||\nabla {p}||}_{L^{\infty}(B_{3r_0/4})}$, $||g||_{L^{\infty}(B_{r_0})}$, 
$\gamma_0$ and on the Lipschitz constant of $F(u)$.

Then, if $r$ is small enough, there holds in $B_2$
\begin{equation*}
\begin{aligned}
|\bar{f}(x)|&\le r ||f||_{L^{\infty}(B_{r_0})} \le \tilde{\ep},\\
|\bar{g}(x)-1|&=|g(r x)-g(0)|\le 2 {r}^{\beta} [g]_{C^{0,\beta}(B_{r_0})} \le \tilde{\ep},\\
|\nabla \bar{p}(x)|&\le r ||\nabla p||_{L^{\infty}(B_{r_0})}\le \tilde{\ep},\\
|\bar{p}(x)-p_0|&=|p(r x)-p(0)|\le 2 r ||\nabla p||_{L^{\infty}(B_{r_0})}\le \tilde{\ep}.
\end{aligned}
\end{equation*}
Hence, for $r$ small enough, $\bar{u}$ is  nondegenerate in $B_{\rho_0}$, for $\rho_0>0$ depending only on the Lipschitz constant of $F(u)$.

That is, $u$ is Lipschitz continuous and nondegenerate in $B_{\hat{\rho_0}}$, for a suitable universal $\hat{\rho_0}>0$, with a universal Lipschitz constant
$L_0$.

\medskip

{\it Step II. Blow up limit}. We now consider the blow up sequence
\begin{equation}\label{defuk}
u_k(x)=u_{\delta_k}(x)=\frac{u(\delta_k x)}{\delta_k},\quad \text{ where }\delta_k\to 0,
\end{equation}
$\delta_k> 0$. As before, each $u_k$ is a viscosity solution to \eqref{fb} with right hand side 
$f_k(x) = \delta_kf(\delta_kx)$, exponent $p_k(x) =  p(\delta_k x)$ and 
free boundary condition  $g_k(x) =  g(\delta_kx)$. 

Our goal is to apply Theorem \ref{flatmain1} to $u_k$, for large $k$. We will first observe that, taking $k$ sufficiently large, the assumption
\eqref{pflat} in that theorem is satisfied for the universal constant $\bar{\ep}$. In fact, in $B_1$,
\begin{equation}\label{condiconv}\begin{split}
&|f_k(x)|=\delta_k|f(\delta_k x)|\leq \delta_k ||f||_{L^{\infty}(B_{r_0})}\leq\bar{\varepsilon},\\
&|\nabla p_k(x)|\le \delta_k ||\nabla p||_{L^{\infty}(B_{r_0})}\le \bar\ep,\\
&|p_{{k}}(x)-p_0|=|p(\delta_k x)-p(0)|\leq\delta_k||\nabla p||_{L^{\infty}(B_{r_0})}\leq \bar{\varepsilon}, \\
&[g_k]_{C^{0,\beta}(B_1)}\leq\delta_k^{\beta}[g]_{C^{0,\beta}(B_{r_0})}\leq\bar{\varepsilon},\\
&|g_{{k}}(x)-1|=|g(\delta_k x)-g(0)|\leq\delta_k^{\beta}[g]_{C^{0,\beta}(B_{r_0})}\leq\bar{\varepsilon}. 
\end{split}
\end{equation}

On the other hand, since $u$ is Lipschitz and nondegenerate in $B_{\hat{\rho_0}}$, with Lipschitz constant $L_0$, then, for every $R>0$, $u_k$ are Lipschitz and uniformly nondegenerate in $B_{R}$,  with Lipschitz constant $L_0$, if $k\ge k_0(R)$. Then, standard arguments (see for instance, \cite{AC}, 4.7) imply that (up to  a subsequence), there holds that
\begin{equation}\label{ukconvu0}
\begin{aligned}
& u_k  \to u_0  \ \text{in} \ C_{\rm loc}^{0,\gamma}(\R^n),\  \text{for all}\  0<\gamma<1,\\
& \partial\{u_{k}>0\}\to\partial\{u_0>0\} \ \text{locally in Hausdorff distance},
\end{aligned}
\end{equation}
to a  function $u_0:\R^n\to \R$, which is globally Lipschitz with constant $L_0$ and nondegenerate in $\R^n$. Moreover, $F(u_0)$ is a global Lipschitz graph.

We also observe that the estimates in \eqref{condiconv} also imply that
\begin{equation*}
f_k\to 0, \quad \nabla p_k\to 0,\quad p_k\to p_0,\quad  g_k\to 1,\quad \text{ uniformly on compacts of } \ \R^n. 
\end{equation*}

\medskip

{\it Step III. Limit equation}. Since $u$ satisfies in the viscosity sense  $\Delta_{p(x)}u=f$ in $\{u>0\}$, then every $u_k$ satisfies in the viscosity sense $\Delta_{p_k(x)}u_k=f_k$ in $\{u_k>0\}$. 
We claim that the  blow up limit $u_0$
is a viscosity solution to $\Delta_{p_0}u_0=0$ in $\{u_0>0\}$.

 In fact, let us see that $u_0$ is a viscosity subsolution to $\Delta_{p_0}u_0=0$ in $\{u_0>0\}$.

Let $x_0 \in \{u_0>0\}$ and let $P$ be a quadratic polynomial such that $P\geq u_0$ in $B_\sigma(x_0),$  $P(x_0)=u(x_0)$  and $\nabla P(x_0)\not=0$. We can assume that $|\nabla P|\ge c>0$ in $B_\sigma (x_0)$ and $B_\sigma (x_0)\subset\{u_k>0\}$ for $k$ large, so 
$\Delta_{p_k(x)}u_k(x)=f_k(x)$ in $B_{\sigma}(x_0)$. We want to prove that $\Delta_{p_0}P(x_0)\geq 0$. We argue by contradiction assuming that there exists $\rho>0$ such that
$\Delta_{p_0}P(x_0)<-\rho<0.$ For $\varepsilon>0$, we define $\tilde{P}(x)=\tilde{P}_{\ep}(x)=P(x)+\varepsilon |x-x_0|^2.$ Hence $\nabla \tilde{P}=\nabla P+2\varepsilon (x-x_0)$  and 
\begin{equation}\label{graddue}|\nabla \tilde{P}|\geq \frac{c}{2} \quad\text{  in }B_{{\sigma}}(x_0),
\end{equation}
 if $\varepsilon$ is sufficiently small. Letting $\varepsilon\to 0$, we get
$$
\Delta_{p_0}\tilde{P}(x_0)\to \Delta_{p_0}P(x_0)<-\rho
$$
and then, if $\varepsilon$ is small enough, we obtain
\begin{equation}\label{gradtre}
\Delta_{p_0}\tilde{P}(x_0)<-\frac{\rho}{2}.
\end{equation}
We now fix $\varepsilon>0$ small  such that \eqref{graddue} and \eqref{gradtre} hold. We have 
\begin{equation}\label{tildePgeu0}
\tilde{P}(x)>u_0(x) \ \text{ in } \overline{B}_{{\sigma}}(x_0)\setminus\{x_0\}, \ \text{ and }\tilde{P}(x_0)=u_0(x_0).
\end{equation} 
Moreover, since $u_k\to u_0$ uniformly in 
$B_{\sigma}(x_0),$ then 
$|u_k-u_0|<\gamma_k$ in $B_{\sigma}(x_0)$ with $\gamma_k\to 0.$ 
Hence, from
$$
\tilde{P}(x)\geq u_0(x)>u_k(x)-\gamma_k \quad \text{ in }B_{{\sigma}}(x_0),
$$
it follows
$$
\tilde{P}(x)+\gamma_k>u_k(x)\quad \text{ in } B_{{\sigma}}(x_0).
$$

Let
 $$
t_k=\sup\{t\geq 0:\quad \tilde{P}(x)+\gamma_k\geq u_k(x)+t\quad\mbox{in}\:\:B_{{\sigma}}(x_0)\}.
$$
Since $\gamma_k\to 0$ and $\tilde{P}(x)+\gamma_k$ is bounded in $B_{{\sigma}}(x_0),$ then $t_k$ is finite so that 
$$
\tilde{P}(x)+\gamma_k\geq u_k(x)+t_k \quad \text{ in }B_{{\sigma}}(x_0)
$$
 and there exists $x_k\in \overline{B}_{{\sigma}}(x_0)$ such that
$$
\tilde{P}(x_k)+\gamma_k=u_k(x_k)+t_k.
$$
Then
\begin{equation*}
\begin{split}
u_k(x_0)+\gamma_k+\gamma_k\geq u_0(x_0)+\gamma_k=\tilde{P}(x_0)+\gamma_k\geq u_k(x_0)+t_k.
\end{split}
\end{equation*}
As a consequence
$$
t_k\leq 2\gamma_k\to 0
$$
and $t_k\to 0.$ Let $\tilde{P}_k(x)=\tilde{P}(x)+\gamma_k-t_k.$

Then 
$$
\tilde{P}_k(x)\geq u_k(x)\quad \text{ in }B_{{\sigma}}(x_0)
$$
 and $\tilde{P}_k(x_k)=u_k(x_k)$ for $x_k\in \overline{B}_{{\sigma}}(x_0).$

Since $\tilde{P}(x)>u_0(x)$ on $\partial B_{{\sigma}}(x_0)$
then 
$$
\tilde{P}(x)-u_0(x)\geq \bar{c}>0 
$$
on $\partial B_{{\sigma}}(x_0)$ and
\begin{equation*}
\begin{split}
&\tilde{P}_k(x)-u_k(x)=\tilde{P}(x)+\gamma_k-t_k-u_k(x)\geq \tilde{P}(x)+\gamma_k-t_k-u_0(x)-\gamma_k\\
&=\tilde{P}(x)-t_k-u_0(x)\geq \bar{c}-t_k\geq \frac{\bar{c}}{2}
\end{split}
\end{equation*}
on $\partial B_{{\sigma}}(x_0)$ if $k\geq k_0,$ since $t_k\to 0.$ We recall here that $u_k\leq u_0+\gamma_k.$ Hence $x_k\not\in\partial 
B_{{\sigma}}(x_0)$ if $k\geq k_0.$
Then
$$
\tilde{P}_k(x)\geq u_k(x)\quad \text{ in }B_{{\sigma}}(x_0),
$$
 $\tilde{P}_k(x_k)=u_k(x_k)$ for $x_k\in B_{{\sigma}}(x_0)$, and 
$\nabla\tilde{P}_k\not=0$ in $B_{{\sigma}}(x_0)$ and thus,
\begin{equation}\label{ineqPk}
\Delta_{p_k(x_k)}\tilde{P}(x_k)=\Delta_{p_k(x_k)}\tilde{P}_k(x_k)\geq f_k(x_k).
\end{equation}
Since $x_k\in B_{{\sigma}}(x_0)$ then, for a subsequence, $x_k\to \bar{x}\in \overline{B}_{{\sigma}}(x_0)$. Hence, using that
 $\gamma_k\to 0$,  $t_k\to 0$ and
$$
\tilde{P}({x_k})+\gamma_k-t_k=\tilde{P}_k(x_k)=u_k(x_k),
$$
we obtain that $\tilde{P}(\bar{x})=u_0(\bar{x}).$
 Then $\bar{x}=x_0,$ because \eqref{tildePgeu0} holds.

Now,  letting $k\to \infty$ in \eqref{ineqPk}, we get 
$$
\Delta_{p_0}\tilde{P}(x_0)\geq 0,
$$
which gives a contradiction to \eqref{gradtre}. Hence $\Delta_{p_0}P(x_0)\geq 0.$

Arguing in a similar way, we deduce  that $u_0$ is a viscosity supersolution to $\Delta_{p_0}u_0=0$ in $\{u_0>0\}$ as well.

\medskip

{\it Step IV. Limit free boundary problem}. We want to show that $u_0$ is a viscosity solution (in the sense of Definition \ref{def-LN}) to problem   
\begin{equation}  \label{fbliminftybis}
\left\{
\begin{array}{ll}
\Delta_{p_0} u_0 = 0, & \hbox{in $\{u_0>0\}$}, \\
\  &  \\
|\nabla u_0|= 1, & \hbox{on $F(u_0).$} 
\end{array}
\right.
\end{equation}
Hence we have to check that free boundary condition is satisfied in the sense of (i) and (ii) of that definition. We divide our analysis into two cases.

\medskip

{\bf Case (a).} Let $x_0\in F(u_0)$ such that there exists a ball $B_{r}(y)\subset \{u_0>0\}$, with $x_0\in \partial B_r(y)$. We denote 
$\nu=\frac{y-x_0}{|y-x_0|}$. Then, by Case (a) in Lemma \ref{lemma51},
\begin{equation}\label{asymp-right}
u_0(x)=\alpha \langle x-x_0,\nu\rangle^++o(|x-x_0|)\quad \text{ in } B_r(y),
\end{equation}
with $\alpha>0.$

 We now consider a sequence $\lambda_j\to 0,$   $\lambda_j> 0$. Since $u_0$ is Lipschitz in $\R^n$ then there exists a function
$u_{00}$ such that, for a subsequence, 
$$
\frac{u_0(x_0+\lambda_j x)}{\lambda_j}\to u_{00}(x)\quad \text {uniformly on compact sets of }\R^n.
$$
{}From \eqref{asymp-right} we know that $u_{00}(x)=\alpha \langle x,\nu\rangle^+$ in $\{\langle x,\nu\rangle\geq 0\}.$ 
Then $\{\langle x,\nu\rangle = 0\} \subset F(u_{00})$. Since $F(u_0)$ is a Lipschitz graph, also $F(u_{00})$ is a Lipschitz graph,  so we have  
$\{\langle x,\nu\rangle = 0\} = F(u_{00}).$ Hence,
$$
u_{00}(x)=\alpha \langle x,\nu\rangle^+ \quad \text{ in }\mathbb{R}^n.
$$

This result holds for any sequence $\lambda_j\to 0,$   $\lambda_j> 0,$ therefore 
\begin{equation}\label{asympu0}
u_0(x)=\alpha \langle x-x_0,\nu\rangle^++o(|x-x_0|)\quad \text{ in } \R^n,
\end{equation}
with $\alpha>0.$ 

We want to show that $\alpha=1$.

Since $x_0\in F(u_0)$ and recalling   \eqref{defuk} and \eqref{ukconvu0}, we know that there exists, up to a subsequence, 
\begin{equation*}
x_k\in F(u_k), \quad |x_k-x_0|<1/k.
\end{equation*}
We fix $R>0$ such that $|x_0|<R$ and let $\mu_j=1/\sqrt{j}$. 

For each $j$ there exists $k_j\ge j$ such that 
$$
|u_{{k_j}}(x)-u_0(x)|\le \frac{\mu_j}{j} \qquad\text{for }x\in B_{R+2}.
$$   
We now define
$$(u_{k_j})_{\mu_j}(x)=\frac1{\mu_j}{u_{k_j}(x_{k_j}+\mu_j x)}, \quad (u_{0})_{\mu_j}(x)=\frac1{\mu_j}{u_0(x_{k_j}+\mu_j x)}.$$
Then, if $j\ge j_0$,
$$
|(u_{k_j})_{\mu_j}(x)-(u_{0})_{\mu_j}(x)|=  \frac{|u_{k_j}(x_{k_j}+\mu_j x) - u_0(x_{k_j}+\mu_j x)|}{\mu_j}\le \frac{1}{j} \qquad\text{for }x\in B_{2}.
$$
We now observe that 
$$\frac{|x_{k_j}-x_0|}{\mu_j}<\frac{1/k_j}{\mu_j}\le \frac{1}{\sqrt{j}}\quad \to 0,$$
and, recalling  \eqref{asympu0}, we obtain
$$
(u_{0})_{\mu_j}(x)\to u_{00}(x)=\alpha \langle x,\nu\rangle^+ \quad \text{ uniformly in } B_2.
$$
Then,
$$
\begin{aligned}
|(u_{k_j})_{\mu_j}(x)-u_{00}(x)|\le &|(u_{k_j})_{\mu_j}(x)-(u_{0})_{\mu_j}(x)|\\
& +|(u_{0})_{\mu_j}(x) - u_{00}(x)| \to 0 \quad \text{ uniformly in } B_2.
\end{aligned}
$$
Denoting  $\rho_j=\delta_{k_j}\mu_j$, $\bar{x}_j=\delta_{k_j}x_{k_j}$ and $u_{\rho_j}(x)=\frac 1{\rho_j}u(\bar{x}_j+\rho_j x)=(u_{k_j})_{\mu_j}(x)$, we get
\begin{equation*}
\begin{aligned}
\rho_j\to 0, \quad &\bar{x}_j \in F(u),\quad \bar{x}_j\to 0,\\
 u_{\rho_j}(x)=\frac 1{\rho_j}u(\bar{x}_j+\rho_j x)&\to \alpha \langle x,\nu\rangle^+ \quad \text{ uniformly in } B_2.
\end{aligned}
\end{equation*}
Reasoning as in {\it Step II}, we see that each $u_{\rho_j}$ is a viscosity solution to \eqref{fb} in $B_2$ with right hand side 
$\bar{f}_j(x) = \rho_j f(\bar{x}_j+\rho_j x)$, exponent $\bar{p}_j(x) =  p(\bar{x}_j+\rho_j x )$ and 
free boundary condition  $\bar{g}_j(x) =  g(\bar{x}_j+\rho_j  x)$,
\begin{equation*}
\bar{f}_j\to 0, \quad \nabla\bar{p}_j\to 0,\quad \bar{p}_j\to p_0,\quad  \bar{g}_j\to 1,\quad \text{ uniformly in }B_2. 
\end{equation*}
Moreover, $u_{\rho_j}$ are uniformly Lipschitz and nondegenerate in $B_2$ for $j\ge j_0$, $\partial\{ u_{\rho_j}>0\}$ are uniform Lipschitz graphs and $\partial\{ u_{\rho_j}>0\}\to \{\langle x,\nu\rangle=0\}$ in Hausdorff distance in $B_2$.

Now, applying Propositions \ref{casealphageq1} and \ref{alpha-leq-1} to the sequence $u_{\rho_j}$, we deduce that $\alpha=1$. Then, (i) in Definition \ref{def-LN} is satisfied in this case.

\medskip

{\bf Case (b).} Let $x_0\in F(u_0)$ such that there exists a ball $B_{r}(y)\subset \{u_0\equiv 0\}$, with $x_0\in \partial B_r(y)$. We denote 
$\nu=\frac{x_0-y}{|x_0-y|}$. Then, from the proof of Case (b) in Lemma \ref{lemma51}, we get
\begin{equation}\label{asymp-left}
u_0(x)=\alpha \langle x-x_0,\nu\rangle^++o(|x-x_0|)\quad \text{ in } B^c_r(y),
\end{equation}
with $\alpha \ge 0.$

 We now consider a sequence $\lambda_j\to 0,$   $\lambda_j> 0$. Then, for a subsequence and a function $u_{00}$,
$$
\frac{u_0(x_0+\lambda_j x)}{\lambda_j}\to u_{00}(x)\quad \text {uniformly on compact sets of }\R^n.
$$
{}From \eqref{asymp-left} we know that $u_{00}(x)=\alpha \langle x,\nu\rangle^+$ in $\{\langle x,\nu\rangle\geq 0\}.$ 
Since $B_{r}(y)\subset \{u_0\equiv 0\}$, we have $u_{00}(x)=0$ in $\{\langle x,\nu\rangle\leq 0\}.$ 
 Hence,
$$
u_{00}(x)=\alpha \langle x,\nu\rangle^+ \quad \text{ in }\mathbb{R}^n.
$$
Now, if $\alpha=0$, then $u_{00}\equiv 0$ in $\mathbb{R}^n$. This contradicts that $F(u_{00})$ is a Lipschitz graph and shows that $\alpha>0$. 

Since this result holds for any sequence $\lambda_j\to 0,$   $\lambda_j> 0,$ we conclude that 
$$
u_0(x)=\alpha \langle x-x_0,\nu\rangle^++o(|x-x_0|)\quad \text{ in } \R^n,
$$
with $\alpha>0.$  Now proceeding as in Case (a), we obtain that $\alpha=1$. Then, (ii) in Definition \ref{def-LN} is satisfied in the present case. 

This shows that $u_0$ is a viscosity solution to problem \eqref{fbliminftybis} in the sense of Definition \ref{def-LN}.

\medskip

{\it Step V. Conclusion}. We have proved that $u_0$ is a viscosity solution (in the sense of Definition \ref{def-LN}) to   
\eqref{fbliminftybis} that is Lipschitz continuous and $F(u_0)$ is a Lipschitz graph.

Thus, from Lemma \ref{fbliminftylinlemma} it follows that, up to a rotation, $u_0(x)=x_n^+.$ Then for sufficiently large $k$ we have that, in $B_1$,
\begin{equation}\label{flat-cond-uk}
(x_n-\bar{\varepsilon})^+\leq u_{k}(x)\leq (x_n+\bar{\varepsilon})^+,
\end{equation}
for $\bar{\ep}$ the universal constant in Theorem \ref{flatmain1}. Recalling \eqref{condiconv}, we deduce that  Theorem \ref{flatmain1} applies and, as a consequence, we conclude that the free boundaries of $u_{k}$ as well as that of $u$ are $C^{1,\alpha}$, in a neighborhood of $0$. 
\end{proof}

As a by-product of Theorems \ref{flatmain1} and  \ref{Lipmain}, we obtain further regularity results for $F(u)$ under additional regularity assumptions on the data.  

\begin{cor}\label{higher-reg-F(u)}
Let $u$ be as in Theorem \ref{flatmain1} or as in Theorem \ref{Lipmain}.
Assume moreover that $p\in C^2(B_1)$, $f\in C^1(B_1)$ and
$g\in C^2(B_1)$, then there exists $\delta>0$ such that
$B_{\delta}\cap F(u)\in C^{2,\sigma}$ for every
$0<\sigma<1$. If $p\in C^{m+1,\sigma}(B_1)$, $f\in C^{m,\sigma}(B_1)$
and $g\in C^{m+1,\sigma}(B_1)$ for some $0<\sigma<1$ and
$m\ge1$, then $B_{\delta}\cap F(u)\in C^{m+2,\sigma}$.

Finally, if $p$, $f$ and $g$ are analytic in $B_1$, then
$B_{\delta}\cap F(u)$ is analytic.
\end{cor}
\begin{proof} The result follows from the application of Theorem 2 in \cite{KN}.
\end{proof}

\section{Some consequences}\label{conseq}

In this section we discuss some consequences of our results.

As already mentioned, in \cite{LW2} problem \eqref{fb} was considered for weak solutions, which is a different notion of solution from the one we are considering here (see Definition \ref{weak-lw} below). One of the consequences of our Theorem \ref{Lipmain} is an analogous
result for weak solutions (Corollary \ref{Lip-lw2}).

The notation and the assumptions on $\Omega, p, f$ and $g$ will be the same as in the rest of the paper (see Subsection \ref{assump} and Section \ref{section2}). In particular we will use the notation $\Omega^+(u)$ and $F(u)$ in \eqref{def-fb}. 

\smallskip

We first have 

\begin{defn}[Definition 2.2 in \cite{LW2}]\label{weak-lw}  
We call $u$ a weak solution of \eqref{fb} in $\Omega$ if
\begin{enumerate}
\item $u$ is continuous and nonnegative in $\Omega$, $u\in W^{1,p(\cdot)}(\Omega)$ and
$\Delta_{p(x)}u=f$ in $\Omega^+(u)$ (in the sense of Definition \ref{defnweak}).
\item For
$D\subset\subset \Omega$ there are constants $c_{\min}=c_{\min}(D)$, $C_{\max}=C_{\max}(D)$, $r_0=r_0(D)$, $0< c_{\min}\leq
C_{\max}$, $r_0>0$, such that for balls $B_r(x)\subset D$ with $x\in
F(u)$ and $0<r\le r_0$
$$
c_{\min}\leq \frac{1}{r}\sup_{B_r(x)} u \leq C_{\max}.
$$
\item For $\mathcal{H}^{n-1}$ a.e.
$x_0\in\partial_{\rm{red}}\{u>0\}$ (that is, for ${\mathcal
H}^{n-1}$-almost every point $x_0\in F(u)$ such that
$F(u)$ has an exterior  unit normal
 $\nu(x_0)$ in the measure theoretic sense)
$u$ has the asymptotic development
\begin{equation*}\label{asym-w}
u(x)=g(x_0)\langle x-x_0,\nu(x_0)\rangle^-+o(|x-x_0|).
\end{equation*}

\item For every $ x_0\in F(u)$,
\begin{align*}
& \limsup_{\stackrel{x\to x_0}{u(x)>0}} |\nabla u(x)| \leq
g(x_0).
\end{align*}

If there is a ball $B\subset\{u=0\}$ touching
$F(u)$ at $x_0$ then,
$$\limsup_{\stackrel{x\to x_0}{u(x)>0}} \frac{u(x)}{\mbox{dist}(x,B)}\geq  g(x_0). $$
\end{enumerate}
\end{defn}

Then we prove

\begin{prop} \label{weak-sol-is-visc-sol} 
 Let $u$ be a weak solution to \eqref{fb} in $\Omega$ in the sense of Definition \ref{weak-lw}. Then $u$ is a viscosity solution 
to \eqref{fb} in $\Omega$ in the sense of Definition \ref{defnhsol1}.
\end{prop}

\begin{proof} Let $u$ be as in the statement. Then $u$ is continuous and nonnegative in $\Omega$ and satisfies condition (i) in 
Definition \ref{defnhsol1}. In order to show that it verifies condition (ii) in that definition, we divide the analysis into two cases.

\smallskip

{\bf Case (a).} Let $\varphi \in C(\Omega)$, $\varphi \in C^2(\overline{\Omega^+(\varphi)})$ be such that $\varphi^+$ touches $u$ from 
below at $x_0 \in F(u)$ and $\nabla \varphi(x_0)\not=0$. We want to show that 
\begin{equation}\label{u-is-sub}
|\nabla \varphi(x_0)| \leq g(x_0).
\end{equation}

We first observe that, under the present assumptions, Proposition 2.1 in \cite{LW2} applies, so $u$ is locally Lipschitz in $\Omega$.

Also there holds that $\varphi^+$ has a $C^2$ extension $\tilde\varphi$ in a neigborhood $\mathcal O$ of $x_0$ ($\tilde\varphi=\varphi^+$ in
$\overline{\Omega^+(\varphi)}\cap\mathcal O$, $\tilde\varphi <0$ otherwise in $\mathcal O$), that to simplify the notation we still denote $\varphi$. 

Moreover, $\varphi$ touches $u$ from below at $x_0 \in F(u)$ as well.

By the implicit function theorem,  $F(\varphi)$ is a $C^2$ hypersurface in a neighborhood of $x_0$. 
Then, $F(\varphi)$ has a tangent ball $B$ at $x_0$, with $B\subset \Omega^+(\varphi)$ and also with  $B\subset \Omega^+(u)$ and
$x_0\in F(u)\cap \partial B$.

 We now consider a sequence $\lambda_j\to 0,$   $\lambda_j> 0$. Since $u$ and $\varphi$ are Lipschitz in a neighborhood of $x_0$, then there exist Lipschitz functions
$u_0$ and $\varphi_0$ such that, for a subsequence, 
$$
u_{\lambda_j}(x)=\frac{u(x_0+\lambda_j x)}{\lambda_j}\to u_{0}(x),\qquad \frac{\varphi(x_0+\lambda_j x)}{\lambda_j}\to \varphi_{0}(x),
$$
uniformly on compact sets of $\R^n$. For simplicity we assume that the interior normal to $\partial B$ at $x_0$ is $e_n$.
Then
$$u_0(x)\ge \varphi_0(x)=|\nabla \varphi(x_0)|x_n^+ \ \text{ in }\{x_n\ge 0\},$$
$$\Delta_{p_0} u_0=0 \ \text{ in }\{u_0>0\}\supset\{x_n > 0\}, \ \text{ with } p_0=p(x_0).$$

Then, the application of Lemma \ref{lemma51}, Case (a), at the origin, gives
$$u_0(x)=\gamma x_n^+ +o(|x|) \ \text{ in } B_1(e_n), \text{ with } \gamma>0.$$

 We now consider a sequence $\mu_j\to 0,$   $\mu_j> 0$. Then, there exist Lipschitz functions
$u_{00}$ and $\varphi_{00}$ such that, for a subsequence, 
$$
({u_0})_{\mu_j}(x)=\frac{u_0(\mu_j x)}{\mu_j}\to u_{00}(x),\qquad \frac{\varphi_0(\mu_j x)}{\mu_j}\to \varphi_{00}(x),
$$
uniformly on compact sets of $\R^n$.  There holds that
$$u_{00}(x)=\gamma x_n^+\ge  \varphi_{00}(x)=|\nabla \varphi(x_0)|x_n^+ \ \text{ in }\{x_n\ge 0\},$$
and
\begin{equation}\label{lower-bound-grad-u00}
|\nabla u_{00}(x)|=\gamma \ge  |\nabla \varphi(x_0)|\ \text{ in }\{x_n> 0\}.
\end{equation}
Now let
$$
\alpha:=\limsup_{\stackrel{x\to x_0}{u(x)>0}} |\nabla u(x)|.
$$
Then, by (iv) in Definition \ref{weak-lw}, we have
\begin{equation}\label{bound-lim-sup-grad}
g(x_0)\ge \alpha.
\end{equation}

Let us see that 
\begin{equation}\label{bound-grad-u00}
|\nabla u_{00}|\leq \alpha \ \text{ in } \R^n.
\end{equation} 
In fact,
let $R>0$ and $\epsilon>0$. Then, there exists $\lambda_0>0$ such that
$|\nabla u(x)|\leq \alpha+\epsilon$ in $B_{\lambda_0
R}(x_0)$. We thus have
$|\nabla u_{\lambda_j}(x)|\leq \alpha +\epsilon$ in $B_R$ for $j$ large. Passing to
the limit, we obtain $|\nabla u_0|\leq \alpha+\epsilon$ in $B_R$ and then  $|\nabla u_{0}|\leq \alpha$ in $\R^n$.
Now also $|\nabla({u_0})_{\mu_j}|\leq \alpha$ in $\R^n$. Passing to the limit again, we obtain \eqref{bound-grad-u00}.

Then, \eqref{bound-lim-sup-grad}, \eqref{bound-grad-u00} and \eqref{lower-bound-grad-u00} give $g(x_0)\ge \alpha\ge \gamma \ge  |\nabla \varphi(x_0)|$. That is,  \eqref{u-is-sub} holds.

\smallskip

{\bf Case (b).} Now let $\varphi \in C(\Omega)$, $\varphi \in C^2(\overline{\Omega^+(\varphi)})$ such that $\varphi^+$ touches $u$ 
from above at $x_0 \in F(u)$ and $\nabla \varphi(x_0)\not=0$. We want to show that 
\begin{equation}\label{u-is-super}
|\nabla \varphi(x_0)| \geq g(x_0).
\end{equation}

Also in this case there holds that $\varphi^+$ has a $C^2$ extension $\tilde\varphi$ in a neighborhood of $x_0$, that to simplify the notation we still denote $\varphi$. 

By the implicit function theorem,  $F(\varphi)$ is a $C^2$ hypersurface in a neighborhood of $x_0$. 
Then, $F(\varphi)$ has a tangent ball $B$ at $x_0$, with $B\subset \Omega\setminus\overline{\Omega^+(\varphi)}$ and also with  
$B\subset \{u= 0\}$ and
$x_0\in F(u)\cap \partial B$.

Now let
$$
\alpha:=\limsup_{\stackrel{x\to x_0}{u(x)>0}} \frac{u(x)}{\mbox{dist}(x,B)}.
$$
Then, by (iv) in Definition \ref{weak-lw}, we have
\begin{equation}\label{lower-bound-lim-sup-grad}
g(x_0)\le \alpha.
\end{equation}

 Let  $x_k \to x_0$ with $u(x_k)>0$ be such that 
\begin{equation}\label{bound1}
\frac{u(x_k)}{\mbox{dist}(x_k,B)}\to \alpha.
\end{equation}

Since $\varphi^+\ge u$ in a neighborhood of $x_0$, then $\varphi(x_k)>0$. Now let $y_k\in \partial B$ such that $\mbox{dist}(x_k,B)=|x_k-y_k|$. Then $\varphi (y_k)\le 0$ and
\begin{equation}\label{bound2}
\frac{\varphi(x_k)-\varphi(y_k)}{|x_k-y_k|}\ge\frac{\varphi(x_k)}{\mbox{dist}(x_k,B)}\ge \frac{u(x_k)}{\mbox{dist}(x_k,B)}.
\end{equation}

But, for a subsequence,
\begin{equation}\label{bound3}
\frac{\varphi(x_k)-\varphi(y_k)}{|x_k-y_k|}=\nabla \varphi(\xi_k)\cdot \frac{(x_k- y_k)}{|x_k-y_k|}\to 
\nabla \varphi(x_0)\cdot  \frac{\nabla \varphi (x_0)}{|\nabla \varphi (x_0)|},
\end{equation}
where  for every $k$, $\xi_k$ is a point in the segment joining $x_k$ and $y_k$. Putting \eqref{bound1}, \eqref{bound2} and \eqref{bound3} together
we get $|\nabla \varphi (x_0)|\ge \alpha$. Now recalling \eqref{lower-bound-lim-sup-grad}, we get \eqref{u-is-super} which completes the proof.
\end{proof}

Then,   we obtain
\begin{cor}\label{Lip-lw2}
Let $u$ be a weak solution to \eqref{fb} in $B_1$ in the sense of Definition \ref{weak-lw}, with
$0\in F(u)$. 
 If $F(u)$ is a Lipschitz graph in a neighborhood of $0$, then 
$F(u)$ is $C^{1,\alpha}$ in a (smaller) neighborhood of $0$.
\end{cor}
\begin{proof} The result is an immediate application of Theorem  \ref{Lipmain} and Proposition \ref{weak-sol-is-visc-sol}.
\end{proof}

\section{Some applications}\label{applic}

In this section we discuss some applications of both the results obtained in the present paper and in \cite{FL}, and we draw some conclusions
on them (see Remark \ref{rem-concl-applic}).

The applications of our results discussed here correspond to three different minimization problems that were already studied in \cite{LW1}, \cite{LW3} and \cite{LW4}. Our results below rely on the thorough understanding of the properties of nonnegative local minimizers achieved in those papers. We also refer to them for the motivation and related literature.

The notation and the assumptions on $\Omega, p$ and $f$ will be the same as in the rest of the paper (see Subsection \ref{assump} and Section \ref{section2}).  In particular we will use the notation $\Omega^+(u)$ and $F(u)$ in \eqref{def-fb}.

\smallskip

Our first application is

\begin{prop}\label{applic-AC} Let $\Omega$, $p$ and $f$ be as above. Let $0<{\lambda_{\min}}\le\lambda(x)\le{\lambda_{\max}}<\infty$ with 
$\lambda\in C^{0, \beta}(\Omega)$.
Let $u\in W^{1,p(\cdot)}(\Omega)\cap L^{\infty}(\Omega)$   be a nonnegative local minimizer of the energy functional 
$J(v)={\displaystyle\int_\Omega\Big(\frac{|\nabla
v|^{p(x)}}{p(x)}+\lambda(x)\chi_{\{v>0\}}+fv\Big)\,dx}$
in $\Omega$.

Then,  $u$ is a viscosity solution to \eqref{fb} in $\Omega$ with $g(x)=(\frac{p(x)}{p(x)-1}\,\lambda(x))^{1/p(x)}$.

Let $x_0\in F(u)$ be such that $F(u)$ is a Lipschitz graph in a neighborhood of $x_0$, then  $F(u)$ is $C^{1,\alpha}$ in a (smaller) neighborhood of $x_0$.

Let  $x_0\in F(u)$ be such that $F(u)$ has a normal in the measure theoretic sense, then $F(u)$ is $C^{1,\alpha}$ in a neighborhood of $x_0$.

Moreover, there is a subset $\mathcal{R}$  of $F(u)$ which is locally a
$C^{1,\alpha}$ surface. The set $\mathcal{R}$ is open and
dense in $F(u)$ and the remainder of the free
boundary has $(n-1)-$dimensional Hausdorff measure zero.
\end{prop}
\begin{proof} By Theorem 5.1 in \cite{LW3}, $u$ is a weak solution  to \eqref{fb} in $\Omega$ with $g(x)=(\frac{p(x)}{p(x)-1}\,\lambda(x))^{1/p(x)}$ in the sense of Definition \ref{weak-lw}. Then by Proposition \ref{weak-sol-is-visc-sol}, $u$ is a viscosity solution 
to \eqref{fb} in $\Omega$ in the sense of Definition \ref{defnhsol1}, with the same $g$.

Let $x_0\in F(u)$ be such that $F(u)$ is a Lipschitz graph in a neighborhood of $x_0$. Then, from the application of Theorem \ref{Lipmain},  $F(u)$ is $C^{1,\alpha}$ in a smaller neighborhood of $x_0$.

Let  $x_0\in F(u)$ be such that $F(u)$ has a normal in the measure theoretic sense. Without loss of generality we assume that $x_0=0$, $g(0)=1$ and that
the inward unit normal to $F(u)$ at $0$ in the measure theoretic sense is $e_n$. Also we denote $p(0)=p_0$.

Then, by Theorem 3.9 in \cite{LW3} there holds that
\begin{equation}\label{develop-for-u}
u(x)=x_n^+ +o(|x|) \ \text{ in } \R^n.
\end{equation}

By Corolary 3.2 and Theorem 3.5 in \cite{LW3} we know that $u$ is Lipschitz and nondegenerate in some ball $B_{r_0}$, with $0<r_0<1$.

Then, as in {\it Step II}  in the proof of Theorem \ref{Lipmain}, we take $\delta_k>0$, $\delta_k\to 0$, and consider a blow up sequence $u_k$ as in 
\eqref{defuk}. As in that theorem, our goal is to apply Theorem \ref{flatmain1} to $u_k$, for large $k$. We first observe that, taking $k$ sufficiently large, the assumption \eqref{pflat} in that theorem is satisfied for the universal constant $\bar{\ep}$. In fact, in $B_1$,
\eqref{condiconv} holds.

Arguing again as in Theorem \ref{Lipmain}, we see that \eqref{ukconvu0} holds with
$u_0(x)=x_n^+ \ \text{ in } \R^n,$
because of \eqref{develop-for-u}. Then, reasoning as in this same theorem and using Theorem 3.6 in \cite{LW3}, we obtain for $k$ sufficiently large that  \eqref{flat-cond-uk} holds in $B_1$,
for $\bar{\ep}$ the universal constant in Theorem \ref{flatmain1}. Therefore,  Theorem \ref{flatmain1} applies to $u_k$, and as a consequence,  $F(u)$ is $C^{1,\alpha}$ in a neighborhood of $0$. 

Finally, denoting $\mathcal{R}$  the set of points in $F(u)$ such that
$F(u)$ has a   normal in the measure theoretic sense, we argue as in Theorem 5.2 in \cite{LW3} and obtain that $\mathcal{R}$ is dense
in $F(u)$ and 
${\mathcal H}^{n-1}(F(u)\setminus\mathcal{R})=0$.
\end{proof}

Our next application is

\begin{prop}\label{applic-sing-pert} For $\ep>0$, let 
$B_\ep(s)=\int _0^s\tilde{\beta}_\ep(\tau) \, d\tau$ where ${\tilde{\beta}}_{\varepsilon}(s)={1 \over \varepsilon} \tilde{\beta}({s \over
\varepsilon})$, with $\tilde{\beta}$  a  Lipschitz  function satisfying
$\tilde{\beta}>0$ in $(0,1)$, $\tilde{\beta}\equiv 0$ outside $(0,1)$. 
Let $\Omega$, $p$ and $f$ be as above,  $1<p_{\min}\le
p_{\varepsilon_j}(x)\le p_{\max}<\infty$ and $\|\nabla p_{\varepsilon_j}\|_{L^{\infty}}\leq L$.   
Let $u^{\varepsilon_j}\in W^{1,p_{\varepsilon_j}(\cdot)}(\Omega)$ be a family of nonnegative local minimizers of the energy functional
$J_{\varepsilon_j}(v)={\displaystyle\int_\Omega \Big(\frac{|\nabla v|^{p_{\varepsilon_j}(x)}}{p_{\varepsilon_j}(x)}+B_{\varepsilon_j}(v)+{f^{{\varepsilon}_j}} v\Big)\, dx}$
in $\Omega$ such that $u^{\varepsilon_j}\rightarrow u$ uniformly on compact subsets of $\Omega$,
${f^{{\varepsilon}_j}}\rightharpoonup f$ $*-$weakly in $L^\infty(\Omega)$,
$p_{\varepsilon_j}\to p$ uniformly on compact subsets of $\Omega$ and
${\varepsilon_j}\to 0$.

Then,  $u$ is a viscosity solution to \eqref{fb} in $\Omega$ with $g(x)=(\frac{p(x)}{p(x)-1}\,M)^{1/p(x)}$ and
$M=\int \tilde{\beta}(s)\, ds$.

Let $x_0\in F(u)$ be such that $F(u)$ is a Lipschitz graph in a neighborhood of $x_0$, then  $F(u)$ is $C^{1,\alpha}$ in a (smaller) neighborhood of $x_0$.

Let  $x_0\in F(u)$ be such that $F(u)$ has a normal in the measure theoretic sense, then $F(u)$ is $C^{1,\alpha}$ in a neighborhood of $x_0$.

Moreover, there is a subset $\mathcal{R}$  of $F(u)$ which is locally a
$C^{1,\alpha}$ surface. The set $\mathcal{R}$ is open and
dense in $F(u)$ and the remainder of the free
boundary has $(n-1)-$dimensional Hausdorff measure zero.
\end{prop}
\begin{proof} We argue exactly as in the proof of Proposition \ref{applic-AC}. We apply again our results in Proposition \ref{weak-sol-is-visc-sol} and Theorems \ref{Lipmain} and \ref{flatmain1}, and in this case we make use of Theorems 5.3, 4.3, 4.4 and Remark 4.2 in \cite{LW3}, and Theorem 5.3 in \cite{LW1}. \end{proof}

\medskip

We also obtain

\medskip

\begin{rem} \label{applic-optim}
In \cite{LW4} an optimization problem with volume constraint for an energy associated to the inhomogeneous $p(x)$-Laplacian was considered. By means of a penalization technique, it was shown that nonnegative minimizers $u$ are weak solutions to \eqref{fb}  in a bounded domain
$\Omega$ in the sense of Definition \ref{weak-lw}  with $g(x)=(\frac{p(x)}{p(x)-1}\,\lambda_{u})^{1/p(x)}$, where $\lambda_{u}>0$ is a constant.

Under the assumptions we made on $p$ and $f$ at the beginning of present section, by combining our results with those in \cite{LW4}, we can argue as in Propositions \ref{applic-AC} and \ref{applic-sing-pert} and obtain  the same conclusions for $u$ and $F(u)$.
\end{rem}

\medskip

\begin{rem} \label{rem-concl-applic}
In Propositions \ref{applic-AC} and \ref{applic-sing-pert} and Remark \ref{applic-optim}, our $C^{1,\alpha}$ regularity results on $F(u)$ under the Lipschitz assumption on $F(u)$ follow from the application of Theorem \ref{Lipmain} in the present paper and are new.

We want to point out that the rest our $C^{1,\alpha}$ regularity results on $F(u)$ in Propositions \ref{applic-AC} and \ref{applic-sing-pert} and Remark \ref{applic-optim}, which follow from Theorem \ref{flatmain1} (i.e., Theorem 1.1 in \cite{FL}), were already obtained in \cite{LW3} and \cite{LW4}, from the application of the results in \cite{LW2}, but under different assumptions on $f$ and $p$.

In fact, our results in \cite{FL} ---inspired in De Silva's approach (see \cite{D})--- require that $f\in C(\Omega)\cap L^{\infty}(\Omega)$ and $p\in C^1(\Omega)$ and Lipschitz, whereas the results in \cite{LW2} ---inspired in Alt - Caffarelli's approach 
(see \cite{AC})--- require that $f\in L^{\infty}(\Omega)\cap W^{1,q}(\Omega)$ and $p\in W^{1,\infty}(\Omega)\cap W^{2,q}(\Omega)$, for $q>\max\{1, n/2\}$.

The reason for this difference in the assumptions relies on the fact that in De Silva's approach for viscosity solutions the estimates are obtained by comparison with
suitable barriers. In Alt - Caffarelli's approach for weak (variational) solutions, certain estimates on $|\nabla u|$ close to the free boundary are obtained by looking for an equation for $v=|\nabla u|$, which requires more delicate computations.
\end{rem}

\appendix

\section{Lebesgue and Sobolev spaces with variable
exponent} \label{appA1}

Let $p :\Omega \to  [1,\infty)$ be a measurable bounded function,
called a variable exponent on $\Omega$, and denote $p_{\max} = {\rm
ess sup} \,p(x)$ and $p_{\min} = {\rm ess inf} \,p(x)$. The variable exponent Lebesgue space $L^{p(\cdot)}(\Omega)$ is defined as the set of all measurable functions $u :\Omega \to \R$ for which
the modular $\varrho_{p(\cdot)}(u) = \int_{\Omega} |u(x)|^{p(x)}\,
dx$ is finite. The Luxemburg norm on this space is defined by
$$
\|u\|_{L^{p(\cdot)}(\Omega)} = \|u\|_{p(\cdot)}  = \inf\{\lambda >
0: \varrho_{p(\cdot)}(u/\lambda)\leq 1 \}.
$$

This norm makes $L^{p(\cdot)}(\Omega)$ a Banach space.

There holds the following relation between $\varrho_{p(\cdot)}(u)$
and $\|u\|_{L^{p(\cdot)}}$:
\begin{align*}
\min\Big\{\Big(\int_{\Omega} |u|^{p(x)}\, dx\Big)
^{1/{p_{\min}}},& \Big(\int_{\Omega} |u|^{p(x)}\, dx\Big)
^{1/{p_{\max}}}\Big\}\le\|u\|_{L^{p(\cdot)}(\Omega)}\\
 &\leq  \max\Big\{\Big(\int_{\Omega} |u|^{p(x)}\, dx\Big)
^{1/{p_{\min}}}, \Big(\int_{\Omega} |u|^{p(x)}\, dx\Big)
^{1/{p_{\max}}}\Big\}.
\end{align*}

Moreover, the dual of $L^{p(\cdot)}(\Omega)$ is
$L^{p'(\cdot)}(\Omega)$ with $\frac{1}{p(x)}+\frac{1}{p'(x)}=1$.

 $W^{1,p(\cdot)}(\Omega)$ denotes the space of measurable
functions $u$ such that $u$ and the distributional derivative
$\nabla u$ are in $L^{p(\cdot)}(\Omega)$. The norm

$$
\|u\|_{1,p(\cdot)}:= \|u\|_{p(\cdot)} + \| |\nabla u|
\|_{p(\cdot)}
$$
makes $W^{1,p(\cdot)}(\Omega)$ a Banach space.

The space $W_0^{1,p(\cdot)}(\Omega)$ is defined as the closure of
the $C_0^{\infty}(\Omega)$ in $W^{1,p(\cdot)}(\Omega)$.

For  further details on these spaces, see \cite{DHHR},
\cite{KR}, \cite{RaRe} and their references.

\section{A Liouville type result}\label{app-liouv}

In this Appendix we prove, for the sake of completeness, a Liouville type result for the $p_0$-Laplace operator,  because we did not find it in the literature in this form.  This result plays a key role in Section \ref{section7}.
\begin{lem}\label{lemma7.3}
Let  $1<p_0<\infty$ be constant. Let $u$ be Lipschitz in  $\mathbb{R}^n\cap\{x_n\ge0\}$ and solution to \begin{equation}  
\label{fbliminftylinboh}
\left\{
\begin{array}{ll}
\Delta_{p_0} u = 0, & \hbox{in $\{x_n>0\}$}, \\
\  &  \\
u= 0, &\mbox{on}\quad \{x_n=0\}.
\end{array}
\right.
\end{equation}
 Then, there exists $C\in\mathbb{R}$ such that $u(x)=C x_n$  in  $\{x_n\ge0\}$.
\end{lem}
\begin{proof}

 We consider, for $x=(x',x_n),$ $x'\in \mathbb{R}^{n-1},$ $x_n\in \mathbb{R},$ the extended function
\begin{equation*}
\tilde{u}(x',x_n)=\left\{\begin{array}{l}
u(x',x_n),\quad x_n\geq 0,\\
-u(x',-x_n), \quad x_n\leq 0.
\end{array} 
\right.
\end{equation*}
{}From the Lipschitz continuity of $u$ in the set $\{x_n\geq 0\}$ it  follows that $\tilde{u}$ is Lipschitz in $\mathbb{R}^n$ and  $\tilde{u}\in W^{1,\infty}_{\rm loc}(\mathbb{R}^n).$ Now let $\varphi\in C_0^\infty(\mathbb{R}^n)$. There holds
\begin{equation}\label{eq-for-utilde}\begin{split}
&\int_{\mathbb{R}^n}|\nabla \tilde{u}|^{p_0-2}\langle\nabla\tilde{u},\nabla\varphi\rangle dx=\int_{\mathbb{R}^n\cap\{x_n>0\}}|\nabla \tilde{u}|^{p_0-2}\langle\nabla\tilde{u},\nabla\varphi\rangle dx\\
&+\int_{\mathbb{R}^n\cap\{x_n<0\}}|\nabla \tilde{u}|^{p_0-2}\langle\nabla\tilde{u},\nabla\varphi\rangle dx\\
&=\int_{\mathbb{R}^n\cap\{x_n>0\}}|\nabla u|^{p_0-2}\langle\nabla u,\nabla\varphi\rangle dx-\int_{\mathbb{R}^n\cap\{x_n>0\}}|\nabla u|^{p_0-2}\langle\nabla u,\nabla\tilde{\varphi}\rangle dx\\
&=\int_{\mathbb{R}^n\cap\{x_n>0\}}|\nabla u|^{p_0-2}\langle\nabla u,\nabla\eta\rangle dx,
\end{split}
\end{equation}
where $\tilde{\varphi}(x',x_n):=\varphi(x',-x_n)$ and $\eta(x):=\varphi(x',x_n)-\varphi(x',-x_n)\in C_0^\infty(\mathbb{R}^n).$ In particular, $\eta(x',0)=0$ and thus, there exists $\{\eta_j\}_{j\in \mathbb{N}}\subset C_0^{\infty}(\mathbb{R}^n\cap\{x_n>0\})$ such that $\eta_j\to\eta$ in $W^{1,p_0}(\mathbb{R}^n\cap\{x_n> 0\})$ with $\mbox{spt}\eta_j,\mbox{spt}\eta\subset B_R$, for some $R>0$. Then,
\begin{equation*}\begin{split}
\int_{\mathbb{R}^n\cap\{x_n>0\}}|\nabla u|^{p_0-2}\langle\nabla u,\nabla\eta_j\rangle dx=0,
\end{split}
\end{equation*}
since $u$ is solution to \eqref{fbliminftylinboh}.

We claim that 
\begin{equation}\label{eq-eta}
\int_{\mathbb{R}^n\cap\{x_n>0\}}|\nabla u|^{p_0-2}\langle\nabla u,\nabla\eta\rangle dx=0
\end{equation}
and therefore, by \eqref{eq-for-utilde},
$$
\int_{\mathbb{R}^n}|\nabla \tilde{u}|^{p_0-2}\langle\nabla\tilde{u},\nabla\varphi\rangle dx=0.
$$
That is, $\tilde{u}$ is  a weak solution to $\Delta_{p_0}\tilde{u}=0$ in  $\mathbb{R}^n.$

In fact,
\begin{equation*}
\begin{split}
&\Big|\int_{\mathbb{R}^n\cap\{x_n>0\}}|\nabla u|^{p_0-2}\langle\nabla u,\nabla\eta_j-\nabla \eta\rangle dx\Big|\leq \int_{\mathbb{R}^n\cap\{x_n>0\}}|\nabla u|^{p_0-1}|\nabla\eta_j-\nabla \eta| dx\\
&\leq \left(\int_{B_R\cap\{x_n>0\}}|\nabla u|^{p_0}dx\right)^{\frac{p_0-1}{p_0}}\left(\int_{B_R\cap\{x_n>0\}}|\nabla\eta_j-\nabla \eta|^{p_0}dx\right)^{1/{p_0}}\to0,
\end{split}
\end{equation*}
thus \eqref{eq-eta} holds.

Hence, $\Delta_{p_0}\tilde{u}=0$  and $|\tilde{u}(x)|\leq L |x|$ in $\mathbb{R}^n$, with $L$  the Lipschitz constant of $\tilde{u}$, and the same result holds for $\tilde{u}_R(x)=\frac{\tilde{u}(Rx)}{R}$, for any $R>0$. Moreover, by the $C^{1,\alpha}$ estimates for the $p_0$-Laplace operator,  there exists $\alpha\in (0,1)$ such that $\tilde{u}_R\in C^{1,\alpha}(\overline{B_1})$ and for every $x,y\in B_1,$
$$
M\geq  \frac{|\nabla \tilde{u}_R(x)-\nabla \tilde{u}_R(y)|}{|x-y|^\alpha}=\frac{|\nabla \tilde{u}(Rx)-\nabla \tilde{u}(Ry)|}{|x-y|^\alpha},
$$
where $M$ and $\alpha$ depend only on $n, p_0$ and $\sup_{B_2}|\tilde{u}_R(x)|\leq 2L. $
 Thus,  it follows  that for $z$ and $\kappa$ in $B_R$,
$$
|\nabla\tilde{u}(z)-\nabla \tilde{u}(\kappa)|\leq M\frac{|z-\kappa|^{\alpha}}{R^\alpha}.
$$
In particular, fixing $z,\kappa \in B_1$ and letting $R\to \infty$, we deduce that 
$$
|\nabla\tilde{u}(z)-\nabla \tilde{u}(\kappa)|=0
$$
for every $z,\kappa\in B_1$. That is, $\nabla\tilde{u}$ is constant  and $\tilde{u}$ is linear in $B_1.$ 

On the other hand, for every $\lambda>0$,  the function $\tilde{u}_{\lambda}(x)=\frac{\tilde{u}(\lambda x)}{\lambda}$ is still a Lipschitz  solution of problem \eqref{fbliminftylinboh}. Hence, by the argument above, $\tilde{u}_\lambda$ is linear in $B_1$ and $\tilde{u}_{\lambda}(x)=\langle v_\lambda, x\rangle$ in $B_1$, for some $v_\lambda\in \mathbb{R}^n$. Thus  $\tilde{u}(\lambda x)=\langle v_\lambda, \lambda x\rangle$ in $B_1$ and therefore, $\tilde{u}(y)=\langle v_\lambda, y\rangle=\langle \nabla \tilde{u}(0),y\rangle$ in  $B_{\lambda}.$

 Since $\lambda> 0$ is arbitrary,  $\tilde{u}(y)=\langle \nabla \tilde{u}(0),y\rangle$ in $\mathbb{R}^n$. Now, denoting $C=\frac{\partial \tilde{u}(0)}{\partial y_n}$, we conclude that $u(x)=C x_n$ in  $\{x_n\ge0\}$.
\end{proof}

\section*{Acknowledgment }
The authors wish to thank Sandro Salsa for very interesting discussions about the subject of this paper.

\section*{ Data availability} This manuscript has no associated data.


\begin{thebibliography}{9999}

\bibitem[AMS]{AMS} R. Aboulaich, D. Meskine, A. Souissi, \emph{New diffusion models in image processing}, Comput.
Math. Appl. 56 (2008), 874--882.

\bibitem[AC]{AC} H. W. Alt,
L. A. Caffarelli, \emph{Existence and regularity for a minimum problem
with free boundary}, J. Reine Angew. Math 325
(1981),105--144.


\bibitem[ACF]{ACF}
H.~W. Alt, L.~A. Caffarelli, A.~Friedman, \emph{A free boundary
problem for
  quasilinear elliptic equations}, Ann. Sc. Norm. Super. Pisa Cl. Sci. (4)
  11 (1) (1984),  1--44.


\bibitem[AR]{AR} S. N. Antontsev, J. F. Rodrigues, {\it On stationary thermo-rheological viscous flows}, Ann. Univ. Ferrara, Sez. VII, Sci. Mat. 
 52 (1) (2006), 19--36.

\bibitem[AF]{AF} R. Argiolas, F. Ferrari, {\it Flat free boundaries regularity in two-phase problems for a class of fully nonlinear elliptic operators with variable coefficients,} Interfaces Free Bound. 11 (2009), no.2, 177-199.

\bibitem[C1]{C1} L. A. Caffarelli, {\it A Harnack
inequality approach to the regularity of free boundaries. Part I:
Lipschitz free boundaries are $C^{1,\alpha}$}, Rev. Mat.
Iberoamericana 3 (1987) no. 2, 139--162.

\bibitem[C2]{C2} L. A. Caffarelli, {\it A Harnack
inequality approach to the regularity of free boundaries. Part II:
Flat free boundaries are Lipschitz}, Comm. Pure Appl. Math.
42 (1989), no.1, 55--78.

\bibitem[CC]{CC} L. A. Caffarelli, X. Cabre, {\it Fully Nonlinear Elliptic
Equations}, Colloquium Publications 43, American Mathematical
Society, Providence, RI, 1995.




\bibitem[CS]{CS} L. A. Caffarelli, S. Salsa, A Geometric Approach to Free Boundary Problems, Amer. Math.
Soc., Providence RI, 2005.


\bibitem[CFS]{CFS} M. C.  Cerutti, F. Ferrari, S. Salsa, {\it Two phase
problems for linear elliptic operators with variable coefficients:
Lipschitz free boundaries are $C^{1,\gamma}$}, Archive for
Rational Mechanics and Analysis, Vol 171, n.3, pp. 329 - 348
(2004)



\bibitem[CLR]{CLR} Y. Chen, S. Levine, M. Rao, {\it Variable exponent, linear growth functionals in image restoration},
SIAM J. Appl. Math. 66 (2006), 1383--1406.


\bibitem[CIL]{CIL} M. G. Crandall, H. Ishii, P. L. Lions, {\it User’s guide to viscosity solutions of second order partial differential equations}, Bull. Amer. Math. Soc. (N.S.) 27 (1) (1992),  1--67.

\bibitem[DP]{DP}
D. Danielli, A. Petrosyan, \emph{A minimum problem with free
boundary for a
  degenerate quasilinear operator}, Calc. Var. Partial Differential Equations
  23 (1) (2005),  97--124.


\bibitem[D]{D} D. De Silva, {\it Free boundary regularity for a problem with right hand side}, Interfaces and free boundaries 13 (2011), 223--238.
\bibitem[DFS1]{DFS1} D. De Silva, F. Ferrari, S. Salsa, {\it Two-phase problems with distributed sources: regularity of the free boundary.} Anal. PDE 7 (2014), no. 2, 267--310. 

\bibitem[DFS2]{DFS2} D. De Silva, F. Ferrari, S. Salsa, {\it Free boundary regularity for fully nonlinear non-homogeneous two-phase problems.} J. Math. Pures Appl. (9) 103 (2015), no. 3, 658--694.

\bibitem[DFS3]{DFS3} D. De Silva, F. Ferrari, S. Salsa, {\it Regularity of higher order in two-phase free boundary problems.} Trans. Amer. Math. Soc. 371 (2019), no. 5, 3691--3720.

\bibitem[DHHR]{DHHR} L. Diening, P. Harjulehto, P. Hasto, M. Ruzicka, {\it Lebesgue and Sobolev Spaces with variable exponents}, Lecture Notes in Mathematics 2017, Springer,  2011.


\bibitem[Fa]{Fan} X. Fan, {\it Global $C^{1,\alpha}$ regularity for variable exponent elliptic equations in divergence form}, J. Differential Equations 235 (2007), 397--417.





\bibitem[F1]{F1} M. Feldman, {\it Regularity for nonisotropic two-phase problems with Lipschitz free boundaries}, Differential Integral Equations 10 (1997), no.6, 1171--1179.

\bibitem[F2]{F2} M. Feldman, {\it Regularity of Lipschitz free boundaries in two-phase problems for fully nonlinear elliptic equations}, Indiana Univ. Math. J. 50 (2001), no.3, 1171--1200.

\bibitem[FMW]{FMW} J. Fernandez  Bonder, S. Mart\'{\i}nez, N. Wolanski, {\it A free boundary problem for the $p(x)$-Laplacian}, Nonlinear Anal.
72 (2010), 1078--1103.


\bibitem[Fe1]{Fe1} F. Ferrari, {\it Two-phase problems for a class of fully nonlinear elliptic operators, Lipschitz free boundaries are $C^{1,\gamma}$},
Amer. J. Math. 128 (2006), 541--571.




\bibitem[FL]{FL}  F. Ferrari, C. Lederman, {\it Regularity of flat free boundaries for a $p(x)$-Laplacian problem with right hand side}, 
Nonlinear Anal. 212 (2021), Article ID 112444, 25 p. 
 
\bibitem[FS1]{FS1} F. Ferrari, S. Salsa, {\it Regularity of the free boundary in two-phase problems for elliptic operators},
 Adv. Math. 214 (2007), 288--322.

\bibitem[FS2]{FS2} F. Ferrari, S. Salsa, {\it Subsolutions of elliptic operators in divergence form and application to two-phase free boundary problems}, Bound. Value Probl. 2007, art. ID 57049, 21pp.






  \bibitem[GS]{GS} B. Gustafsson, H. Shahgholian, {\it Existence and
geometric properties of solutions of a free boundary problem in
potential theory}, J. Reine Angew. Math. 473 (1996),
137--179.

\bibitem[IS]{IS} C. Imbert, L. Silvestre, {\it $C^{1,\alpha}$ regularity of solutions of some degenerate fully
non-linear elliptic equations}, Adv. Math. 233 (2013), 196--206.


\bibitem[JK]{JK}D. S. Jerison, C. E. Kenig,  {\it Boundary behavior of harmonic functions in nontangentially accessible domains,} Adv. in Math. 46 (1982), no. 1, 80--147.



\bibitem[JJ]{JJ} V. Julin, P. Juutinen, {\it A new proof for the equivalence of weak and viscosity solutions for the $p$-Laplace equation}, Communications in PDE 37 (2012), no. 5, 934 -- 946.

\bibitem[JLM]{JLM} P. Juutinen, P. Lindqvist, J. Manfredi, {\it On the equivalence of viscosity solutions and weak solutions for a quasi-linear equation}. SIAM J. Math. Anal. 33 (2001), no. 3, 699--717.


 \bibitem[JLP]{JLP} P. Juutinen, T. Lukkari, M. Parviainen, {\it Equivalence of viscosity and weak solutions for the
$p(x)$-Laplacian}. Ann. Inst. H. Poincare Anal. Non Lineaire 27 (2010), no. 6, 1471--1487.

\bibitem[KN]{KN} D. Kinderlehrer,  L. Nirenberg,  {\it Regularity in free boundary problems}, Ann. Sc. Norm. Super. Pisa Cl. Sci. (4) 
{4 (2)} (1977),  373--391.

\bibitem[KR]{KR}
O. Kov\'a\v{c}ik, J. R\'akosn{\'i}k, \emph{On spaces ${L}^{p(x)}$ and
  ${W}^{k,p(x)}$}, Czechoslovak Math. J  41 (1991), 592--618.

\bibitem[Le]{Le}
C.~Lederman, \emph{A free boundary problem with a volume penalization}, Ann.
  Sc. Norm. Super. Pisa Cl. Sci. (4) 23 (2) (1996),  249--300.


\bibitem[LW1]{LW1} C.  Lederman, N. Wolanski, {\it An inhomogeneous singular perturbation problem for the $p(x)$-Laplacian,} Nonlinear Anal. 138 (2016), 300--325.

\bibitem[LW2]{LW2} C.  Lederman, N. Wolanski, {\it Weak solutions and regularity of the interface in an inhomogeneous free boundary problem for the $p(x)$-Laplacian}. Interfaces Free Bound. 19 (2017), no. 2, 201--241.

\bibitem[LW3]{LW3} C. Lederman, N. Wolanski, {\it Inhomogeneous minimization problems for the $p(x)$-Laplacian,} J. Math. Anal. Appl. 475 (2019), no. 1, 423--463.

\bibitem[LW4]{LW4} C. Lederman, N. Wolanski, {\it An optimization problem with volume constraint for an inhomogeneous operator with nonstandard growth},  Discrete Contin. Dyn. Syst. Series A { 41 (6)} (2021), 2907--2946. 

\bibitem[LR]{LR} R. Leit$\tilde{a}$o, G. Ricarte, 
Free boundary regularity for a degenerate problem with right hand side,
Interfaces Free Bound. 20 (2018), no. 4, 577--595.


\bibitem[LN1]{LN1} J. Lewis, K. Nystr{\"o}m, {\it Regularity of Lipschitz free boundaries in two phase problems for the $p$-Laplace operator,} Adv. in Math. 225, (2010) 2565-2597. 

\bibitem[LN2]{LN2} J. Lewis, K. Nystr{\"o}m, {\it Regularity of flat free boundaries in two-phase problems for the $p$-Laplace operator,} Ann. Inst. H. Poincar\'e Anal. Non Lin\'aire 29 (2012), no. 1, 83--108.

\bibitem[MW]{MW}
S.~Mart\'{\i}nez, N.~Wolanski, \emph{A minimum problem with free boundary in
  {O}rlicz spaces}, Adv. Math. 218 (6) (2008),
  1914--1971.


\bibitem[MO]{MO} M. Medina, P. Ochoa, {\it On the viscosity and weak solutions for non-homogeneous $p$-Laplace equations}. Adv. in Nonlinear Anal., 8 (2019), no. 1, 468--481.



\bibitem[RR]{RaRe} V. D. Radulescu, D. D. Repovs, {\it Partial differential equations with variable exponents: variational methods and qualitative analysis}, Monographs and Research Notes in Mathematics, Book 9. Chapman \& Hall / CRC Press, Boca Raton, FL, 2015.

\bibitem[R]{R} M. Ruzicka, {\it Electrorheological Fluids: Modeling and
Mathematical Theory}, Springer-Verlag, Berlin, 2000.


\bibitem[S]{S}  O. Savin, {\it Small perturbation solutions for elliptic equations}. Comm. Partial Differential Equations 32 (2007), no. 4-6, 557--578.

\bibitem[SS]{SS} L. Silvestre, B. Sirakov,  {\it Boundary regularity for viscosity
solutions of fully nonlinear elliptic equations}, Comm. Partial Differential Equations
39 (9)  (2014), 1694--1717.

\bibitem[Si]{Si}  B. Sirakov, {\it Solvability of uniformly elliptic fully
nonlinear PDE}. Arch. Rational Mech. Anal. 195 (2010), 579--607.

\bibitem[T]{T} N. S. Trudinger, {\it On Harnack type inequalities and
their application to quasilinear elliptic equations}, Comm. Pure
Appl. Math. {20} (1967), 721--747.


\bibitem[W1]{W1}  P. Y. Wang, {\it Regularity of free boundaries of two-phase problems for fully nonlinear elliptic equations of second order. I. Lipschitz free boundaries are $C^{1,\alpha}$}, Comm. Pure Appl. Math. 53 (2000), 799--810.

\bibitem[W2]{W2}  P. Y. Wang, {\it Regularity of free boundaries of two-phase problems for fully nonlinear elliptic equations of second order. II. Flat free boundaries are Lipschitz}, Comm. Partial Differential Equations 27 (2002), 1497--1514.

\bibitem[Wo]{Wo} N. Wolanski, {\it Local bounds, Harnack inequality and H\"older continuity for divergence type elliptic equations with non-standard growth},
Rev. Un. Mat. Argentina 56 (1) (2015), 73--105.

\bibitem[Z1]{Z1} V. V. Zhikov, {\it Averaging of functionals of the calculus of variations and elasticity theory}, Math. USSR. Izv. 29 (1)  (1987), 33--66.


\bibitem[Z2]{Z2} V. V. Zhikov, {\it Solvability of the three-dimensional thermistor problem}, Tr. Mat. Inst. Steklova D (Differ. Uravn. i Din. Sist.) 261 (2008) 101–114.


\end{thebibliography}
\end{document}